\documentclass[12pt,a4paper,reqno]{amsart}

\usepackage{amsmath,amssymb,amscd,mathrsfs,mathtools,amsthm}
\usepackage[dvipsnames,svgnames,table]{xcolor}
\usepackage{tikz} 
\usepackage{blkarray}
\usepackage{url}
\usepackage{graphicx}
\usepackage{setspace}  
\usepackage{ytableau}
\usepackage{slashed}
\ytableausetup{centertableaux,boxsize = 1.3em}

\usepackage{nicematrix,tikz}
\usepackage{subcaption}

\definecolor{blond}{rgb}{0.92, 0.89, 0.80}
\definecolor{champagne}{rgb}{0.97, 0.91, 0.81}

\definecolor{RedDef}{RGB}{160, 4, 23}
\definecolor{GreenDef}{RGB}{27, 132, 5}
\definecolor{BlueDef}{RGB}{11, 78, 188}
\definecolor{ThmDef}{RGB}{50, 6, 211}

\usepackage{titletoc}
\usepackage{titlesec}

\titleformat{\section}[hang]
{\scshape\small\color{RedDef}\filcenter}{\S\ \thesection.}{0.2em}{}

\titleformat{\subsection}[runin]
{\scshape\color{RedDef}}{\thesubsection.}{0.2em}{}[.]

\contentsmargin{2.55em}
\dottedcontents{section}[3.8em]{}{1.5em}{1pc}
\dottedcontents{subsection}[6.1em]{}{2.5em}{1pc}
\dottedcontents{subsubsection}[8.1em]{}{3.2em}{1pc}

\usepackage[breaklinks,colorlinks,pagebackref]{hyperref} 

\definecolor{Red}{RGB}{160, 4, 23}
\definecolor{Green}{RGB}{27, 132, 5}
\definecolor{Blue}{RGB}{11, 78, 188}

\hypersetup{linkcolor=Red,citecolor=Green}

\usepackage[capitalise,nameinlink]{cleveref}
\crefformat{equation}{(#2#1#3)}

   \AddToHook{env/lemma/begin}{\crefalias{theorem}{lemma}}
   \AddToHook{env/corollary/begin}{\crefalias{theorem}{corollary}}
   \AddToHook{env/proposition/begin}{\crefalias{theorem}{proposition}}
   \AddToHook{env/definition/begin}{\crefalias{theorem}{definition}}
   \AddToHook{env/remark/begin}{\crefalias{theorem}{remark}}
   \AddToHook{env/example/begin}{\crefalias{theorem}{example}}
   \AddToHook{env/problem/begin}{\crefalias{theorem}{problem}}
   \AddToHook{env/notation/begin}{\crefalias{theorem}{notation}}

\usepackage[msc-links,lite,alphabetic]{amsrefs}
\def\MR#1{%
    \relax\ifhmode\unskip\spacefactor3000 \space\fi
\href{http://www.ams.org/mathscinet-getitem?mr=#1}{ MR#1}
}

\makeatletter
\renewcommand{\BibLabel}{%
    \Hy@raisedlink{\hyper@anchorstart{cite.\CurrentBib}\hyper@anchorend}%
    [\thebib]%
}
\makeatother

\usepackage{mathscinet}

\theoremstyle{plain}

\newtheorem{theorem}{Theorem}[section]
\newtheorem{lemma}[theorem]{Lemma}
\newtheorem{proposition}[theorem]{Proposition}
\newtheorem{corollary}[theorem]{Corollary}

\theoremstyle{definition}

\newtheorem{definition}[theorem]{Definition}
\newtheorem{remark}[theorem]{Remark}
\newtheorem{example}[theorem]{Example}

\newtheorem{notation}[theorem]{Notation}

\newcommand{\defin}{\overset{\mathrm{def}}{=}}

\newcommand{\ssyt}{\mathrm{SSYT}}
\newcommand{\less}{\unlhd}

\usepackage{dsfont}

\newcommand{\ZZ}{{\mathbb{Z}}}

\newcommand{\RR}{{\mathbb{R}}}

\newcommand{\QQ}{{\mathbb{Q}}}
\newcommand{\FF}{{\mathbb{F}}}
\newcommand{\1}{\mathds{1}}

\newcommand{\C}{\mathcal{C}}

\newcommand{\G}{{\mathrm{G}}}
\newcommand{\GL}{\mathrm{GL}}

\renewcommand{\P}{\mathrm{P}}
\newcommand{\F}{\mathrm{F}}
\newcommand{\E}{\mathrm{E}}
\newcommand{\Hung}{\H}
\newcommand{\HH}{\mathrm{H}}
\newcommand{\Ind}{\mathrm{Ind}}
\newcommand{\conj}{\mathrm{Cl}_\G}
\newcommand{\cen}{\mathrm{Z}_\G}
\newcommand{\K}{{\mathrm{K}}}
\newcommand{\Q}{\mathrm{Q}}

\newcommand{\Z}{\mathrm{Z}}
\newcommand{\A}{\mathrm{A}}
\newcommand{\B}{\mathrm{B}}
\newcommand{\R}{\mathrm{R}}

\newcommand{\W}{\mathrm{W}}
\newcommand{\syt}{\mathrm{SYT}}

\renewcommand{\u}{\mathfrak{u}}
\newcommand{\gl}{\mathfrak{gl}}

\newcommand{\x}{{\bf x}}
\newcommand{\y}{{\bf y}}

\newcommand{\U}{\mathsf{U}}

\newcommand{\fn}{\mathsf{n}}

\newcommand{\hes}{\mathsf{h}}
\newcommand{\graph}{\mathsf{G}}
\newcommand{\EE}{\mathsf{E}}

\DeclareMathOperator{\rk}{rank}

\DeclareMathOperator{\spn}{span}

\usepackage[top=2cm,bottom=2cm,right=1.9cm,left=1.9cm]{geometry}

\title{Counting $\FF_q$-points of orbital varieties in ad-nilpotent ideals of type $A_n$}

 \author[M. Bardestani]{Mohammad Bardestani}
 \address{ John Abbott College, 275 Rue Lakeshore, Sainte-Anne-de-Bellevue, QC H9X 1S2, Canada}
 \email{mohammad.bardestani@gmail.com}
 \author[K. Mallahi-Karai]{Keivan Mallahi-Karai}
 \address{Constructor University, Campus Ring I, 28759 Bremen, Germany.}
 \email{kmallahika@constructor.university }
 \author[S. Ram]{Samrith Ram}
 \address{Indraprastha Institute of Information Technology Delhi, New Delhi, India.}
 \email{samrithram@gmail.com} 
 \author[H. Salmasian]{Hadi Salmasian}
 \address{Department of Mathematics, University of Ottawa, 585 King Edward, Ottawa, ON K1N
 6N5, Canada.}
 \email{hadi.salmasian@uottawa.ca}

\begin{document}

\begingroup
\def\uppercasenonmath#1{} 
\let\MakeUppercase\relax 
\maketitle
\endgroup

{
\footnotesize
\hypersetup{linkcolor=RedDef}
\tableofcontents
}

\begin{abstract}

Let $\mathfrak b_n(\FF_q)$  denote the Lie algebra of upper triangular $n\times n$ matrices with entries in the finite field $\FF_q$. For every $\mathfrak b_n(\FF_q)$-stable ideal $\mathfrak a$ of the Lie algebra $\u_n(\FF_q)$ of strictly upper triangular matrices, and every partition $\mu$ of $n$, we prove two explicit formulas for the number of elements of $\mathfrak a$ of Jordan type $\mu$: the first one is given 
by the Hall scalar product of a modified Hall-Littlewood function indexed by $\mu$ and a  chromatic quasisymmetric function associated to $\mathfrak a$. The second one is  given in terms of  a sum of products of $q$-integers over certain standard tableaux that are compatible with a partial order associated to $\mathfrak a$.
  
In the special case that $\mathfrak a=\u_\Lambda(\FF_q)$, where 
$\u_\Lambda(\FF_q)$ is the nilradical of the standard parabolic subalgebra of $\gl_n(\FF_q)$ associated to a composition $\Lambda$ of $n$, our formulas reduce to the statement (recently proved in~\cite[Theorem 5.13]{Karp-Thomas}) that 
up to an explicit polynomial factor in $q$, the number of elements in $\u_\Lambda(\FF_q)$ of Jordan type $\mu$ is equal to the coefficient of the monomial $\x^\Lambda$ in the specialization of the dual Macdonald symmetric function $\Q_{\mu'}(\x;q^{-1},t)$ at $t=0$.
We obtain a new and shorter  proof of the latter assertion
using a parabolic variation of Borodin's division algorithm~\cite{Borodin,Kirillov}.

We give three applications of our main result stated above. 
First, we obtain an analogous formula for the number of points of a nilpotent Hessenberg variety. 
Second, we give an explicit formula for the number of  $X\in \u_\Lambda(\FF_q)$ that satisfy $X^2=0$. In the special case $\Lambda=(1^n)$, our formula is different from the one conjectured by Kirillov and Melnikov~\cite{KirillovMelnikov} (and proved by Ekhad and Zeilberger~\cite{Ekhad}). However, we obtain a new proof of the Kirillov-Melnikov-Ekhad-Zeilberger 
formula
using our latter result and a formula for two-rowed Macdonald polynomials by Jing and J\'{o}zefiak~\cite{Jing}.
As a byproduct, we demonstrate that a recurrence relation of Kirillov in~\cite{KirillovX=0} for the number of strictly upper triangular matrices of Jordan type $\mu$ is a straightforward consequence of our main results combined with the Cauchy identity for Macdonald polynomials.
Third, we give an explicit formula for the number of double cosets $\U_1\backslash\GL_n(\FF_q)/\U_2$ where $\U_1$ and $\U_2$ denote unipotent subgroups of $\GL_n(\FF_q)$ that correspond to two $\mathfrak b_n(\FF_q)$-stable ideals of $\u_n(\FF_q)$. 
\end{abstract}

\section{Introduction}

Let $\U_n(\mathbb{F}_q)$ be the group of $n \times n$ upper-triangular matrices over the finite field $\mathbb{F}_q$ with all the diagonal entries equal to 1, and let $\mathfrak{u}_n(\mathbb{F}_q)$ denote the Lie algebra of $n \times n$ upper-triangular nilpotent matrices over $\mathbb{F}_q$. In an influential and elegant paper 
\cite{Kirillov}, Kirillov initiated a comprehensive study of the orbit method for $\U_n(\mathbb{F}_q)$, establishing various connections between the adjoint and coadjoint orbit structures of 
$\U_n(\mathbb{F}_q)$ and
$\mathfrak{u}_n(\mathbb{F}_q)$ 
and representation theory. 
In connection to these ideas, Kirillov posed the following question:
\medskip
\begin{quote}
Given a partition $\mu$ of $n$, count the number of  $X\in \mathfrak{u}_n(\mathbb{F}_q)$ of Jordan type  $\mu$.
\end{quote}
\medskip

This problem was investigated by Kirillov himself and revisited later on by Borodin~\cite{Borodin}, who used  a technique called  ``the division algorithm'' to provide a recursive formula for the number of these matrices. Borodin applied this formula to study the asymptotic behavior of the Jordan type of strictly upper triangular $n \times n$ matrices as $n \to \infty$. 
\smallskip

One can think of  $\u_n(\FF_q)$ as the nilradical of the standard Borel subalgebra of $\mathfrak{gl}_n(\FF_q)$. Then Kirillov's question is about the number of points of the (not necessarily irreducible) variety obtained as  the intersection of $\u_n(\FF_q)$ and a nilpotent orbit of $\mathfrak{gl}_n(\FF_q)$. At least over $\mathbb C$, such varieties are sometimes called \emph{orbital varieties}; see for example~\cite{JOSEPH}.
 From this viewpoint it is natural to ask Kirillov's question for nilradicals of other parabolic subalgebras,
or even more generally, for any ideal of $\u_n(\FF_q)$ that is stable under the action of the standard Borel subalgebra 
$\mathfrak b_n(\FF_q)$ 
of $\mathfrak{gl}_n(\FF_q)$. The latter ideals are sometimes simply called \emph{ad-nilpotent ideals} of $\u_n(\FF_q)$. Their interaction with nilpotent orbits  has attracted attention (for example see~\cite{ConciniLusztigProcesi,CelliniPapi,MR4334164}). Our primary goal in this paper is to answer this more general form of Kirillov's  question (see~\cref{mainthm:Fmula-comb,Intro_them: Main_Macdonald,Hess_theorem}). Subsequently, we give various applications of the latter theorems.
\smallskip

Recall that a composition $\Lambda$ of $n$ is a finite tuple of positive integers whose sum equals $n$. 
When the parts of the composition form a weakly decreasing sequence, we call it  a partition.
The standard parabolic subgroup of $\G = \GL_n(\mathbb{F}_q)$ associated to $\Lambda$ is denoted by $\P_\Lambda = \P_\Lambda(\mathbb{F}_q)$, and the  unipotent radical 
of $\P_\Lambda$
is denoted by $\mathrm{U}_\Lambda = \mathrm{U}_\Lambda(\mathbb{F}_q)$. The 
nilradical of the parabolic subalgebra of $\mathfrak{gl}_n(\FF_q)$ associated to $\Lambda$ will be denoted by  $\u_\Lambda=\u_\Lambda(\FF_q)$. 
For instance when $\Lambda=(2,3,1)$ we have:
$$
\mathrm{P}_\Lambda=
\scalebox{0.8}{$
  \begin{pNiceMatrix}[left-margin=0.6em, right-margin=0.6em]
    \Block[fill=gray]{2-2}{} 
    \ast & \ast & \ast & \ast & \ast & \ast  \\
    \ast & \ast & \ast & \ast & \ast & \ast  \\
    0 & 0 & \Block[fill=gray]{3-3}{} \ast & \ast & \ast & \ast  \\
    0 & 0 & \ast &  \ast & \ast & \ast  \\
    0 & 0 & \ast & \ast & \ast & \ast  \\
    0 & 0 & 0 & 0 & 0 & \Block[fill=gray]{1-1}{} \ast 
\end{pNiceMatrix}
$}, \,\,
\mathrm{U}_\Lambda=
\scalebox{0.8}{$
\begin{pNiceMatrix}[left-margin=0.6em, right-margin=0.6em]
    \Block[fill=gray]{2-2}{} 
    1 & 0 & \ast & \ast & \ast & \ast  \\
    0 & 1 & \ast & \ast & \ast & \ast  \\
    0 & 0 & \Block[fill=gray]{3-3}{} 1 & 0 & 0 & \ast  \\
    0 & 0 & 0 &  1 & 0 & \ast  \\
    0 & 0 & 0 & 0 & 1 & \ast  \\
    0 & 0 & 0 & 0 & 0 & \Block[fill=gray]{1-1}{} 1
\end{pNiceMatrix}
$}, \,\,
\u_\Lambda=
\scalebox{0.8}{$
\begin{pNiceMatrix}[left-margin=0.6em, right-margin=0.6em]
    \Block[fill=gray]{2-2}{} 
    0 & 0 & \ast & \ast & \ast & \ast  \\
    0 & 0 & \ast & \ast & \ast & \ast  \\
    0 & 0 & \Block[fill=gray]{3-3}{} 0 & 0 & 0 & \ast  \\
    0 & 0 & 0 &  0 & 0 & \ast  \\
    0 & 0 & 0 & 0 & 0 & \ast  \\
    0 & 0 & 0 & 0 & 0 & \Block[fill=gray]{1-1}{} 0
\end{pNiceMatrix} 
$},
$$
where the $*$'s represent arbitrary elements of $\FF_q$.
For $\Lambda=(1^n)$, we obtain $\u_\Lambda=\u_n(\FF_q)$. 
\smallskip

There is a simple yet quite useful characterization of the ad-nilpotent  ideals of $
\u_n(\FF_q)$, connecting them to Dyck paths.  
To give this characterization, we need the notion of a Hessenberg function. 
\begin{definition}[Hessenberg function]
A Hessenberg function is a non-decreasing function $\hes \colon [n] \to [n]$ such that $\hes(i) \geq i$ for every $i \in [n]$, where $[n] := \{1, \dots, n\}$. The set of all Hessenberg functions on $[n]$ will be denoted by $\mathcal{H}_n$. We sometimes identify a Hessenberg function by the $n$-tuple $(\hes(1),\ldots,\hes(n))$. 
\end{definition}
 One can naturally associate a graph to $\hes$ with vertex set $[n]$ and set of edges $E = \{\{i, j\}: i < j \leq \hes(i)\}$. This graph, known as the \emph{indifference graph} of $\hes$, will be denoted by $\graph(\hes)$. A Hessenberg function $\hes\in\mathcal{H}_n$ also produces a partial order on $[n]$ by declaring $i\prec_\hes j$ if and only if $\hes(i)<j$. We denote the poset $([n],\preceq_\hes)$ by $\mathcal{P}_\hes$. Associated to $\hes$, we also define the following Lie subalgebra of $\u_n(\FF_q)$: 
\begin{equation}\label{hess_subalg}
    \u_\hes = \u_\hes(\FF_q) \defin \{(a_{ij}) \in \u_n(\FF_q) : a_{ij} = 0 \mbox{ for }  j \leq \hes(i)\}.
  \end{equation}

  \begin{example}\label{eg:hess}
For $\hes\in\mathcal H_n$ where 
$\hes=(2,3,5,5,5)$ we have 
$$
\u_\hes=\left\{\small
\begin{pmatrix}
    0 & 0 & * & * & *\\
    0 & 0 & 0 &  * & *\\
    0 & 0 & 0 & 0 & 0\\
    0 & 0 & 0 & 0 & 0\\
    0 & 0 & 0 & 0 & 0
\end{pmatrix}: *\in \FF_q
\right\}, \qquad  
\raisebox{-2em}{
\begin{tikzpicture}[scale=0.6]

  \foreach \i/\x in {1/0, 2/2, 3/4, 4/6, 5/8} {
    \node[fill=black, circle, inner sep=1.5pt] (n\i) at (\x,0) {};
    \node[below=2pt] at (n\i) {\i};
  }

  \draw (n1) -- (n2);
  \draw (n2) -- (n3);
  \draw (n3) -- (n4);
  \draw (n4) -- (n5);
  \draw (n3) to[bend left=30] (n5);

  \node[draw=none, fill=none] at (4,-1.6) {$\graph(\hes)$};

\end{tikzpicture}
}.
$$
  \end{example}
The following assertion is well-known and straightforward to prove.

\begin{proposition}  
  The assignment $\hes\mapsto\u_\hes$ is a bijective correspondence between $\mathcal{H}_n$ and ad-nilpotent ideals of $\u_n(\FF_q)$. 
\end{proposition}  
Throughout the paper, we denote $\sum_{i=1}^n\hes(i)$ by $|\hes|$. 
The (non-commutative) product of $\hes_1 \in \mathcal{H}_n$ and $\hes_2 \in \mathcal{H}_m$ is defined by
$$
\hes=\hes_1\sqcup \hes_2\in\mathcal{H}_{n+m}\qquad 
\hes(i)=\begin{cases}
\hes_1(i) & 1\leq i\leq n;\\
n+\hes_2(i-n) & n+1\leq i\leq n+m.
\end{cases}
$$
The Hessenberg function $k_n = (n, \dots, n) \in \mathcal{H}_n$ is called the complete Hessenberg function of size $n$. To each composition $\Lambda=(\Lambda_1,\dots,\Lambda_s)$ of $n$ we assign the Hessenberg function 
\begin{equation}\label{k_partiton}
k_\Lambda\defin k_{\Lambda_1} \sqcup k_{\Lambda_2} \sqcup \cdots \sqcup k_{\Lambda_s},
\end{equation}
and we note that $\u_{k_\Lambda}=\u_\Lambda$.

For any $k\geq 1$ let $J_k$ be the  $k\times k$ nilpotent (upper triangular) Jordan block.
Given a partition  $\mu=(\mu_1,\dots,\mu_k)$ of $n$, we use $J_\mu$ to denote the square block-diagonal matrix of size $|\mu|=\mu_1+\cdots+\mu_k$  with diagonal blocks $J_{\mu_1},\ldots,J_{\mu_k}$.  We say that the Jordan type of a matrix is $\mu$ when its Jordan form is $J_\mu$. 


 \begin{definition}[Jordan type counting function]
Let $\mu$ be a partition of $n$.  
For a given composition $\Lambda$ of $n$ we set
\begin{equation}\label{Flambda}
\F_{\mu\Lambda}(q)\defin
| \C_\mu\cap \mathfrak u_\Lambda(\FF_q)|,
\end{equation}
where $\C_\mu=\C_\mu(\FF_q)$ is the $\GL_n(\FF_q)$-conjugacy class of $J_\mu$.
Similarly, for a Hessenberg function $h\in\mathcal{H}_n$ we set 
\begin{equation}\label{Hess_jordan}
    \F_{\mu\hes}(q)\defin |\mathcal{C}_\mu\cap \u_\hes(\FF_q)|.
\end{equation}

 \end{definition} 

\subsection{Main results}
\label{subsec:mainthm}
In order to state our first main result (see ~\cref{Hess_theorem}), we start by reviewing the definition of  the chromatic quasisymmetric functions introduced in an influential paper by Shareshian and Wachs~\cite{Shareshian-Wachs}.
Let $\graph = (V, E)$ be a graph with $V$ a finite subset of $\mathbb{Z}_{\geq 1}$. Given a subset $S$ of $\mathbb{Z}_{\geq 1}$, a proper $S$-coloring of $\graph$ is a function $\kappa \colon V \to S$ such that $\kappa(i) \neq \kappa(j)$ whenever $\{i,j\} \in E$. Let $\mathcal{C}(\graph)$ be the set of proper $\mathbb{Z}_{\geq 1}$-colorings of $\graph$. Stanley~\cite{Stanley_symmetric} defined the chromatic symmetric function of $\graph$ as
\[
X_\graph(\x) := \sum_{\kappa \in \mathcal{C}(\graph)} \prod_{v \in V} x_{\kappa(v)},
\]
where $\x=(x_1,x_2,\ldots)$. The \emph{chromatic quasisymmetric function} of $\graph = (V, E)$ is a $q$-deformation of Stanley's chromatic symmetric function defined in \cite{Shareshian-Wachs} by
\begin{align*}
  X_\graph(\x;q):=\sum_{\kappa \in \mathcal{C}(\graph)} q^{\mathrm{asc}({\kappa})}\prod_{v \in V} x_{\kappa(v)},
\end{align*}
where $\mathrm{asc}(\kappa):=\#\{\{i,j\}\in E(\graph):i<j \text{ and }\kappa(i)<\kappa(j)\}$,
is the number of ascents of the coloring $\kappa$. The chromatic quasisymmetric function is not symmetric in general. However, for the graphs $\graph(\hes)$, where $\hes\in\mathcal{H}_n$, Shareshian and Wachs~\cite[Theorem 4.5]{Shareshian-Wachs} show that $X_{\graph(\hes)}(\x,q)$ is a  symmetric function with coefficients in $\mathbb{Z}[q]$. Therefore, $X_{\graph(\hes)}(\x,q)$ can be expressed in terms of the elementary symmetric functions:
$$
X_{\graph(\hes)}(\x;q)=\sum_{\lambda\,\vdash n}c_{\lambda\hes}(q)e_\lambda(\x). 
$$
An explicit formula for the coefficients $c_{\lambda\hes}(q)$ was recently given by Hikita \cite[Theorem 1.7]{hikita2024}.

Recall that the length of a partition 
$\mu$, denoted 
by $\ell(\mu)$, is the number of its parts.
 It is convenient to define $\mu_i = 0$ for all $i >\ell(\mu)$.  The conjugate of $\mu$, denoted by $\mu'=(\mu_1',\dots, \mu_l')$,  is defined by 
$$\mu'_j=\#\{1\leq i\leq k: \mu_i\geq j\}.$$
 Given partitions $\mu$ and $\nu$ of $n$, we say that $\nu$ ``dominates'' $\mu$, and we write $\mu \less \nu$, if and only if
$$
\mu_1+\mu_2+\dots+\mu_i\leq\nu_1+\nu_2+\dots+\nu_i, \quad\text{for } i\geq 1.$$
  Assigned to each partition $\mu$ we define 
\begin{equation}\label{n_mu:dif}
\fn(\mu)\defin \sum_{i\geq 1}(i-1)\mu_i=\sum_{i\geq 1}\binom{\mu_i'}{2}. 
\end{equation}

Let  $\tilde{\HH}_\mu(\x;t)$ denote the \emph{modified Hall-Littlewood function} indexed by $\mu$,  obtained from the modified Macdonald polynomials of~\cite{HaglundHaimanLoehr} by specialization at $q=0$. More concretely, if  $J_\mu(\x;q,t)$ is the integral form of the Macdonald polynomial~\cite[p.~352, Eq.~(8.3)]{Macdonald} then
\[
\tilde{\HH}_\mu(\x;t)=t^{\fn(\mu)}
J_\mu\left[
\frac{\x}{1-t^{-1}};0,t^{-1}
\right].
\]
For a non-negative integer $n$ we adopt the following conventions:
\[
[0]_q\defin 0,\quad [n]_q \defin \frac{q^n-1}{q-1},\ \  [0]_q!\defin 1, \ \  [n]_q!\defin [1]_q\cdots [n]_q, \ \  {n \brack k}_q\defin \frac{[n]_q!}{[k]_q![n-k]_q!}. 
\]
Let $\cen(I+J_\mu)$ denote the centralizer of $I + J_\mu$ in $\G$. By~\cite[p. 181, Eq. (1.6)]{Macdonald} we have 
\begin{equation}\label{center_size1}
\left|\cen(I+J_\mu)\right|=q^{n+2\fn(\mu)}\prod_{i\geq 1}\varphi_{m_i(\mu)}\left(1/q\right),
\end{equation}
where $\varphi_m(t)=(1-t)(1-t^2)\cdots(1-t^m)$ and $m_i(\mu)=\mu_i'-\mu'_{i+1}$. Indeed
\begin{equation}\label{varphi_m}
\varphi_m\left(1/q\right)=[m]_q!(q-1)^mq^{-\binom{m}{2}-m}.
\end{equation}

We are now able to state the first main result of this paper. Henceforth we denote the Hall scalar product~\cite[Sec. I.4]{Macdonald} of symmetric functions by $\langle\cdot,\cdot\rangle$. 
 \begin{theorem}\label{Hess_theorem} Let $\hes\in \mathcal{H}_n$ be a Hessenberg function and let $\mu$ be a partition of $n$. Then $\F_{\mu\hes}(q)\in \ZZ[q]$ is a polynomial in $q$ and 
\begin{equation}\label{F_muh_Hall_inner}
      \F_{\mu \hes}(q)=|\mathfrak{u}_\hes|\sum_{\lambda\,\vdash n}\frac{c_{\lambda \hes}(q)}{[\lambda]_q!}\frac{\mathrm{F}_{\mu\lambda}(q)}{|\mathfrak{u}_\lambda|}=\frac{q^{n^2-|\hes|} (q-1)^n}{|\cen(I+J_\mu)|}\langle \tilde{\HH}_\mu(\x;q),X_{\graph(\hes)}(\x;q)\rangle.
\end{equation}    
\end{theorem}

The strategy of the proof of~\cref{Hess_theorem} is to reduce the computation of  $\F_{\mu\hes}(q)$ to the special case $\F_{\mu\lambda}(q)$ using  the \emph{modular law} of Abreu and Nigro~\cite{Abreu-Nigro} (see~\cref{subsec:modular}), and then use the fact that $\F_{\mu\lambda}(q)$ can be expressed in terms of monomial coefficients of the $q$-Whittaker functions $\P_{\mu'}(\x;1/q,0)$ (see \cref{Intro_them: Main_Macdonald}).
We remark that \cref{Intro_them: Main_Macdonald} has also been proved recently by Karp and Thomas~\cite[Theorem 5.13]{Karp-Thomas}.  
In~\cite{Karp-Thomas}, 
the authors extract~\cref{Intro_them: Main_Macdonald} from a  combination of results, including a probabilistic bijection between non-negative integer matrices and pairs of semistandard Young tableaux of the same shape, which they call the $q$-Burge correspondence. 
Our proof of \cref{Intro_them: Main_Macdonald}
is more direct: first we obtain a recursive formula for $\F_{\mu\lambda}(q)$ with explicit coefficients (see~\cref{recursive_theortem}). 
Then we solve this recursion by a simple iterative substitution. This leads in a natural way to a solution for the recursion in terms of the well-known tableau product formulas that occur in the coefficients of Macdonald polynomials!

As usual, to  any $\mu=(\mu_1,\dots,\mu_k)$ we associate a Young diagram of shape $\mu$: a left-justified shape of $k$ rows of boxes, where the $i$-th row from the top has  $\mu_i$ boxes for $1\leq i\leq k$. 
 A \emph{standard} Young tableau is a filling of the Young diagram
of $\mu$
by $\{1,\ldots,|\mu|\}$
with strictly increasing rows and columns.  A \emph{semi-standard} Young tableau of shape $\mu$ is 
a filling of this Young diagram by positive integers  
with 
weakly increasing rows and strictly increasing columns. 
A semi-standard Young tableau with shape $\mu$ has \emph{content} $\Lambda$ if there are $
\Lambda_i$ occurrences of $i$ for every $i
\geq 1$.  The number of semi-standard Young tableaux with shape $\mu$ and content $\Lambda$ is called a  \emph{Kostka number} and it is denoted by $\K_{\mu\Lambda}$.

Our second main result is a tableau formula for $\F_{\mu\hes}(q)$ using a statistic introduced by \cite{CarlssonMellit,BasuBhattacharya}. We use the same notation as~\cite{BasuBhattacharya}. Let $\syt^\hes_\mu$ be the set of standard Young tableaux of shape $\mu$ such that if $i$ is above $j$ in the same column, then $i \prec_\hes j$.  For $T \in \syt^\hes_\mu$ and a box $b \in T$, let $\mathrm{coleg}(b)$ (resp.\ $\mathrm{arm}(b)$) be the number of boxes strictly above $b$ in its column (resp.\ strictly to its right in its row). The entry in $b$ is denoted by $T(b)$.
For $T \in \syt^\hes_\mu$ and a box $b \in T$, let $\gamma(T,b)$ be the number of boxes $c$ in any row above $b$ such that $T(c) < T(b)$ but $T(c) \nprec_\hes T(b)$. Set $\gamma(T,b)=0$ if $\mathrm{coleg}(b)=0$. Equivalently, $\gamma(T,b)$ counts entries in any row above $b$ that are smaller in the usual order on $\ZZ$ but not in the partial order generated by $\hes$. Define
 $\gamma(T):=\sum_{\mathrm{coleg}(b)\ge 1} \gamma(T,b)$.
If $\mathrm{coleg}(b)\ge 1$, let $\mathrm{up}(b)$ denote the box directly above $b$. We now consider the following polynomial in $q$: 
$$
\underline\F_{\mu\hes}(q)\defin \sum_{T \in \syt^\hes_\mu}  
		q^{\gamma(T)} \prod_{\substack{b \in T \\ \mathrm{coleg}(b)\geq 1}} [\mathrm{arm}_\hes(\mathrm{up}(b),T(b))+1]_q
$$
where $\mathrm{arm}_\hes(\mathrm{up}(b),T(b))$ is the number of boxes $c$ on the right side  of $\mathrm{up}(b)$ in the same row of $\mathrm{up}(b)$ such that $T(c)\prec_\hes T(b)$. Note that the coefficients of $\underline\F_{\mu\hes}(q)$ are non-negative integers. 

\begin{theorem}\label{Hess_theorem_tableaux} Let $\hes\in \mathcal{H}_n$ be a Hessenberg function and let $\mu$ be a partition of $n$. Then \[
\F_{\mu\hes}(q)
=
q^{\binom{n}{2}-\fn(\mu)-(n-\ell(\mu))}(q-1)^{n-\ell(\mu)}\underline\F_{\mu'\hes}(1/q).
\]
\end{theorem}

It is interesting to determine for which $\mu$ and $\hes$ we have  $\F_{\mu\hes}(q)\neq 0$. In~\cref{Hessenberg_non-zero}, we address this problem and indeed prove a stronger result over an arbitrary field.  Let $\mathcal{P}$ be any poset of cardinality $n$ and let $c_k(\mathcal{P})$ denote the maximum possible cardinality of a union $\bigcup_{i=1}^k C_i$ where each $C_i\, (1\leq i\leq k)$ is a chain in $\mathcal{P}$. For each $i\geq 1$, define integers $\lambda_i$ by $c_k(\mathcal{P})=\lambda_1+\cdots +\lambda_k$ for $k\geq 1$. Then it can be shown that the sequence $\lambda_i\, (i\geq 1)$ is weakly decreasing and forms a partition of $n$ (see Britz and Fomin~\cite[Theorem 2.1]{MR1814900}) called the Greene-Kleitman shape of $\mathcal{P}$.  
Denote by $\lambda_\mathcal{P}$ the Greene-Kleitman shape of $\mathcal{P}$ and write $\lambda_\hes$ for $\lambda_{\mathcal{P}_\hes}$ when $\hes\in\mathcal{H}_n$ is a Hessenberg function. 
Our third main result is the following.
\begin{theorem}\label{Hessenberg_non-zero} Let $\FF$ be an arbitrary field, and let $\hes \in \mathcal{H}_n$ be a Hessenberg function with Greene–Kleitman shape $\lambda_\hes$. 
Let $\mathcal C_\mu=\mathcal C_\mu(\FF)$ denote the $\GL_n(\FF)$-conjugacy class of $J_\mu$. 
Then, for any partition $\mu$ of $n$, we have $\mathcal{C}_\mu \cap \u_\hes(\FF) \neq \emptyset$ if and only if $\mu \less \lambda_\hes$.
\end{theorem}
\begin{remark}\label{Hesenberg_partition} Let $\lambda$ be a partition of $n$. One can easily verify that the Greene-Kleitman shape of the poset generated by the Hessenberg function $k_\lambda$, defined in~\cref{k_partiton}, is $\lambda'$.   
\end{remark}  
 
We remark that for $\FF=\mathbb C$,~\cref{Hessenberg_non-zero} can be deduced from~\cite{MR4334164}. Our proof is completely different and works over a general field. 
The following assertion is an immediate consequence of~\cref{Hess_theorem,Hessenberg_non-zero}.
\begin{corollary}
\label{cor:n-ell}
 Let $\hes\in\mathcal{H}_n$ be a Hessenberg function and let $\mu$ be a partition of $n$ with $\mu\less \lambda_\hes$. Then the exact exponent of $q-1$ in $\F_{\mu\hes}(q)$ is $n-\ell(\mu)$.
\end{corollary}

\begin{remark} \Cref{Table1} contains the explicit computation of $\F_{\mu\hes}(q)$ when $\hes = (1,3,5,6,7,7,7)$.
\begin{table}[h!]\scriptsize
\renewcommand{\arraystretch}{1.5} 
\centering
\renewcommand{\arraystretch}{1.5} 
\centering
\begin{tabular}{|c|c|}
\hline
\text{$\mu$} & $\F_{\mu\hes}(q)$ \\
\hline
(7) & 0 \\
\hline
(6, 1) & 0 \\
\hline
(5, 2) & 0 \\
\hline
(5, 1, 1) & 0 \\
\hline
(4, 3) & 0 \\
\hline
(4, 2, 1) & $(q - 1)^4 q^9$ \\
\hline
(4, 1, 1, 1) & $(q - 1)^3 q^9$ \\
\hline
(3, 3, 1) & $(q - 1)^4 q^8$ \\
\hline
(3, 2, 2) & $(q + 1)^2 (q - 1)^4 q^6$ \\
\hline
(3, 2, 1, 1) & $(q^3 + 5q^2 + 4q + 1)(q + 1)(q - 1)^3 q^5$ \\
\hline
(3, 1, 1, 1, 1) & $(q + 1)^3 (q - 1)^2 q^5$ \\
\hline
(2, 2, 2, 1) & $(q^4 + 6q^3 + 7q^2 + 4q + 1)(q - 1)^3 q^3$ \\
\hline
(2, 2, 1, 1, 1) & $(q^6 + 5q^5 + 10q^4 + 10q^3 + 7q^2 + 3q + 1)(q - 1)^2 q$ \\
\hline
(2, 1, 1, 1, 1, 1) & $(q^5 + 2q^4 + 4q^3 + 3q^2 + 2q + 1)(q - 1)$ \\
\hline
(1, 1, 1, 1, 1, 1, 1) & 1 \\
\hline
\end{tabular}
    \caption{\small Calculation of $\F_{\mu\hes}$ when $\hes=(1,3,5,6,7,7,7)$}
    \label{Table1}
\end{table}
\end{remark}

\subsection{The intermediate special case of $\F_{\mu\lambda}(q)$}
Recall from~\cref{subsec:mainthm} that the computation of $\F_{\mu\hes}(q)$ is reduced to the computation of $\F_{\mu\lambda}(q)$, which is achieved in ~\cref{mainthm:Fmula-comb} and ~\cref{Intro_them: Main_Macdonald}. In this subsection, we describe the latter theorem and its corollaries. 

For partitions $\nu $ and $\mu$, we  write $\nu\subseteq \mu$ to mean that the Young diagram of $\mu$ contains the Young diagram of $\nu$, i.e., $\nu_i\leq \mu_i$ for all $i\geq 1$. The set-theoretic difference of the Young diagrams of $\mu$ and $\nu$, denoted by $\mu/\nu$, is called a \emph{skew diagram}. 
A \emph{horizontal $k$-strip} is a skew diagram with $k$ boxes with no two squares in the same column.  For instance the skew diagram 
$(5,4,4,1)/(4,4,2)$ is a $4$-horizontal strip:
\smallskip
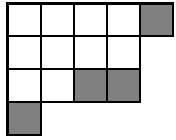
\begin{center}
\scriptsize
\begin{ytableau}
{} &  & & & *(gray) \\
 & & &   \\
  & & *(gray)  & *(gray)  \\
 *(gray) 
\end{ytableau}
\end{center}
For a composition $\Lambda$ of $n$ we use
$\mathrm{sort}(\Lambda)$ to denote the partition obtained by sorting the parts of $\Lambda$ in decreasing order.

\begin{definition}\label{theta:dif}
Let $\lambda$ and $\mu$ be partitions of $n$ with $\ell(\lambda)=s$. We  define 
\[
b_{\mu\lambda}(q)\defin\sum_{\mathcal{F}}\prod_{j=1}^{s}\theta_{\mu^j/\mu^{j-1}}(q),
\]
with
\begin{equation}\label{theta:eq}
\quad \theta_{\eta/\rho}(q)\defin\frac{\left[|\eta|-|\rho|\right]_q!}{\left[\eta_1-\rho_1\right]_q!}\, \prod_{i\geq 1}{\rho_i-\rho_{i+1} \brack \eta_{i+1}-\rho_{i+1}}_q,
\end{equation}
where $\eta/\rho$ is a horizontal strip and $\mathcal{F}:=\mathcal{F}(\mu,\lambda)$ 
is the set of all chains of partitions 
$\emptyset=\mu^{0}\subseteq \mu^{1}\subseteq \cdots\subseteq \mu^{s}=\mu$,
such that $\mu^{j}/\mu^{j-1}$ is a horizontal $\lambda_j$-strip.
By convention, 
$b_{\mu\lambda}(q)=0$ when $\mathcal F=\emptyset$.
Note  that in~\cref{theta:eq} we have $\rho_i-\rho_{i+1}\geq \eta_{i+1}-\rho_{i+1}$ because $\eta/\rho$ is a horizontal strip.
\end{definition}

\begin{theorem}\label{mainthm:Fmula-comb} 
Let $\Lambda$ be a composition of $n$ and set $\lambda = \mathrm{sort}(\Lambda)$. Let $\mu$ be a partition of $n$.
Then  
$\F_{\mu\Lambda}(q)\neq 0$ 
if and only if $\lambda \less \mu'$.
 When $\lambda \less \mu'$, the following hold:
\begin{enumerate} \setlength{\itemsep}{5pt}
\item $
\label{Fmula-comb:Formula}
\F_{\mu \Lambda}(q) = \F_{\mu \lambda}(q)=(q-1)^{n-\ell(\mu)}q^{\binom{n}{2}-\fn(\mu)-(n-\ell(\mu))}\,\, b_{\mu'\lambda}(1/q)$.

\item  $\F_{\mu \lambda}(q)\in \ZZ[q]$,  the degree of 
$\F_{\mu \lambda}(q)$
is $\binom{n}{2} - \fn(\mu)$, and its leading coefficient is $\K_{\mu' \lambda}$, the Kostka number of shape $\mu'$ and content $\lambda$.

\item There exists a polynomial  
$\R_{\mu\lambda}(q)\in\mathbb Z[q]$ such that $\R_{\mu\lambda}(0)=1$, $\R_{\mu\lambda}(1)> 0$, and  
\begin{equation}\label{Into:eq_Fuchs-Kirillov}
\F_{\mu\lambda}(q)=(q-1)^{n-\ell(\mu)}q^{\binom{n}{2}-\binom{\ell(\mu)}{2}-\sum_{j\geq 1}\mu_j'\mu_{j+1}'}\, \R_{\mu\lambda}(q).
\end{equation}

\end{enumerate}

\end{theorem}
\begin{remark}\label{Lambda-length} When the composition $\Lambda$ has only one part i.e., $\Lambda=(n)$, we have $\u_\Lambda(\FF_q)=\{0\}$. Thus $\C_\mu\cap \u_\Lambda(\FF_q)\neq \emptyset$ if and only if $\mu=(1^n)$. In this case the zero matrix is the only element in $\C_\mu\cap\u_\Lambda(\FF_q)$ and so $\F_{\mu\Lambda}(q)=1$.
\end{remark}

Parts (2) and (3) of~\cref{mainthm:Fmula-comb} extend some of the results of
\cite{Borodin} and~\cite[Theorem 4.1]{Fuchs-Kirillov}. The main  tool in the proof of ~\cref{mainthm:Fmula-comb} is a parabolic variant of Borodin's division algorithm to reduce the counting problem  from $\u_{\Lambda}(\FF_q)$ to a smaller nilradical. 

For $\lambda=(1^n)$, Fulman~\cite{Fulman}, showed that $\F_{\mu\lambda}(q)$ can be expressed in terms of the Hall-Littlewood
polynomials. In the special case $\lambda=(1^n)$, 
the formula given for $\F_{\mu\lambda}(q)$ in~\cref{mainthm:Fmula-comb} is different from Fulman's result. The $b_{\mu\lambda}(q)$ are related to the coefficients of Macdonald polynomials  up to an explicit normalization factor, as we explain below.

\begin{definition}\label{psi:dif}
Let $\lambda=(\lambda_1,\dots,\lambda_s)$ be a partition of $n$. For any partition $\mu\vdash n$, we  define 
\begin{equation}\label{functiona:dif}
  a_{\mu\lambda}(q)\defin \sum_{\mathcal F} \prod_{j=1}^{s} \psi_{\mu^{j}/\mu^{j-1}}(q)
\quad\text{with}\quad
\psi_{\eta/\rho}(q)\defin \prod_{i\geq 1}{\eta_i-\eta_{i+1} \brack \eta_i-\rho_i}_q,
\end{equation}
where $\eta/\rho$ is a horizontal strip and $\mathcal{F}$ represents $\mathcal{F}(\mu, \lambda)$ as in~\cref{theta:dif}. By convention, 
$a_{\mu\lambda}(q)=0$ when $\mathcal F=\emptyset$.
\end{definition}

  Let $\P_{\mu}(\x;q,t)$ denote the  Macdonald symmetric function in  $\x=(x_1,x_2,\ldots)$ indexed by $\mu$.
Also, let the $\Q_\mu(\x;q,t)$ denote the duals of the  $\P_\mu(\x;q,t)$ with respect to the $q,t$-scalar product~\cite[p. 323, Eq. (4.12)]{Macdonald}.
From~\cite[p.~346, Eq. (7.13')]{Macdonald}, one obtains 
\begin{equation}\label{monomila_expansion:Mac}
    \P_{\mu}(\x;q,0)=\sum_{\nu}a_{\mu\nu}(q)m_\nu (\x),
\end{equation}
 where $m_\nu(\x)$ is the monomial symmetric function associated to $\nu$. Then by a result of Ram and Schlosser \cite[Theorem~5.3]{Ram}, we have  
\begin{equation}\label{bmuamu}
b_{\mu\lambda}(q)=\left(\prod_{i\geq 1}\frac{[\lambda_i]_q!}{[\mu_i-\mu_{i+1}]_q!}\right)\, a_{\mu\lambda}(q).
\end{equation}
We denote the coefficient of $\x^\lambda=x_1^{\lambda_1}x_2^{\lambda_2}\cdots$ in  $\Phi(\x)$ by $[\x^\lambda]\Phi(\x)$. 
Then using~\cref{bmuamu}, we can state the formula for $\F_{\mu\lambda}(q)$ in ~\cref{mainthm:Fmula-comb}  as follows.

\begin{theorem}\label{Intro_them: Main_Macdonald}
Let $\mu$ and $\lambda=(\lambda_1,\dots,\lambda_s)$ be  partitions of $n$ and $\lambda \less \mu'$. Then
\begin{align*}
\F_{\mu\lambda}(q)&=\frac{(q-1)^{n-\ell(\mu)}q^{\sum_{i\geq 1}\binom{\mu'_i-\mu'_{i+1}}{2}+\binom{n}{2}+\ell(\mu)-\fn(\mu)-\fn(\lambda')-n}[\lambda]_q!}{\prod_{i\geq 1}{[\mu_i'-\mu_{i+1}']_q!}}\, [\x^\lambda]\P_{\mu'}(\x;1/q,0) \\[0.2cm]
&= (q-1)^n q^{\binom{n}{2} - \fn(\mu) - \fn(\lambda') - n} [\lambda]_q!\, [\x^\lambda] \Q_{\mu'}(\x; 1/q, 0),
\end{align*}
where $[\lambda]_q!\defin[\lambda_1]_q!\dots [\lambda_s]_q!$. 
\end{theorem} 

\begin{example}\label{example1} Let $\lambda=(2^k)$ and $\mu=(k,k)$ be two partitions of $2k$. From~\cref{Intro_them: Main_Macdonald} we have:
\begin{equation} \label{M-Foru:application}
\F_{\mu\lambda}(q)=(q-1)^{2k-2}q^{\binom{2k-2}{2}}(q+1)^{k-1}.
\end{equation}
To see this, we first observe that $\mu'=\lambda$.  Indeed, by~\cite[p. 325, Eq. (4.17)]{Macdonald}
we have $\P_{\mu'}(x_1,\dots,x_k; q,t)=(x_1\cdots x_k)^2$.
Thus from~\cref{Intro_them: Main_Macdonald} we obtain: 
\begin{equation*}
\begin{split}
\F_{\mu\lambda}(q) &= (q-1)^{2k-\ell(\mu)} q^{\sum_{i\geq 1}\binom{\mu'_i-\mu'_{i+1}}{2}+\binom{2k}{2}+\ell(\mu)-\fn(\mu)-\fn(\lambda')-n} \left( \prod_{i \geq 1} \frac{[\lambda_i]_q!}{[\mu_i' - \mu_{i+1}']_q!} \right)\\
&=(q-1)^{2k-2}q^{\binom{2k-2}{2}}\frac{([2]_q!)^k}{[2]_q!}=(q-1)^{2k-2}q^{\binom{2k-2}{2}}(q+1)^{k-1}.
\end{split}
\end{equation*} 
\end{example}

We remark that~\cref{Intro_them: Main_Macdonald} is a special case of~\cref{Hess_theorem}  in a direct way.
Recall that $\G = \GL_n(\FF_q)$ and $\U_\lambda = \{I + X : X \in \u_\lambda(\FF_q)\}$. 
Now $|\U_\lambda|=q^{\binom{n}{2}-\fn(\lambda')}$. Furthermore,
if $\langle
\,,\,\rangle$ 
denotes the Hall scalar product~\cite[Sec. I.4]{Macdonald}, then
$\left\langle h_\mu(\x), m_\lambda(\x)\right\rangle = \delta_{\mu\lambda}$, where the $h_\mu$ are the complete symmetric functions.
It follows that
\begin{align*}
[\x^\lambda]\, q^{\fn(\mu)}&\P_{\mu'}(\x;1/q,0)=
\left\langle h_\lambda(\x), q^{\fn(\mu)}\P_{\mu'}(\x;1/q,0)\right\rangle\\
&=
\left\langle h_\lambda(\x),\omega\tilde{\HH}_{\mu}(\x; q)\right\rangle=
\left\langle\omega h_\lambda(\x),\tilde{\HH}_{\mu}(\x; q)\right\rangle
=
\left\langle e_\lambda(\x),\tilde{\HH}_{\mu}(\x; q)\right\rangle,
\end{align*}
where the $e_\lambda(\x)$ are the elementary symmetric functions. 
Thus,~\cref{Intro_them: Main_Macdonald} is equivalent to the following assertion.

\begin{theorem}\label{F_mulambda_G:thm} Let $\mu$ and $\lambda$ be partitions of $n$ with $\lambda\less \mu'$. Then 
\begin{align}\label{F_mulambda_G:eq}
    \F_{\mu\lambda}(q)&
    =\frac{(q-1)^n|\U_\lambda|}{|\Z_G(I+J_\mu)|} [\lambda]_q!\,  [\x^\lambda] q^{\fn(\mu)}\P_{\mu'}(\x;1/q,0)
    =\frac{(q-1)^n|\U_\lambda|}{|\Z_G(I+J_\mu)|}\left\langle [\lambda]_q!e_{\lambda}(\x),\tilde{\HH}_{\mu}(\x;q)\right\rangle.
\end{align}   
\end{theorem}
Now the key observation is that  $[\lambda]_q!e_\lambda(\x)$  is the chromatic quasisymmetric function of a disjoint union of complete graphs. 
\cref{Hess_theorem} establishes that this interpretation of the formula for $\F_{\mu\lambda}(q)$ extends  verbatim to $\F_{\mu \hes}(q)$!

We conclude this section by noting  another perspective on $\F_{\mu\hes}(q)$ that comes from representation theory. Let $\G$ be a finite group. For $u\in \G$ the conjugacy class of $u$ and the centralizer of $u$ are denoted by  $\conj(u)$ and $\cen(u)$, respectively. 
Then for any subgroup $\HH$ of $\G$ we have
\begin{equation}\label{Counting}
    \#\{g\in \G: g^{-1}ug\in \HH\}= |\conj(u)\cap \HH|\cdot |\cen(u)|.
\end{equation}   
Now suppose $\mathrm{ch}(\Ind_{\HH}^{\G}1)(u)$ denotes the value of the character of the induced representation $\Ind_{\HH}^{\G}1$ at $u\in\G$. Then 
\begin{align}\label{ind-form}
\mathrm{ch}(\Ind_{\HH}^{\G}1)(u)&=\#\left\{g\HH: g^{-1}ug\in \HH\right\}
=
\frac{\#\left\{g\in \G: g^{-1}ug\in \HH\right\}}{|\HH|}=\frac{ |\cen(u)|}{|\HH|}|\conj(u)\cap \HH|. 
\end{align}
Let
\begin{equation}\label{U_hes:eq}
\U_\hes=\U_\hes(\FF_q) \defin \{I + A : A \in \u_\hes(\FF_q)\}
\end{equation}
denote the unipotent subgroup of $\G=\GL_n(\FF_q)$ associated with $\hes$. Let $\chi_\hes$ be the character of the induced representation $\Ind_{\U_\hes}^\G \mathbf{1}$. Then from~\cref{Counting,ind-form} we have:
\begin{equation}\label{ind-form-char}
\chi_\hes(I + J_\mu) = \frac{|\Z_\G(I + J_\mu)|}{|\U_\hes|} \F_{\mu\hes}(q).
\end{equation}
Therefore, understanding $\F_{\mu\hes}(q)$ reduces to analyzing the induced representation of the trivial representation from $\U_\hes$ to $\G$. This approach was explored by Gagnon, among other things, in his recent and inspiring paper~\cite{Gagnon}. To describe the connection with Gagnon's work, we need to recall the work of  Zelevinsky~\cite{Zelevinsky} which introduces a family of Hopf algebras, known as PSH-algebras, that unify the representation theories of $\mathrm{S}_n$ and $\GL_n(\FF_q)$. In Gagnon's notation, Zelevinsky's construction produces a graded Hopf algebra isomorphism
\[
\mathbf{p}_{\{1\}} \colon \mathsf{cf}^{\mathrm{uni}}_{\mathrm{supp}}(\mathrm{GL}_\bullet) \to \mathrm{Sym}, \quad \delta_\lambda \mapsto \tilde{\P}_\lambda(\x; q) := q^{-\fn(\lambda)} \P_\lambda(\x; q^{-1}),
\]
where $\mathrm{Sym}$ is the ring of symmetric functions with coefficients in $\mathbb Q(q)$ and \[
\P_\lambda(\x; t)\defin\P_\lambda(\x;0,t)
\] denotes a Hall-Littlewood  function. Using this, Gagnon proves~\cite[Theorem 3.1]{Gagnon} that 
\[
\chi_\hes = (q - 1)^n \, \mathbf{p}_{\{1\}}^{-1}(X_{\graph(\hes)}(\x; q)).
\]
Gagnon then applies this identity to derive a formula for $\chi_\hes$. As 
$\tilde{\P}_\lambda(\x; q)$ is dual to the modified Hall-Littlewood basis $\tilde{\HH}_\lambda(\x;q)$ with respect to the Hall scalar product, this formula is in principle equivalent to~\cref{Hess_theorem} above. Our proof of~\cref{Hess_theorem} is entirely combinatorial and thus in some sense more elementary.

\section{Applications}
In this section, we explain three applications of our 
results on $\F_{\mu\lambda}(q)$ and $\F_{\mu\hes}(q)$. 
Our first application  is in counting the number of points on nilpotent Hessenberg varieties. 
Varieties of this kind have been studied in~\cite{Tymoczko,Ji-Precup,Precup-Sommers}.
We recall the definition of a  nilpotent Hessenberg variety from~\cite{Precup-Sommers}. 
In what follows, the conjugate of a Hessenberg function $\hes\in\mathcal{H}_n$ is defined by 
$$
\hes': [n]\to [n],\quad i\mapsto \#\{j\in [n]: \hes(j)\geq n+1-i\}. 
$$
\begin{definition}[Nilpotent Hessenberg variety]\label{Nil_hes:def} Let $\hes \in \mathcal{H}_n$ be a Hessenberg function, and let $X$ be a nilpotent $n\times n$ matrix with entries in $\FF_q$. Let 
$\mathscr{F}_n$ denote
the set of complete flags in $\FF_q^n$. The nilpotent Hessenberg variety associated to $\hes$ and $X$ is defined by 
$$
\text{$\mathcal{H}$ess$_\mathrm{nil}$}(\hes,X)\defin \left\{V_\bullet=(V_i)_{1\leq i\leq n}\in\mathscr{F}_n: XV_i\subseteq V_{\mathsf{e}(i)} \right\},
$$
where $\mathsf{e}(i):=n-\hes'(n+1-i)$.
\end{definition}
\begin{example}
  For the Hessenberg function $\hes=(2,3,5,5,5)$ (see Example \ref{eg:hess}), we have $\hes'=(3,3,4,5,5)$ and $\mathsf{e}=(0,0,1,2,2)$.
\end{example}

\begin{remark} We have $\#\text{$\mathcal{H}$ess$_\mathrm{nil}$}(\hes,X)=\#\text{$\mathcal{H}$ess$_\mathrm{nil}$}(\hes,g^{-1}Xg)$ for any $g\in \GL_n(\FF_q)$. 
  In~\cref{Nil_hes:def}, the function $\mathsf{e}$ is non-decreasing and satisfies $\mathsf{e}(i)<i$. 
  We refer the reader to~\cite{Abreu-Nigro-Ram} for counting the number of $\FF_q$-points on a variant of the above Hessenberg varieties, defined by
$$
\text{$\mathcal{H}$ess}(\hes,X)\defin \left\{V_\bullet=(V_i)_{1\leq i\leq n}\in\mathscr{F}_n: XV_i\subseteq V_{\mathsf{h}(i)} \right\}.
$$
\end{remark}

\begin{theorem}
\label{thm:nilphess}
 Let $\hes\in\mathcal{H}_n$ be a Hessenberg function and let $X$ be a nilpotent $n\times n$ matrix with entries in $\FF_q$, with Jordan form $J_\mu$. Then 
$$
\#\mathrm{\mathcal{H}ess_{nil}}(\hes,X)=\frac{|\Z_\G(I+J_\mu)|}{(q-1)^nq^{\binom{n}{2}}}\F_{\mu\hes}(q)=q^{-E_\hes}\langle \tilde{\HH}_\mu(\x;q),X_{\graph(\hes)}(\x;q)\rangle,
$$
where $E_\hes$ denotes the number of edges of the graph $\graph(\hes)$.
\end{theorem}

Our second application involves a formula conjectured by Kirillov–Melnikov~\cite{KirillovMelnikov}, and  proved by Ekhad-Zeilberger~\cite{Ekhad}.
Let $\Lambda$ be a composition of $n$ and define
\begin{equation}
\label{eq:ankaka}
a_{nk}(\Lambda)\defin \# \left\{X\in \u_{\Lambda}(\FF_q): X^2=0, \mathrm{rank}(X)=k\right\}.
\end{equation}
From~\cref{mainthm:Fmula-comb} it follows that \[
a_{nk}(\Lambda)=a_{nk}(\lambda),
\quad\text{ for }\lambda=\mathrm{sort}(\Lambda). 
\]
Our analysis of $a_{nk}(\lambda)$ starts with the observation that a matrix $X \in \mathfrak{u}_\lambda(\mathbb{F}_q)$, with rank $k$ and $X^2 = 0$, must have a Jordan form of type $\mu = (2^k, 1^{n-2k})$ and so $\mu' = (n-k, k)$, where $0 \leq k \leq n/2$.
Therefore, thanks to~\cref{Intro_them: Main_Macdonald}, we have
\begin{equation}\label{Into:ankQ_formula}
a_{nk}(\lambda)=(q-1)^n q^{nk-k^2- \fn(\lambda') - n} \left( \prod_{1\leq i\leq s} [\lambda_i]_q! \right) [\x^\lambda] \Q_{\mu'}(\x; 1/q, 0).
\end{equation}

Let the $f_{\nu\rho}^\lambda(t)$ be the polynomials defined in \cite[Sec. III.3, p. 215]{Macdonald} which describe the structure constants of products of Hall-Littlewood polynomials. Indeed the $f_{\nu\rho}^\lambda(t)$ are related to the Hall polynomials $g_{\nu\rho}^\lambda(t)$ as given in  \cite[Eq. (3.6), p. 217]{Macdonald}.
Using~\cref{Intro_them: Main_Macdonald}, we prove the following combinatorial formula for $a_{nk}(\lambda)$: 

\begin{theorem}\label{anklambda-explicit} Let $\lambda$ be a partition of $n$. Then 
for $0\leq k\leq n/2$ we have 
\[
a_{nk}(\lambda)=q^{nk-k^2-\fn(\lambda')} \bar a_{nk}(\lambda),
\] where $\bar a_{nk}(\lambda)$ is equal to 
\begin{equation*}
\sum_{j=0}^k (-1)^jq^{-\binom{j+1}{2}-(n-2k)j}{n-2k+j\brack j}_q\,\frac{q^{n-2k+2j}-1}{q^{n-2k+j}-1} \sum_{\substack{\nu\, \vdash n-k+j\\ \rho\,\vdash k-j}}\frac{q^{\fn(\nu')+\fn(\rho')}f_{\nu\rho}^\lambda(1)}{[\nu]_q!\,[\rho]_q!}\, ,    
\end{equation*}
with the convention that $(q^0-1)/(q^0-1)=1$.

\end{theorem}
Note that $f_{\nu\rho}^\lambda(t)=0$ unless $\nu,\rho\subseteq \lambda$ and $|\nu|+|\rho|=|\lambda|$.
When $\lambda=(1^n)$, from the definition of $f_{\nu\rho
}^\lambda(t)$ it follows easily that \begin{equation}\label{Hall}
\sum_{\substack{\nu\, \vdash n-k+j\\ \rho\,\vdash k-j}}\frac{q^{\fn(\nu')+\fn(\rho')}f_{\nu\rho}^\lambda(1)}{[\nu]_q!\,[\rho]_q!}=\binom{n}{k-j}.
\end{equation}

Next define
\[
a_{nk}\defin\#\left\{X\in\u_n(\FF_q)\,:\,X^2=0\text{ and }\mathrm{rank}(X)=k\right\}.
\]
For $\lambda=(1^n)$ 
the 
formula in~\cref{anklambda-explicit} for $a_{nk}(\lambda)$ specializes to \begin{equation}\label{ank_explicitformula}
a_{nk}=\sum_{j=0}^k(-1)^jq^{nk-k^2-\binom{j+1}{2}-(n-2k)j}{n-2k+j \brack j}_q\binom{n}{k-j}\frac{q^{n-2k+2j}-1}{q^{n-2k+j}-1}.
\end{equation}
Now define
\[
\mathrm{C}_n=\mathrm{C}_n(q)\defin\#
\left\{X\in\u_n(\FF_q)\,:\,X^2=0\right\}.
\]
Formula~\cref{ank_explicitformula} has the following immediate consequence:
\begin{corollary}
\label{cor-diffekhad}
For $n\geq 1 $ we have \begin{equation*}
\mathrm{C}_{n}=
\sum_{0\leq k_1+k_2\leq\lfloor\frac n 2 \rfloor}q^{nk_1-k_1^2}\binom{n}{k_1}
(-1)^{k_2}q^{\binom{k_2}{2}}{n-2k_1-k_2 \brack k_2}_q\frac{[n-2k_1]_q}{[n-2k_1-k_2]_q}
.\end{equation*}
\end{corollary}
The problem of computing $\mathrm{C}_n$ was considered by Kirillov and Melnikov in~\cite{KirillovMelnikov}, where they conjectured a different formula for $\mathrm{C}_n$ that was later proved by Ekhad and Zeilberger~\cite{Ekhad}; see~\cref{Ekhad_thm} below.
Let $\chi_3$ be the Dirichlet character defined by
\begin{equation}\label{Dirichlet}
\chi_3(n):=\begin{cases}
0 & n\equiv 0\pmod{3};\\
1 & n\equiv 1\pmod{3};\\
-1 &  n\equiv 2\pmod{3}.\\
\end{cases}
\end{equation}
\begin{theorem}\label{Ekhad_thm}  (Ekhad--Zeilberger) For $n\geq 1$ we have 
\begin{equation*}
    \begin{split}
        \mathrm{C}_{2n+1}&=\sum_{j=0}^n \binom{2n+1}{n-j}\left[\chi_3(j+1)q^{\frac{1}{3}\binom{2j+1}{2}}-\chi_3(j)q^{\frac{1}{3}\binom{2j+2}{2}}\right]q^{n^2+n-j^2-j}\\
        &=\sum_{j\in\ZZ}\left[\binom{2n+1}{n-3j}-\binom{2n+1}{n-3j-1}\right]q^{n^2+n-3j^2-2j},        
    \end{split}
\end{equation*}
 and
\begin{equation*}
    \begin{split}
         \mathrm{C}_{2n}&= \binom{2n}{n}q^{n^2}+\sum_{j=1}^{n}\binom{2n}{n-j}\left[\chi_3(2j+1)q^{\frac{1}{3}\binom{2j}{2}}-\chi_3(2j-1)q^{\frac{1}{3}\binom{2j+1}{2}}\right]q^{n^2-j^2}\\
        &=\sum_{j\in\ZZ}\left[\binom{2n}{n-3j}-\binom{2n}{n-3j-1}\right]q^{n^2-3j^2-2j}.      
    \end{split}
\end{equation*}  
\end{theorem}

The Ekhad-Zeilberger proof of
~\cref{Ekhad_thm}
is very short, but 
R. Stanley notes in~\cite[p. 194]{Stanley} that ``no conceptual reason is known for such a simple formula''. It turns out that 
\cref{cor-diffekhad} sheds some light on the formulas for $\mathrm{C}_n$ in~\cref{Ekhad_thm}.
Indeed 
we give a short proof that \cref{cor-diffekhad} implies~\cref{Ekhad_thm},
using a formula for Macdonald polynomials $\Q_{(\mu_1,\mu_2)}(\x\, ;q,t)$ due to Jing and J\'{o}zefiak~\cite[Eq. (1)]{Jing}.

In~\cite{KirillovX=0}, Kirillov discovered the mysterious relation  
\begin{equation}\label{X2Hermit}
\sum_{k \geq 0} a_{nk} \, q^{k(k-n)} \, H_{n-2k}(w \mid 1/q) = (2w)^n,
\end{equation}
where the \( H_{n-2k}(w \mid 1/q) \) are the continuous $q$-Hermite polynomials as known in the Askey–Wilson scheme. The proof of~\cref{X2Hermit} 
in~\cite{KirillovX=0}
is based on an algebraic structure which Kirillov calls ``fusion rules'' on the set of coadjoint orbits of \( \GL_n(\FF_q) \) on \( \u_n(\FF_q) \). In~\cite{Kirillov}, Kirillov writes ``This algebra is in fact one of the realizations of the Hall algebra, but we shall not discuss this phenomenon here.'' 
In~\cref{eq: Kirillov-general} below, we obtain a generalization of 
Kirillov's mysterious formula~\cref{X2Hermit}, and show that it follows easily from the Cauchy identity for Macdonald polynomials combined with~\cref{Intro_them: Main_Macdonald}.
Recall that the continuous $q$-Hermite polynomials $H_n(w\, |\, q)$ are defined recursively by
$$
H_{n+1}(w\, |\, q)=2wH_n(w\, |\, q)+(q^n-1)H_{n-1}(w\, |\,q), \qquad n\geq 0,
$$
with the initial conditions $H_0(w\,|\,q):=1$ and $H_{n}(w\,|\,q):=0$ when $n\leq -1$. One can also show~\cite[Chapter 13]{Ismail}
 that 
 \begin{equation}
\label{HnvsRogerSzego}
H_n(\cos(\theta)\, |\, q)=\sum_{k=0}^n{n \brack k}_qe^{i(n-2k)\theta}.
\end{equation}
It is  a classical result that the $H_n(w\, |\, 1/q)$ form a sequence of  orthogonal polynomials, i.e.,
\begin{equation}\label{ank:integral}
\frac{1}{2\pi}\int_{0}^\pi H_n(\cos(\theta)\, |\, 1/q)H_m(\cos(\theta)\, |\, 1/q) (e^{2i\theta},e^{-2i\theta};1/q)_{\infty}=\frac{(1/q;1/q)_n}{(1/q;1/q)_\infty}\delta_{nm},
\end{equation}
where 
\begin{align*}
(e^{2i\theta},e^{-2i\theta}\,;1 /q)_{\infty}&\defin\left(e^{2i\theta}\,; 1/q\right)_{\infty}\,\left(e^{-2i\theta}; 1/q\right)_{\infty}
=\prod_{j\geq 0}\left(1-q^{-j}e^{2i\theta}\right)\prod_{j\geq 0}\left(1-q^{-j}e^{-2i\theta}\right).
\end{align*}
From this, one can conclude the following remarkable formula (due to Kirillov)
\footnote{
From~\cref{HnvsRogerSzego} it is probably more natural to express~\cref{ank-Hermite,ank:integral} in terms of Roger-Szeg\Hung{o} polynomials $h_n(x;q)=\sum_{k=0}^n{n \brack k}_qx^k$. However, we have kept our exposition loyal to Kirillov's articles~\cite{KirillovX=0,Kirillov}. We thank Ole Warnaar for bringing this point to our attention.
}
\begin{equation}\label{ank-Hermite}
a_{nk}=M(n,k,q)\frac{1}{2\pi}\int_{0}^\pi ( 2\cos(\theta))^{n} H_{n-2k}(\cos(\theta)\, |\, 1/q)(e^{2i\theta},e^{-2i\theta}\, ;1/q)_{\infty}\, d\theta 
\end{equation}
where
\[
M({n,k,q})\defin
q^{k(n-k)}\prod_{j=1}^{n-2k}\left(1-\frac{1}{q^j}\right)^{-1}.
\]Below we discuss a generalization of~\cref{X2Hermit}. To this end, first we introduce some notation. Recall that the Chebyshev polynomials of the first kind $T_n(w)$ are defined by
$
T_n(\cos(\theta)) = \cos(n\theta),
$
which are also derived from the recurrence relation for $n \geq 2$:
\[
T_{n+1}(w) = 2wT_n(w) - T_{n-1}(w), \qquad T_0(w) = 1, \quad T_1(w) = w.
\]
For a partition $\rho = (\rho_1, \rho_2, \dots)$, define
\begin{equation}
\label{eq:Trho}
T_{\rho}(w) \defin T_{\rho_1}(w)T_{\rho_2}(w)\cdots .
\end{equation}
For any partition $\rho \vdash n$, we set
\begin{equation}\label{z_rho}
z_\rho \defin \prod_{i \geq 1} i^{m_i} \cdot m_i!\, ,
\end{equation}
where $m_i = m_i(\rho)$ is the number of parts of $\rho$ equal to $i$. Note that $z_\rho$ is the size of the centralizer of a permutation of type $\rho$ in the symmetric group $\mathrm{S}_n$. 

Evidently we have $a_{nk}(\lambda)=0$ when $k>n/2$. 
We prove the following assertion. 

\begin{theorem}\label{eq: Kirillov-general} Let $\lambda=(\lambda_1,\dots,\lambda_s)$ be a partition of $n$. Then
\begin{equation*}
\sum_{k=0}^n a_{nk}(\lambda)q^{k(k-n)}H_{n-2k}(w\,|\, 1/q)=[\lambda]_q! (q-1)^nq^{-\fn(\lambda')}\prod_{i=1}^s\sum_{\rho\, \vdash \lambda_i}\frac{2^{\ell(\rho)}T_{\rho}(w)}{z_\rho\prod_{i=1}^{\ell(\rho)}\left(q^{\rho_i}-1\right)}.
\end{equation*}
\end{theorem}
This theorem is indeed a special case of~\cref{Cauch_PQg}. To illustrate this formula we provide two examples.
\begin{example} Consider $\lambda=(1^n)$. In this case we have $[\lambda]_q!=1$, $\fn(\lambda')=0$.
Moreover for any $1\leq i\leq n$, the only partition $\rho$ that contributes in the right hand side of the formula in~\cref{eq: Kirillov-general}  is $\rho=(1)$ which gives $z_\rho=1$. Then the right hand side of the formula in~\cref{eq: Kirillov-general} is 
$$
(q-1)^n\prod_{i=1}^n\frac{2T_1(w)}{q-1}=(2T_1(w))^n=(2w)^n.
$$
Thus, ~\cref{X2Hermit} is a special case of~\cref{eq: Kirillov-general}.  
\end{example}

\begin{example}
Now for $\lambda=(2^l)\vdash n$ we obtain $[\lambda]_q!=(q+1)^l$, $\fn(\lambda')=l$. Moreover for any $1\leq i\leq l$, the partitions $\rho$ that contributes in the right hand side of the formula in~\cref{eq: Kirillov-general}  are $\rho_1=(1,1)$ and $\rho_2=(2)$ which gives $z_{\rho_1}=z_{\rho_2}=2$. Thus  the right hand side of the formula in~\cref{eq: Kirillov-general} is 
$$
(q+1)^l(q-1)^{2l}q^{-l}\prod_{i=1}^l\left(\frac{4T_1^2(w)}{2(q-1)^2}+\frac{2T_2(w)}{2(q^2-1)}\right),
$$ which simplifies to
$$
(q+1)^l(q-1)^{2l}q^{-l}\left(\frac{4qw^2-q+1}{(q-1)^2(q+1)}\right)^l.
$$
Thus for $\lambda=(2^l)\vdash n$ we obtain 
$$\sum_{k\geq 0} a_{nk}(\lambda)q^{k(k-n)}H_{n-2k}(w\,|\, 1/q)=\left(4w^2-1+\frac{1}{q}\right)^l. $$
\end{example}
We can use orthogonality to represent $a_{nk}(\lambda)$ as an integral, similar to~\cref{ank:integral}. This integral can be evaluated using Jacobi's triple product, but we do not  discuss this calculation.

We end this section with the third application of our main results. For a partition $\mu$ of $n$, as before we denote the nilpotent Jordan matrix of type $\mu$ by $J_\mu$.  

\begin{theorem}\label{Double_cosets:thm}  Let $\hes_1$ and $\hes_2$ be two Hessenberg functions on $[n]$, and let $\U_{\hes_1}$ and $\U_{\hes_2}$ be the unipotent subgroups of $\G = \GL_n(\FF_q)$ associated with $\hes_1$ and $\hes_2$ as defined in~\cref{U_hes:eq}. Then
\begin{equation}\label{double-center}
\left|\U_{\hes_1}\backslash\G/\U_{\hes_2}\right|={q^{|\hes_1|+|\hes_2|-2n^2}}\sum_{\mu\, \vdash n} |\cen(I+J_\mu)|\,\F_{\mu\hes_1}(q)\,\F_{\mu\hes_2}(q),
\end{equation}
where $\cen(I+J_\mu)$ is the centralizer  of $I+J_\mu$ in $\G$.
\end{theorem}
By combining~\cref{center_size1,varphi_m} we obtain 
\begin{equation}\label{center_size2}
  \left|\cen(I+J_\mu)\right|=(q-1)^{\ell(\mu)}q^{n+2\fn(\mu)-\ell(\mu)-\sum_{i\geq 1}\binom{\mu_i'-\mu_{i+1}'}{2}}\prod_{i\geq 1}[\mu_i'-\mu_{i+1}']_q!.  
\end{equation}
From~\cref{center_size2} we infer that $|\Z_\G(I+J_\mu)|$ is a polynomial in $q$ of degree $n+2\fn(\mu)$. Recall that $\U_{k_\Lambda}=\U_\Lambda$ where $k_\Lambda$ is the Hessenberg function defined in~\cref{k_partiton}. Moreover, for a composition $\Lambda$ of $n$, one can easily verify that
\begin{equation}\label{hess_part}
n^2 - |k_\Lambda| = \binom{n}{2} - \fn(\lambda'),
\end{equation}
where $\lambda = \mathrm{sort}(\Lambda)$. Thus, for $h_1 = k_\Lambda$ and $h_2 = k_\Gamma$ where $\Lambda$ and $\Gamma$ are compositions of $n$, from~\cref{Double_cosets:thm}, the equality~\cref{hess_part}, the second part of~\cref{mainthm:Fmula-comb}, and the well-known fact that $\K_{\mu\Lambda} = \K_{\mu\lambda}$~\cite[Theorem 9.25]{Loehr}, we obtain

\begin{corollary}\label{Cor:doublecoset} 
Let $\Lambda$ and $\Gamma$ be two compositions of $n$. Set $\lambda=\mathrm{sort}(\Lambda)$ and $\gamma=\mathrm{sort}(\Gamma)$.
Then 
$ |\U_\Lambda \backslash \G / \U_\Gamma|$ is a polynomial in $q$ of degree $n + \fn(\lambda') + \fn(\gamma')$, with leading coefficient 
$\sum_{\mu} \K_{\mu' \lambda} \K_{\mu' \gamma}$.
\end{corollary}

\begin{remark} This corollary has an interesting combinatorial interpretation. First note that when $\lambda=(1^n)$ we have $\K_{\mu'\lambda}=\K_{\mu\lambda}$ and in this case $\K_{\mu\lambda}$ is the number of standard Young tableaux of shape $\mu$, denoted by $f^\mu$. Moreover,  when $\lambda=\gamma=(1^n)$,  from Bruhat decomposition we have $\GL_n(\FF_q)= \sqcup_{t \in T} \sqcup_{w \in \mathrm{S}_n} U t w U$ 
where $\U=\U_\lambda(\FF_q)$ and $T$ is the subgroup of diagonal  matrices of $\G$. From this we get $|\U \backslash G / \U|=n!(q-1)^n$. Thus from~\cref{Cor:doublecoset} we obtain the well-known formula
$
n!=\sum_{\mu\, \vdash n} \left(f^\mu\right)^2
$.
\end{remark}

\subsection{Computation of $\F_{\mu\lambda}(q)$} 
Computing the polynomials $\F_{\mu\lambda}(q)$ based on~\cref{Intro_them: Main_Macdonald} can be a highly  time consuming computation. Here we explain an alternative and significantly more efficient method for this computation that is based on 
a result of 
Ram and Schlosser~\cite{Ram}. 
To this end, we first introduce two functions $r(T)$ and s$(T)$ of a semistandard Young tableau $T$ ($r(T)$ was  initially introduced by Prasad and Ram~\cite{Prasad} for standard Young tableaux).
Henceforth,  the set of semistandard Young tableaux of shape $\mu$ and content $\lambda$ will be denoted by $\ssyt(\mu,\lambda)$.

\begin{definition}\label{def:bit}
Suppose $T\in \ssyt(\mu,\lambda)$. For each $i\geq 1$, let $\beta^i_{j}=\beta^i_j(T)$ denote the number of occurrences of the integer $i$ in the $j$th row of $T$ for $j\geq 1$. Define $\beta^i=\beta^i(T):=(\beta^i_{1},\beta^i_{2},\ldots)$ for $i\geq 1$. 
\end{definition}

Note that $\beta^i$ is not necessarily a partition, but it is a composition. Also, we have $|\beta^i|=\lambda_i$.

\begin{definition}\label[definition]{def:sqt}
Given $T\in \ssyt(\mu,\lambda)$, let $\beta^i=\beta^i(T)$ and define
\begin{align*}
  s_q(T)\defin\prod_{i\geq 1}{\lambda_i \brack \beta^i}_q\,\qquad \text{where}\qquad 
{\lambda_i \brack \beta^i}_q\defin\frac{[\lambda_i]_q!}{\prod_{j\geq 1}[\beta^i_j]_q!}\, .
\end{align*}
\end{definition}

\begin{definition}\label{def:rqt}
Suppose $T\in \ssyt(\mu,\lambda)$.  
Let $T_{ij}$ denote the entry in row $i$ and column $j$ of  $T$. For each box of $T$ in the $(i,j)$ position where $i\geq 2$, we define $r_{ij}(T)$ to be
\begin{equation}\label{rijT}
  r_{ij}(T)\defin\#\{j':j'\geq j\mbox{ and }T_{i-1,j'}<T_{ij}\}.
\end{equation}
Now we set
\begin{align*}  r_q(T)\defin\prod_{
\substack{
(i,j)
\in T\\
 i
 \geq 2}}[r_{ij}(T)]_q.
\end{align*}  
\end{definition}
We illustrate these definitions by the following  example.
\begin{example}\label{example} Consider the semistandard tableau $T$ below. The table on the right describes $r_{ij}(T)$ as defined in~\cref{rijT}:   
  \begin{align}\label{Example_T}
  \scriptsize
T=\ytableaushort{112,233,4 \none}\qquad\qquad r_{ij}(T): 
   \ytableaushort{\ \ \  ,221\none,3 \none}
  \end{align}
Thus $r_q(T)=[2]_q^2[3]_q=(1+q)^2(1+q+q^2)$. Moreover, for this tableau $T$, we have $\beta^1=(2,0,0)$, $\beta^2=(1,1,0)$, $\beta^3=(0,2,0)$ and $\beta^4=(0,0,1)$, which yield $s_q(T)=[2]_q=q+1$.
  \end{example} 

Ram and Schlosser observed the following relation~\cite[Definition 3.7 and Proposition 5.1]{Ram}. For $T\in\ssyt(\mu',\lambda)$ we have 
$$
s_q(T)r_q(T)=\prod_{j=1}^s\theta_{\mu^j/\mu^{j-1}}(q),
$$
where $\emptyset=\mu^0\subseteq\mu^1\subseteq\cdots\subseteq\mu^s=\mu'$ is a chain of partitions such that every $\mu^j/\mu^{j-1}$ is a horizontal $\lambda_j$-strip. As elements of $\mathcal{F}(\mu',\lambda)$ are in one-to-one correspondence with elements of $\ssyt(\mu',\lambda)$, from the definition of $b_{\mu'\lambda}(q)$ we infer that

\begin{equation}\label{Ram-Schlosser} 
b_{\mu'\lambda}(q)=\sum_{T\in \mathrm{SSYT}(\mu',\lambda)}s_q(T)r_q(T). 
\end{equation}
The next example explains how to compute $\F_{\mu\lambda}(q)$ using~\cref{mainthm:Fmula-comb} combined with~\cref{Ram-Schlosser}.
\begin{example}\label{Example:bmu} Let $\mu=(3,2,2)$ and $\lambda=(2,2,2,1)$. Then $\binom{n}{2}+\ell(\mu)-\fn(\mu)-n=11$ and $n-\ell(\mu)=4$. There are only three tableaux of shape $\mu'$ and content $\lambda$:
  \begin{align*}
  \scriptsize
T_1=\begin{ytableau}
 1 & 1 & 2 \\
 2 & 3 & 3\\
 4
\end{ytableau}\qquad\qquad 
T_2=\begin{ytableau}
 1 & 1 & 2 \\
 2 & 3 & 4\\
 3
\end{ytableau}\qquad\qquad
T_3=\begin{ytableau}
 1 & 1 & 3 \\
 2 & 2 & 4\\
 3
\end{ytableau}
  \end{align*}
Hence
$s_q(T_1)r_q(T_1)= [2]_q^3[3]_q$, $s_q(T_2)r_q(T_2)=[2]_q^4$ and 
$s_q(T_3)r_q(T_3)=[2]_q^3$. From~\cref{Ram-Schlosser} we obtain: 
$$
b_{\mu'\lambda}(q)=(q+1)^3(q^2+2q+3).
$$
This, in combination with the first part of~\cref{mainthm:Fmula-comb}, gives:
\begin{equation*}
\F_{\mu\lambda}(q)=(q-1)^4q^{11}\left(1+1/q\right)^3\left(3+2/q+1/q^2\right)=(q-1)^4q^{6}(q+1)^3(3q^2 + 2q + 1).
\end{equation*}
\end{example}

\begin{remark} When $\lambda = (2^k)$ and $\mu = (k,k)$, we obtain $\K_{\mu'\lambda} = 1$. The same argument as in~\cref{Example:bmu} can also be used to show~\eqref{M-Foru:application}. Moreover, using the same approach as in~\cref{Example:bmu}, one can show that for $\mu = \lambda = (k^k)$, we have:
$$
\F_{\mu\lambda}(q)=(q-1)^{k^2-k}q^{(k^2-k-1)\binom{k}{2}}\, \left([k]_q!\right)^{k-1}.
$$
For arbitrary partitions $\lambda$ and $\mu$, from~\cref{Into:eq_Fuchs-Kirillov}, we know the exact exponent of $q$ and $q-1$ in $\F_{\mu\lambda}(q)$. \cref{Table2} shows that $\F_{\mu\lambda}(q)$ for $\lambda=(2,2,2,2)$ is a product of powers of $q$, $q\pm 1$, $q$-integers and another polynomial that is irreducible over $\mathbb Q$. It is tempting to conjecture that this holds in general, and in particular to compute the exponent of $q+1$. 

\begin{table}[h!]\scriptsize
\renewcommand{\arraystretch}{1.5} 
\centering
\begin{tabular}{|c|c|}
\hline
\text{$\mu$} & $\F_{\mu\lambda}(q)$ \\
\hline
$ (4^2)$ & $(q - 1)^6 q^{15}(q + 1)^3 $ \\
\hline
$(4, 3, 1)$ & $(q - 1)^5 q^{13}(q + 1)^4(3q + 1) $ \\
\hline
$ (4, 2^2)$ & $(q - 1)^5 q^{12}(q + 1)^4(2q + 1) $ \\
\hline
$ (4, 2, 1^2)$ & $(q - 1)^4 q^{11}(q + 1)^4(3q^2 + 2q + 1) $ \\
\hline
$ (4, 1^4)$ & $(q - 1)^3 q^{11}(q + 1)^4$   \\
\hline
$ (3^2, 2)$ & $(q - 1)^5 q^{10}(q + 1)^4(3q^2 + 2q + 1) $ \\
\hline
$ (3^2, 1^2)$ & $(q - 1)^4 q^{10}(q + 1)^2(6q^4 + 9q^3 + 9q^2 + 3q + 1) $ \\
\hline
$ (3, 2^2, 1)$ & $(q - 1)^4 q^7(q + 1)^4(7q^4 + 7q^3 + 7q^2 + 2q + 1) $ \\
\hline
$ (3, 2, 1^3)$ & $(q - 1)^3 q^6(q + 1)^3 (8q^5 + 9q^4 + 10q^3 + 6q^2 + 2q + 1)$ \\
\hline
$ (3, 1^5)$ & $(q - 1)^2 q^6(q + 1)^3(3q^2 + 1) $ \\
\hline
$ (2^4)$ & $(q - 1)^4 q^6(q + 1)^2 (3q^2 + q + 1)(q^2 + q + 1)$ \\
\hline
$ (2^3, 1^2)$ & $(q - 1)^3 q^3(q + 1)^3(6q^6 + 6q^5 + 10q^4 + 5q^3 + 5q^2 + q + 1) $ \\
\hline
$ (2^2, 1^4)$ & $(q - 1)^2 q(q + 1)(6q^8 + 11q^7 + 16q^6 + 15q^5 + 14q^4 + 8q^3 + 5q^2 + 2q + 1)$ \\
\hline
$ (2, 1^6)$ & $(q - 1)(q + 1)^2(3q^4 + 2q^2 + 1) $ \\
\hline
$ (1^8)$ & $1$ \\
\hline
\end{tabular}
    \caption{\small Calculation of $\F_{\mu\lambda}$ when $\lambda=(2,2,2,2)$}
    \label{Table2}
\end{table}
\end{remark}

 When $\lambda=(1^n)$, Yip~\cite{Yip}  obtained several interesting closed formulas of $\F_{\mu\lambda}(q)$, including for the case $
\mu=(k+1,1^{n-k-1})$. 
\cref{hook_formula} below generalizes Yip's results.  Consider an arbitrary partition $\lambda=(\lambda_1,\dots,\lambda_s)$ of $n$ and a hook partition $\mu=(k+1,1^{n-k-1})$ where $k\leq s-1$ (the latter condition is equivalent to $\lambda\less \mu'$). Recall that $[n]_{1/q}=q^{1-n}[n]_q$.
\begin{theorem}\label{hook_formula} Let $\lambda=(\lambda_1,\dots,\lambda_s)$ be a partition of $n$ and let $\mu=(k+1,1^{n-k-1})$ with $k\leq s-1$. Then
\begin{equation}
\F_{\mu\lambda}(q)=(q-1)^kq^{\binom{n}{2}-\binom{n-k}{2}-k}\sum_{j=2}^{s-k+1}[\lambda_1+\dots+\lambda_{j-1}]_{1/q}\sum_{\substack{A\subseteq \{2,\dots,s\}\\
|A|=k\\ \min A=j}}\,\prod_{i\in A}[\lambda_i]_{1/q}
\, .
\end{equation}
\end{theorem}

\begin{example} 
We illustrate some special cases of \cref{hook_formula}.
For $\lambda=(1^n)$, from~\cref{hook_formula} and the identity
$\binom{n}{k}+\binom{n-1}{k}+\cdots+\binom{k}{k}=\binom{n+1}{k+1}$, we obtain
\begin{equation}\label{Yip-Formula}
\begin{split}
\F_{\mu\lambda}(q)&=(q-1)^kq^{\binom{n}{2}-\binom{n-k}{2}-k}\sum_{j=2}^{n-k+1}[j-1]_{1/q}\binom{n-j}{k-1}\\
&=(q-1)^kq^{\binom{n-1}{2}-\binom{n-k-1}{2}}\sum_{j=1}^{n-k}\binom{n-j}{k}q^{-(j-1)}.
\end{split}
\end{equation}
When $\lambda=(2^l)\vdash n=2l$ with $k\leq l-1$, a similar argument to the one above gives the identity:
\begin{equation}
\begin{split}
\F_{\mu\lambda}(q)&=(q-1)^kq^{\binom{n}{2}-\binom{n-k}{2}-k}\sum_{j=2}^{l-k+1}[2(j-1)]_{1/q}\binom{l-j}{k-1}([2]_{1/q})^k\\
&=(q-1)^k(q+1)^kq^{\binom{n-1}{2}-\binom{n-k-1}{2}-k}\sum_{j=1}^{2(l-k)}\binom{l-\lceil j/2\rceil}{k}q^{-(j-1)}\, ,
\end{split}
\end{equation}
where $\lceil x \rceil$ is the ceiling function.  
\end{example}

\section{The division algorithm}
The goal of this section is to describe a
 parabolic extension of the \emph{division algorithm}, originally introduced in~\cites{Borodin,Kirillov} in the special case of $\lambda=(1^n)$,
which will be used in~\cref{sec:rec} to 
find a recursive formula for $\F_{\mu\lambda}$.


Throughout this section, 
we consider matrices  over an arbitrary field $\FF$, 
and we use
  $\mathcal C_\mu=\mathcal C_\mu(\FF)$ to denote the $\GL_n(\FF)$-conjugacy class of $J_\mu$. We abbreviate $\u_\Lambda(\FF)$ to $\u_\Lambda$.

\subsection{Properties of the rank sequence}\label{seq: Division algorithm}  
We fix an integer $n \geq 2$.

\begin{definition}[Eligible partition and admissible matrix]\label{eligible partition} Let $1 \leq m \leq n$, and let $\nu$ be a partition of $n-m$. A partition $\mu$ of $n$ is called $(\nu,m)$-eligible  if there exists  an $(n-m)\times m$ matrix $\Z$, such that
\begin{equation}\label{admisible-con}
\left(\begin{array}{c|c}
J_{\nu} & \Z\\ \hline
0_{m\times (n-m)} & 0_{m\times m}
\end{array}\right)\in \C_\mu. 
\end{equation}
The set of all $(\nu,m)$-eligible partitions will be denoted by $\HH_{m}(\nu)$. Moreover a matrix $\Z$ satisfying~\cref{admisible-con} will be called a $(\mu,\nu,m)$-admissible matrix. 
\end{definition}

\begin{definition}[Rank sequence]\label{rank-seq} Let $\Z$ be a  $(\mu,\nu,m)$-admissible matrix with
\begin{equation}\label{matrixA}
A:=\left(\begin{array}{c|c}
J_{\nu} & \Z\\ \hline
0_{m\times (n-m)} & 0_{m\times m}
\end{array}\right)\in \C_\mu. 
\end{equation}
The rank sequence of $\Z$ is defined by $\mathfrak{r}(\Z)=(\omega_j)_{j\geq 0}$ where
\begin{equation}\label{omega_i}
\omega_j=\omega_j(\mu,\nu,\Z)\defin
\begin{cases}
m & j=0,\\
\rk(A^j)-\rk(J_{\nu}^j) &  j\geq 1. 
\end{cases}
\end{equation}
\end{definition}

For any $j\geq 1$ we have 
\begin{equation}\label{A^j}
A^j=\left(\begin{array}{c|c}
J_{\nu}^j & J_{\nu}^{j-1}\Z\\ \hline
0_{m\times (n-m)} & 0_{m\times m}
\end{array}\right),
\end{equation}
and, following~\cite{Yip}, we can describe $J_{\nu}^{j-1}\Z$ combinatorially in terms of $\Z$ and $\nu$ as follows. Let $z_1,\dots,z_{n-m}$ be the rows of $\Z$ from top to bottom.  We put these vectors, from top right to bottom left, in the boxes of the Young diagram of $\nu$. To obtain the rows of  $J_{\nu}^{j-1}\Z$ we shift the entries in each row of $\nu$ to the right by $j-1$ boxes (the vectors 
that fall out are eliminated
and we write the zero vector in the boxes that become vacant).  
Henceforth the matrix whose rows are the ones that appear in the $j$'th column of $\nu$ after $j-1$ shifts will be denoted by $\Z_j$.
\begin{example}\label{Yip-Z} Let $\nu=(4,3,1)\vdash 8$ and consider the following diagrams filled with vectors:  
\medskip

\begin{center}
\begin{ytableau}
z_4 & z_3 & z_2 & z_1 \\
z_7 &z_6 & z_5  \\
z_8
\end{ytableau},\qquad
 \begin{ytableau}
0 & z_4 & z_3 & z_2 \\
0 &z_7 & z_6  \\
0
\end{ytableau},\qquad
\begin{ytableau}
0 & 0 & z_4 & z_3 \\
0 & 0 & z_7  \\
0
\end{ytableau}\qquad
\begin{ytableau}
0 & 0 & 0 & z_4 \\
0 & 0 & 0  \\
0
\end{ytableau}
\end{center}
\medskip
The first diagram says that the rows of $\Z$ are $z_1,\dots, z_8$, while the second diagram indicates that the rows of $J_\nu \Z$ are  $z_2,z_3,z_4,0,z_6,z_7,0,0$  and so  on. We also have 
$$
\Z_1=\begin{pmatrix}
z_4\\z_7\\z_8
\end{pmatrix},\qquad 
\Z_2=\begin{pmatrix}
z_4\\z_7
\end{pmatrix},\qquad
\Z_3=\begin{pmatrix}
z_4\\z_7
\end{pmatrix},\qquad
\Z_4=\begin{pmatrix}
z_4
\end{pmatrix}.\qquad   
$$
\end{example}

From nilpotence of $J_\nu$ and ~\cref{A^j} it follows that
\begin{equation}
\label{eqq:admi}\omega_j=\rk(A^j)-\rk(J_{\nu}^j)=\rk(\Z_j),\qquad j\geq 1.
\end{equation}
The following description of $\Z_j$ also emerges from~\cref{Yip-Z}. 
\begin{lemma}\label{Z_j-description} Let $\nu=(\nu_1,\dots,\nu_k)$ be a partition of $n-m$ with the conjugate $\nu'=(\nu'_1,\dots,\nu'_\ell)$ and let $\Z$ be a $(\mu,\nu,m)$-admissible matrix. Then for $1\leq j\leq \ell$ we have
\begin{equation}\label{Z_j-matrix}
    \Z_j=\begin{pmatrix}
    z_{\varepsilon(1)}\\ \vdots\\ z_{\varepsilon(\nu_j')}
\end{pmatrix}, \qquad\quad  \varepsilon(l)\defin\sum_{i\leq l}\nu_i
\end{equation}
\end{lemma}
\begin{definition}[Essential and free vectors]
\label{dfn-free-adm}
Let $\Z$ be a $(\mu,\nu,m)$-admissible matrix. We call the rows 
$\{z_{\varepsilon(1)},\ldots,
z_{\varepsilon(k)}\}$ of $\Z$ the \emph{essential vectors}, and its other rows the  \emph{free vectors}.
\end{definition}

\begin{remark}\label{new-entires}
    \cref{Z_j-description} in particular implies that $\Z_1$ has $k$ rows and $\Z_\ell$ has $\nu'_\ell$ rows. Moreover the number of new rows in $\Z_j$ that were not in $\Z_{j+1}$ is $\nu'_j-\nu'_{j+1}$.
\end{remark}

The proof of the following lemma 
from~\cite[Lemma 2]{Yip}
is straightforward. 

\begin{lemma}\label{rankJmu}  Let $\nu=(\nu_1,\dots,\nu_k)$ be a partition of $n$ with conjugate  $\nu'=(\nu'_1,\dots,\nu'_\ell)$. Then
\begin{equation}\label{Jordan-rank}
\rk\left(J_\nu^j\right)=
\begin{cases}
\sum_{i=j+1}^\ell \nu'_i  & 1\leq j\leq \ell-1,\\
0 & j\geq \ell.
\end{cases}
\end{equation}
\end{lemma} 
From this we deduce the following lemma. 
\begin{lemma}\label{lem:iffcondition} Let $\nu$ be a partition of $n-m$ with $\ell = \ell(\nu')=\nu_1$, and let $(\omega_j)_{j \geq 0}$ be the sequence defined in~\cref{omega_i}. The following statements hold:

\begin{enumerate}\setlength{\itemsep}{0.5em}
\item $\omega_{\ell+1}=0$.
\item $0\leq \cdots\leq \omega_{j+1}\leq \omega_j\leq \cdots\leq \omega_1\leq \omega_0=m$.
\item $0\leq \omega_j-\omega_{j+1}\leq \nu'_j-\nu'_{j+1}$ for $j\geq 1$.
\item $\mu_j'=\nu'_j+(\omega_{j-1}-\omega_{j})$ for $j\geq 1$.
\item $\sum_{j\geq 1}(\omega_{j-1}-\omega_j)=m.$
\end{enumerate} 
\end{lemma}
\begin{proof}
Note that $\ell=\nu_1$ and since
$J_{\nu}^\ell=J_{\nu}^{\nu_1}=0$, we obtain  $A^{\ell+1}=0$ which proves the first statement. For the second statement, note that  $\omega_j=\rk(\Z_j)$ and the set of rows of $\Z_{j+1}$ is a subset of the set of rows of $\Z_j$. Moreover, since $\Z_1$ has $m$ columns,  $\omega_1\leq m$.  For the third statement, observe that the matrix $\Z_j^T$ has $\nu'_j$ columns and thus its nullity is $\mathrm{Nul}_j:=\nu'_j-\omega_j$. Since the columns of $\Z_{j+1}^T$ form a subset of the columns of $\Z_j^T$, we have
$\mathrm{Nul}_j\geq \mathrm{Nul}_{j+1}$.
For the fourth one, note that by~\cref{rankJmu}
we have
$$
\rk(A^j)=\sum_{i\geq j+1}\mu'_j=n-\sum_{i\leq j}\mu'_i,\qquad \rk(J_{\nu}^j)=\sum_{j\geq j+1}\nu'_j=n-m-\sum_{i\leq j}\nu'_i. 
$$
Thus for $j\geq 1$ we have
$
\omega_j=m-\sum_{i\leq j}(\mu'_i-\nu'_i)$, or equivalently 
$\sum_{i\leq j}(\mu'_i-\nu'_i)=\omega_0-\omega_j$. Taking the difference of two such relations for successive values of $j$ yields the fourth statement. Finally, the last assertion follows from the first and the second ones. 
\end{proof}

\begin{lemma}\label{admissible-seq} Let $\mu$ and $\nu$ be partitions of $n$ and $n-m$, respectively. Set $\ell=\ell(\nu') = \nu_1$, and let $(\omega_j)_{j \geq 0}$ satisfy conditions (1)–(4) of~\cref{lem:iffcondition}. Then, there exist a $(\mu,\nu,m)$-admissible matrix $\Z$ whose rank sequence is $(\omega_j)_{j \geq 0}$.
\end{lemma}

\begin{proof}

We construct an $(n-m) \times m$ matrix $\Z$ with row vectors $\{z_1, \dots, z_{n-m}\}$ in $\FF^m$ such that $\rk(\Z_j) = \omega_j$ for $1 \leq j \leq \ell$, where $\Z_j$ is as in~\cref{Z_j-matrix}. 
 The value of $\rk(\Z_j)$ is determined by the essential vectors that occur in $\Z_j$. 

We construct the essential vectors inductively, starting with those that occur in $\Z_\ell$.  
By condition (3) of \cref{admissible-seq},  we have $\omega_\ell \leq \nu'_\ell$, and by~\cref{new-entires}, the number of rows in $\Z_\ell$ is $\nu'_\ell$. Thus, we can select $\nu'_\ell$ 
rows 
 in $\FF^m$ 
 for $\Z_\ell$
 such that $\rk(\Z_\ell) = \omega_\ell$. 
To construct $\Z_j$ from $\Z_{j+1}$ with $\rk(\Z_j) = \omega_j$, observe from~\cref{new-entires} that the number of new essential vectors in $\Z_j$ not appearing in $\Z_{j+1}$ is $\nu'_j - \nu'_{j+1}$, which by condition (4) of \cref{admissible-seq} is at least $\omega_j - \omega_{j+1}$. Hence, we can select new essential vectors in $\Z_j$, linearly independent of those in $\Z_{j+1}$, such that their span has dimension $\omega_j - \omega_{j+1}$ and so we obtain $\Z_j$ with $\rk(\Z_j)=\omega_j$. Now construct $\Z$ by choosing the free vectors  arbitrarily. This completes the construction of $\Z$. From~\cref{rankJmu} and the assumption $\mu'_j = \nu'_j + (\omega_{j-1} - \omega_j)$ for $j \geq 1$, we deduce that $A \in \C_\mu$.
\end{proof}

\begin{example}
Let $\mu = (4,4,2,1,1,1)$, $\nu = (4,3,1)$, and $\omega_0 = 5, \, \omega_1 = 2, \, \omega_2 = \omega_3 = 1, \, \omega_4 = 0$, with $n = 13$ and $m = 5$. The essential vectors are $\{z_4, z_7, z_8\}$ as in~\cref{Yip-Z}. In this example we can pick $z_4=0$ and two linearly independent vectors $z_7,z_8$. The rest of vectors are arbitrary.
\end{example}
We are now ready to state the main theorem of this section. 

\begin{theorem}[Division algorithm]\label{Div-alg} Let $1\leq m\leq n$ and let  $\nu$ be a partition of $n-m$. Then $\mu\in \HH_{m}(\nu)$ if and only if $\mu'/\nu'$ is a horizontal strip of size $m$.
\end{theorem}

\begin{proof}
Assume $\mu\in \HH_m(\nu)$. From parts (4) and (5) of~\cref{lem:iffcondition} it follows that  $\mu'$ is obtained by adding $m$ boxes into the Young diagram of $\nu'$. From parts (1) and (3)  of~\cref{lem:iffcondition} it follows that no two boxes are in the same row of $\mu$. Thus,  $\mu'/\nu'$ is a horizontal strip of size $m$. 
\smallskip

Conversely, assume that $\mu'/\nu'$ is a horizontal strip of size $m$ and let $\ell=\ell(\nu')$. Then $\ell(\mu')\leq \ell+1$ since otherwise $\mu'/\nu'$ cannot be a horizontal strip. Choose non-negative integers $\alpha_j$ for $1\leq j\leq \ell+1$ with $\mu'_j=\nu'_j+\alpha_j$ and $\alpha_1+\cdots+\alpha_{\ell+1}=m$. Since $\mu'/\nu'$ is a horizontal strip, we have $\alpha_{j+1}\leq \nu'_{j}-\nu_{j+1}'$. Define $\omega_0=m$ and $\omega_j=\omega_{j-1}-\alpha_j$ for $j\geq 1$. It is straightforward to  verify that the $(\omega_j)_{j\geq 0}$ satisfy conditions (1)–(5) of~\cref{lem:iffcondition}. Hence, by~\cref{admissible-seq} there exists a corresponding $(\mu,\nu,m)$-admissible matrix $\Z$,  which proves that $\mu\in \HH_m(\nu)$.
\end{proof}

\begin{example} Let $m=2$ and let $\nu=(3,2,2)$ with the conjugate $\nu'=(3,3,1)$. Then $\HH_m(\nu)$ has five elements and the following Young diagrams are the conjugates of $\mu\in\HH_m(\nu)$
\medskip

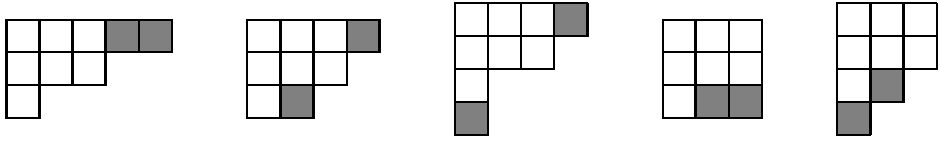
\begin{center}
\scriptsize
\begin{ytableau}
{} &  & & *(gray) & *(gray) \\
 & &   \\
  
\end{ytableau}\qquad\quad 
\begin{ytableau}
{} &  & & *(gray)  \\
 & &   \\
  &  *(gray)
\end{ytableau}\qquad\quad
\begin{ytableau}
{} &  & & *(gray)  \\
 & &   \\
  \\
 *(gray)
\end{ytableau}\qquad\quad
\begin{ytableau}
{} &  &  \\
 & &   \\
  & *(gray) & *(gray)
\end{ytableau}\qquad\quad
\begin{ytableau}
{} &  &  \\
 & &   \\
  &  *(gray)\\
  *(gray)
\end{ytableau}
\end{center}
\end{example}

\subsection{Non-emptiness of $\C_\mu\cap \u_\Lambda$} Given a composition $\Lambda=(\Lambda_1,\dots,\Lambda_s)$ on $n$, the partition of $n$ obtained by sorting  $\Lambda$ will be denoted by $\lambda=(\lambda_1,\dots,\lambda_s)$. 
We
define $\Lambda^*=(\Lambda_1,\dots,\Lambda_{s-1})$ and 
set $\lambda^*=\mathrm{sor
t}(\Lambda^*)$.
We can write any $X\in\u_\Lambda$ as
\begin{equation}\label{X_presentation}
X=\left(\begin{array}{c|c}
X_1 & Y\\ \hline
0_{\Lambda_s\times (n-\Lambda_s)} & 0_{\Lambda_s\times \Lambda_s}
\end{array}\right),
\end{equation}
where $X_1$ is  a square matrix of size $(n-\Lambda_s)$.
The main goal of this section is to prove the following assertion.
\begin{proposition}\label{Gale-Ryser-application} Let $\mu=(\mu_1,\dots,\mu_k)$ be a partition of $n$ and let $\Lambda=(\Lambda_1,\dots,\Lambda_s)$ be a composition of $n$ with $\lambda=\mathrm{sort}(\Lambda)$. Then $\C_\mu\cap \u_\Lambda$ is non-empty if and only if $\lambda\unlhd \mu'$. 
\end{proposition}

\begin{proof}
Suppose that $X\in \C_\mu\cap\u_\Lambda$.
Write $X$ as in~\cref{X_presentation} and let $J_\nu$ be the Jordan form of $X_1$.   Therefore, $\nu\vdash n-\Lambda_s$ and $X_1\in \C_\nu\cap \u_{\Lambda^*}$ where $\Lambda^*=(\Lambda_1,\dots,\Lambda_{s-1})$. Now pick $g\in \GL_{n-\Lambda_s}(\FF)$ such that $gX_1g^{-1}=J_{\nu}$ and set 
$\tilde{g}:=\mathrm{diag}(g,\mathrm{I}_{\Lambda_s})$. Then 
$\tilde gX\tilde g^{-1}\in \C_\mu$
and  thus $\mu\in \HH_{\Lambda_s}(\nu)$. \cref{Div-alg} implies that $\mu'/\nu'$ is a horizontal $\Lambda_s$-strip. By an iterative application of this process, we obtain a sequence of partitions $\emptyset=\mu^0\subseteq\cdots \subseteq\mu^{s-2}\subseteq \mu^{s-1}\subseteq \mu^{s}=\mu$ such that $(\mu^j/\mu^{j-1})'$ is a horizontal $\Lambda_j$-strip. Equivalently, we obtain a semistandard Young tableau of shape $\mu'$ and content $\Lambda$. Thus $\K_{\mu'\Lambda}=\K_{\mu'\lambda}\neq 0$ which implies $\lambda\less\mu'$.

Conversely, assume $\lambda\less\mu'$. Then there exists a semistandard Young tableau of shape $\mu'$ and content $\Lambda$ 
(see~\cite[Exercise A.11, p. 457]{Fulton-Harris}). Therefore there exists a sequence of partitions $\emptyset=\mu^0\subseteq\cdots \subseteq\mu^{s-2}\subseteq \mu^{s-1}\subseteq \mu^{s}=\mu$ such that $(\mu^j/\mu^{j-1})'$ is a horizontal $\Lambda_j$-strip. Starting from a $\Lambda_1\times \Lambda_1$ zero matrix and successively using~\cref{Div-alg} for $(\mu^j/\mu^{j-1})'$, we can enlarge the matrix inductively to construct an element of $\C_\mu\cap\u_\Lambda$.
\end{proof}

Let $X\in \C_\mu\cap \mathfrak u_\Lambda$. We can write $X$ as
in~\cref{X_presentation}.
 As before, let $J_\nu$ be the Jordan form of $X_1$, where $\nu\vdash (n-\Lambda_s)$ and $X_1\in \C_{\nu}\cap \u_{\Lambda^*}$ (in particular,  $\lambda^*\less\nu'$).  If $g_{X_1}^{}\in \GL_{n-\Lambda_s}(\FF)$ is chosen such that $gXg^{-1}=J_\nu$, then
\begin{equation}
\tilde gX\tilde g^{-1}
=
\left(\begin{array}{c|c}
J_{\nu} & gY\\ \hline
0_{\Lambda_s\times (n-\Lambda_s)} & 0_{\Lambda_s\times\Lambda_s}
\end{array}\right)
\in \C_\mu,  \qquad\text{ for } \tilde{g}:=\mathrm{diag}(g_{X_1},\mathrm{I}_{\Lambda_s}).
\end{equation}

Therefore, $\mu\in \HH_{\Lambda_s}(\nu)$ and $g_{X_1}^{}Y$ is a $(\mu,\nu,\Lambda_s)$-admissible matrix. Thus, we
obtain a map
\begin{equation}\label{pi}
\pi: \C_\mu\cap \u_\Lambda \longrightarrow  \bigsqcup_{\substack{\nu\,\vdash\, n-\Lambda_s\\[0.06cm]\lambda^*\less\, \nu' \\[0.06cm] \mu\in \HH_{\Lambda_s}(\nu)}} (\C_\nu\cap \u_{\Lambda^*}),\quad\pi(X)=X_1.
\end{equation}
\begin{lemma}\label{recursive-main_lemma} The map $\pi$ defined in~\cref{pi} is surjective. Moreover for any $X_1\in \C_\nu \cap \u_{\Lambda^*}$, there is a bijective correspondence between  the fiber $\pi^{-1}(X_1)$ and the set of $(\mu,\nu,\Lambda_s)$-admissible matrices, given by $Y\mapsto g_{X_1}^{}Y$.  
\end{lemma}
\begin{proof} We first show that $\pi$ is surjective. Let $X_1\in \C_\nu\cap\u_{\Lambda^*}$. Since $\mu\in \HH_{\Lambda_s}(\nu)$, there exists also a $(\mu,\nu,\Lambda_s)$-admissible matrix $\Z$ satisfying
\begin{equation}\label{Z-admis}
\left(\begin{array}{c|c}
J_{\nu} & \Z\\ \hline
0_{\Lambda_s\times (n-\Lambda_s)} & 0_{\Lambda_s\times \Lambda_s}
\end{array}\right)\in \C_\mu. 
\end{equation}
Now define 
$$
X:=\left(\begin{array}{c|c}
X_1 & g_{X_1}^{-1}\,\Z\\ \hline
0_{\Lambda_s\times (n-\Lambda_s)} & 0_{\Lambda_s\times \Lambda_s}
\end{array}\right)\in \u_\Lambda.
$$
From~\cref{Z-admis} we have
\begin{equation}
\tilde gX\tilde g^{-1}
=\left(\begin{array}{c|c}
J_\nu & \Z\\ \hline
0_{\Lambda_s\times (n-\Lambda_s)} & 0_{\Lambda_s\times \Lambda_s}
\end{array}\right)\in \C_\mu,  \qquad \tilde{g}:=\mathrm{diag}(g_{X_1}^{},\mathrm{I}_{\Lambda_s}).
\end{equation}
This shows $X\in \C_\mu\cap \u_\Lambda$ and $\pi(X)=X_1$. From this argument we  deduce that the map
\begin{equation*}
\begin{split}
\{(\mu,\nu,\Lambda_s)\text{-admissible matrices}\} & \rightarrow \pi^{-1}(X_1)\\
\Z & \longmapsto \left(\begin{array}{c|c}
X_1 & g_{X_1}^{-1}\,\Z\\ \hline
0_{\Lambda_s\times (n-\Lambda_s)} & 0_{\Lambda_s\times \Lambda_s}
\end{array}\right)
\end{split}
\end{equation*}
is a bijection, which proves the second part of the lemma. 
\end{proof}

\section{A recursion for $\F_{\mu\Lambda}(q)$}
\label{sec:rec}
From now on we assume that $\FF=\FF_q$. 
The goal of this section is to provide a recursive formula for $\F_{\mu\Lambda}(q)$. Because of~\cref{Gale-Ryser-application} we will assume $\lambda\less \mu'$, where $\lambda=\mathrm{sort}(\Lambda)$. The following theorem is an immediate consequence of 
\cref{recursive-main_lemma}. 
\begin{theorem}\label{recursive_theortem} Let $\Lambda = (\Lambda_1, \dots, \Lambda_s)$ be a composition of $n$ and set $\lambda = \mathrm{sort}(\Lambda)$. Additionally, let $\Lambda^* = (\Lambda_1, \dots, \Lambda_{s-1})$ and set $\lambda^* = \mathrm{sort}(\Lambda^*)$. For any partition $\mu$ of $n$ with $\lambda \less \mu'$, we have
\begin{equation}\label{F_recursive}
\F_{\mu\Lambda}(q)=\sum_{\substack{\nu\, \vdash\, n-\Lambda_s\\[0.06cm] \lambda^*\less\, \nu' \\[0.06cm] \mu\in \HH_{\Lambda_s}(\nu)}} \G^*(\mu,\nu,\Lambda_s)\, \F_{\nu\Lambda^*}(q),
\end{equation}  
where $\G^*(\mu,\nu,\Lambda_s)$ is the number of $(\mu,\nu,\Lambda_s)$ admissible matrices. 
\end{theorem}

We can also introduce bottom-up notions of eligible partitions and admissible matrices as in~\cref{recursive-main_lemma}, but working with  
matrices of the form
$$
\left(\begin{array}{c|c}
0_{\Lambda_1\times\Lambda_1} & \Z\\ \hline
0_{(n-\Lambda_1)\times\Lambda_1} & J_\nu
\end{array}\right).
$$
In this fashion, we obtain the following variant of~\cref{recursive_theortem}: suppose that 
$\Lambda_*=(\Lambda_2,\dots,\Lambda_s)$ and $\lambda_*=\mathrm{sort}(\Lambda_*)$.
Then
\begin{equation}\label{F_recursive_modify}
\F_{\mu\Lambda}(q)=\sum_{\substack{\nu\, \vdash\, n-\Lambda_1\\[0.06cm] \lambda_*\less\, \nu' \\[0.06cm] \mu\in \HH_{\Lambda_1}(\nu)}} \G_*(\mu,\nu,\Lambda_1)\, \F_{\nu\Lambda_*}(q),
\end{equation}
where 
$\G_*(\mu,\nu,\Lambda_1)$ is the number of bottom-up $(\mu,\nu,\Lambda_1)$-admissible matrices. 
It is straightforward to verify that 
$\G_*(\mu,\nu,|\mu|-|\nu|)=\G^*(\mu,\nu,|\mu|-|\nu|)$.

\subsection{Symmetry of $\F_{\mu\Lambda}(q)$}\label{Symmetry} As before we assume $\Lambda=(\Lambda_1,\dots,\Lambda_s)$ is a composition of $n\geq 2$ with $\lambda=\mathrm{sort}(\Lambda)$. The goal of this section is to prove the result below. 
\begin{proposition}\label{symmetry_prop} Let $\mu$ be a partition of $n$. Then $\F_{\mu\Lambda}(q)=\F_{\mu\lambda}(q)$.
\end{proposition}
\begin{proof}
    We prove this by induction on $s=\ell(\Lambda)$. When $s = 1$, the statement is trivial by~\cref{Lambda-length}. When $s=2$, let $\Lambda = (\Lambda_1, \Lambda_2)$ be a composition of $n$ with $\Lambda_1 \leq \Lambda_2$ so that $\lambda = (\Lambda_2, \Lambda_1)$. Let $X \in \u_\Lambda$ be a matrix of rank $k$. Since $X^2 = 0$, the Jordan type of such a matrix is $\mu = (2^k, 1^{n-2k})$ with $\mu' = (n-k, k)$. For this partition, $\F_{\mu\Lambda}(q)$ is the number of $\Lambda_1 \times \Lambda_2$ matrices of rank $k$, which is clearly the same as the number of $\Lambda_2 \times \Lambda_1$ matrices of rank $k$.
Therefore 
\begin{equation}\label{eq_Sym_s=2}
    \F_{\mu\Lambda}(q)=\F_{\mu\lambda}(q),\qquad \Lambda=(\Lambda_1,\Lambda_2). 
\end{equation}
Now assume $s\geq 3$. Define $\Lambda^\#=(\lambda^*,\Lambda_s)$ and $\Lambda_\#=(\Lambda_1,\lambda_*)$ where $\lambda^*$ and $\lambda_*$ are defined in~\cref{F_recursive,F_recursive_modify}.  From~\cref{F_recursive,F_recursive_modify} and the induction hypothesis we deduce 
\begin{equation}
    \F_{\mu\Lambda}(q)=\F_{\mu\Lambda^\#}(q)=\F_{\mu\Lambda_\#}(q).
\end{equation}
But we have $\lambda=((\Lambda_\#)^\#)_\#$, hence $\F_{\mu\Lambda}(q)=\F_{\mu\lambda}(q)$. 
\end{proof}

\subsection{Counting admissible matrices}
For $0\leq r\leq l$, set 
\begin{equation}\label{Cnr}
    \mathrm{C}_l^r\defin \left(q^l-1\right)\left(q^l-q\right)\cdots\left(q^l-q^{r-1}\right)=q^{\binom{r}{2}}(q-1)^r\,[r]_q!\, {l\brack r}_q. 
\end{equation}

\begin{proposition}\label{addmissible_number} Let $1 \leq m \leq n$, and let $\nu$ be a partition of $n - m$ with $\ell = \ell(\nu')$. Suppose $\mu \in \HH_m(\nu)$. The total number of $(\mu, \nu, m)$-admissible matrices is given by
\begin{equation}
\G^*(\mu,\nu,m)=q^{m(n-m-\nu_1')}\prod_{j=1}^{\ell}q^{\omega_{j+1}(\nu'_j-\nu'_{j+1})}{m-\omega_{j+1}\brack \omega_{j}-\omega_{j+1}}_q\,\mathrm{C}^{\omega_j-\omega_{j+1}}_{\nu'_j-\nu'_{j+1}}\, ,
\end{equation}
where $\omega_0=m$, $\omega_{\ell+1}=0$ and $\mu'_{j}=\nu_j'+(\omega_{j-1}-\omega_j)$ for $1\leq j\leq \ell+1$. 
\end{proposition}
\begin{proof}
By \cref{lem:iffcondition} an $(n-m)\times n$ matrix $\Z$ is $(\mu,\nu,m)$-admissible if and only if \cref{eqq:admi} holds. 
There are $n-m-\nu_1'$ free vectors
in $\Z$ (see \cref{dfn-free-adm}), and there are $q^{m(n-m-\nu'_1)}$ ways to fill their entries. 
Thus, counting the number of  admissible matrices boils down to counting the number of (essential) vectors 
$z_{\varepsilon(1)},z_{\varepsilon(2)},\dots,z_{\varepsilon(\nu_1')}\in \FF_q^m$
such that for the corresponding $\Z$ we have  $\rk(\Z_j)=\omega_j$ for $1\leq j\leq \nu'_1$. 
The statement follows from setting $n_i=\nu_i'$ and $r_i=\omega_i$ 
in~\cref{rank_profile}. 
\end{proof}



\section{Proof of~\cref{mainthm:Fmula-comb}: first and second parts}\label{F-Combin:Section} Let $\lambda$ and $\mu$ be partitions of $n$ with $\lambda\less \mu'$. As before we set $s=\ell(\lambda)$ and $\lambda^*=(\lambda_1,\dots,\lambda_{s-1})$.  By~\cref{Gale-Ryser-application} the intersection $\C_\lambda\cap \u_\lambda$ is non-empty and from  
\cref{recursive_theortem} we obtain 
\begin{equation}\label{rec:F}
\F_{\mu\lambda}(q)=\sum_{\substack{\nu\, \vdash\, n-\lambda_s\\[0.06cm] \lambda^*\less\, \nu' \\[0.06cm] \mu\in \HH_{\lambda_s}(\nu)}} \G^*(\mu,\nu,\lambda_s)\, \F_{\nu\lambda^*}(q).
\end{equation}
Setting $m=\lambda_s$ and $n=|\mu|$ in~\cref{addmissible_number} we obtain
\begin{equation}\label{GC-eq}
\G^*(\mu,\nu,\lambda_s)=q^{\lambda_s(|\mu|-\lambda_s-\nu_1')}\prod_{j=1}^{\ell}q^{\omega_{j+1}(\nu'_j-\nu'_{j+1})}{\lambda_s-\omega_{j+1}\brack \omega_{j}-\omega_{j+1}}_q\,\mathrm{C}^{\omega_j-\omega_{j+1}}_{\nu'_j-\nu'_{j+1}},
\end{equation}
where $\ell = \ell(\nu')$ and $\mu'_j = \nu'_j + (\omega_{j-1} - \omega_j)$, for $1 \leq j \leq \ell+1$, with $\omega_0 = \lambda_s$ and $\nu_{\ell+1}'=\omega_{\ell+1} = 0$.
We now simplify $\G^*(\mu,\nu,\lambda_s)$ by relating it to $\theta_{\mu'/\nu'}(q)$, defined in~\cref{theta:eq}. 
\begin{lemma}\label{G:Simplification} Let $\nu$ be a partition of $n-\lambda_s$ with $\ell=\ell(\nu')$. Then for any $\mu\in\HH_{\lambda_s}(\nu)$ we have
$$
\G^*(\mu,\nu,\lambda_s)=q^{\alpha(\mu,\nu,\lambda_s)}\, (q-1)^{\beta(\mu,\nu,\lambda_s)}
\theta_{\mu'/\nu'}(q),
$$

where 
\begin{equation}\label{alpha-beta}
\begin{split}
\alpha(\mu,\nu,\lambda_s)&=\lambda_s(|\mu|-\lambda_s-\nu_1')+\sum_{j=1}^\ell \binom{\mu'_{j+1}-\nu'_{j+1}}{2}+\sum_{j=1}^\ell \omega_{j+1}(\nu'_j-\nu'_{j+1}),
\\
\beta(\mu,\nu,\lambda_s)&=\lambda_s-(\mu'_1-\nu_1').
\end{split}
\end{equation}
\end{lemma}
\begin{proof}
From the definition of $\mathrm{C}^r_l$ given in~\cref{Cnr} and using the fact that $\omega_j-\omega_{j+1}=\mu'_{j+1}-\nu'_{j+1}$, one obtains
\begin{equation}\label{C:equality}
    \begin{split}
        \mathrm{C}_{\nu'_j-\nu'_{j+1}}^{\omega_j-\omega_{j+1}}&=q^{\binom{\mu'_{j+1}-\nu'_{j+1}}{2}}\,(q-1)^{\mu'_{j+1}-\nu'_{j+1}} \cdot [\omega_j-\omega_{j+1}]_q! \cdot {\nu'_j-\nu'_{j+1} \brack \mu_{j+1}'-\nu'_{j+1}}_q.
    \end{split}
\end{equation}
Furthermore, using $\lambda_s-\omega_1=\mu'_1-\nu'_1$, we find
\begin{align*}
\prod_{j=1}^{\ell} {\lambda_s-\omega_{j+1} \brack \omega_j-\omega_{j+1}}_q\, [\omega_j-\omega_{j+1}]_q!\, {\nu'_j-\nu'_{j+1} \brack \mu_{j+1}'-\nu'_{j+1}}_q&=\frac{[\lambda_s]_q!}{[\lambda_s-\omega_1]_q!}\cdot\prod_{j=1}^\ell {\nu'_j-\nu'_{j+1} \brack \mu_{j+1}'-\nu'_{j+1}}_q
=\theta_{\mu'/\nu'}(q).
\end{align*}
Using this and~\cref{GC-eq,C:equality}, we obtain
\begin{equation}
\label{eq:G*final}
\begin{split}
&\G^*(\mu,\nu,\lambda_s)=q^{\lambda_s(|\mu|-\lambda_s-\nu_1')}\prod_{j=1}^{\ell}q^{\omega_{j+1}(\nu'_j-\nu'_{j+1})}{\lambda_s-\omega_{j+1}\brack \omega_{j}-\omega_{j+1}}_q\,\mathrm{C}^{\omega_j-\omega_{j+1}}_{\nu'_j-\nu'_{j+1}}
=
q^E\theta_{\mu'/\nu'}(q)
,
\end{split}
\end{equation}
where 
\begin{equation}\label{dif:E}
E\defin  q^{\lambda_s(|\mu|-\lambda_s-\nu_1')+\sum_{j=1}^\ell\binom{\mu'_{j+1}-\nu'_{j+1}}{2}+\sum_{j=1}^\ell \omega_{j+1}(\nu'_j-\nu'_{j+1})}\, (q-1)^{\sum_{j=1}^\ell\left(\mu'_{j+1}-\nu'_{j+1}\right)}.
\end{equation}
Observe that
$$
\sum_{i=1}^\ell(\mu'_{j+1}-\nu'_{j+1})=(n-\mu'_1)-(n-\lambda_s-\nu'_1)=\lambda_s-(\mu'_1-\nu'_1). 
$$
Combining this with~\cref{eq:G*final} proves the statement.  
\end{proof}
Recall that a polynomial $f(q)$ of degree $N$ is called palindromic if $q^Nf(1/q)=f(q)$.
\begin{lemma}\label{deg_gama:lemma}
    Let $\nu$ be a partition of $n-\lambda_s$ with $\ell=\ell(\nu')$. Then for any $\mu\in\HH_{\lambda_s}(\nu)$ the polynomial $\theta_{\mu'/\nu'}(q)$ is palindromic of degree
    \begin{equation*}
        \gamma(\mu,\nu,\lambda_s)\defin\binom{\lambda_s}{2}-\binom{\mu_1'-\nu_1'}{2}+\sum_{j=1}^\ell (\nu_j'-\mu_{j+1}')(\mu_{j+1}'-\nu_{j+1}').
   \end{equation*}
\end{lemma}
\begin{proof}
Since $[n]_q$ is palindromic, it follows immediately that $\theta_{\mu'/\nu'}(q)$ is palindromic. Next note that 
 $\deg_q([n]_q!)=\binom{n}{2}$ and also
\begin{equation}\label{degg_q}
    \deg_q\left( {n \brack k}_q\right)= \binom{n}{2}-\binom{k}{2}-\binom{n-k}{2}
=k(n-k).
\end{equation}    
The definition of $\theta_{\mu'/\nu'}(q)$ and~\cref{degg_q} prove the lemma. 
\end{proof}

\begin{lemma}\label{alph_simplify:lemma}
    Let $\nu$ be a partition of $n-\lambda_s$.
    Then for any $\mu\in\HH_{\lambda_s}(\nu)$ we have
    \begin{equation*}
        \alpha(\mu,\nu,\lambda_s)+\gamma(\mu,\nu,\lambda_s)=\lambda_s|\mu|-\lambda_s^2+\fn(\nu)-\fn(\mu)-\lambda_s+\ell(\mu)-\ell(\nu)+\binom{\lambda_s}{2}.
    \end{equation*}
\end{lemma}
\begin{proof}
    Recall that $\mu_{j+1}'-\nu_{j+1}'=\omega_j-\omega_{j+1}$ and $\omega_0=\lambda_s$. Thus, 
    $
    \omega_{j+1}=\lambda_s-\sum_{i=1}^{j+1}(\mu_i'-\nu_i')
    $.
    Set $\ell=\ell(\nu')$.
By applying summation by parts, and the fact that $\nu_{\ell+1} = 0$, we obtain
\begin{alignat}{1}\label{alpha_part1:eq}
    &\sum_{j=1}^{\ell}\omega_{j+1}(\nu_j'-\nu_{j+1}')=\lambda_s\nu_1'-\sum_{j=1}^\ell\sum_{i=1}^{j+1}(\mu_i'-\nu_i')(\nu_j'-\nu_{j+1}')\nonumber\\ 
    &=\lambda_s\nu_1'-\nu_1'(\mu_1'-\nu_1')-\sum_{j=1}^\ell \nu_j'(\mu_{j+1}'-\nu_{j+1}') \nonumber\\
    &=\lambda_s\nu_1'-\nu_1'(\mu_1'-\nu_1')
    -\sum_{j=1}^\ell (\nu_j'-\mu_{j+1}')(\mu_{j+1}'-\nu_{j+1}')-\sum_{j=1}^\ell \mu_{j+1}'(\mu_{j+1}'-\nu_{j+1}'). 
\end{alignat}
We now simplify $-\sum_{j=1}^\ell \mu_{j+1}'(\mu_{j+1}'-\nu_{j+1}')$ using the identity  
$$
-ab=\binom{b-a}{2}-\binom{b}{2}-\binom{a}{2}-a.
$$
Set $b=\mu_{j+1}'$ and $a=\mu_{j+1}'-\nu_{j+1}'$. Then 
\begin{alignat}{1}\label{alpha_part2:eq}
   & -\sum_{j=1}^\ell \mu_{j+1}'(\mu_{j+1}'-\nu_{j+1}')
    =\sum_{j=1}^\ell\binom{\nu_{j+1}'}{2}-\sum_{j=1}^\ell\binom{\mu_{j+1}'}{2}-\sum_{j=1}^\ell\binom{\mu_{j+1}'-\nu_{j+1}'}{2}-\sum_{j=1}^\ell (\mu_{j+1}'-\nu_{j+1}') \nonumber\\
    &=\fn(\nu)-\fn(\mu)-\binom{\nu'_1}{2}+\binom{\mu_1'}{2}-\lambda_s+\mu_1'-\nu_1'-\sum_{j=1}^\ell\binom{\mu_{j+1}'-\nu_{j+1}'}{2}.
\end{alignat}
From now on, the proof is a straightforward simplification. From
the definition of $\alpha(\mu,\nu,\lambda_s)$ in~\cref{alpha-beta},~\cref{deg_gama:lemma}, and~\cref{alpha_part1:eq,alpha_part2:eq} we obtain
\begin{equation*}
\begin{split}
\alpha(\mu,\nu,\lambda_s)+\gamma(\mu,\nu,\lambda_s)&=\lambda_s|\mu|-\lambda_s^2+\fn(\nu)-\fn(\mu)-\lambda_s+\binom{\lambda_s}{2}\\
&-\nu_1'(\mu_1'-\nu_1')-\binom{\nu'_1}{2}+\binom{\mu_1'}{2}-\binom{\mu_1'-\nu_1'}{2}+\mu_1'-\nu_1'.
\end{split}
\end{equation*}
But 
$$
-\nu_1'(\mu_1'-\nu_1')-\binom{\nu'_1}{2}+\binom{\mu_1'}{2}-\binom{\mu_1'-\nu_1'}{2}+\mu_1'-\nu_1'=\mu_1'-\nu_1'=\ell(\mu)-\ell(\nu),
$$
and this completes the proof of the lemma.  
\end{proof}

\begin{proposition}
\label{F_simpler:Corr}
\label{F_chain:prop} Let $\lambda$ and $\mu$ be two partitions of $n$ with $\lambda\less\mu'$. Set $s=\ell(\lambda)$. Furthermore, let $\mathcal{F'}:=\mathcal{F'}(\mu,\lambda)$ denote the set of all chains  \[
\emptyset=\mu^{0}\subseteq \mu^{1}\subseteq \cdots\subseteq \mu^{s}=\mu\] of partitions such that each $(\mu^{j}/\mu^{j-1})'$ is a horizontal $\lambda_j$-strip. Then $\F_{\mu\lambda}(q)$ is equal to
\begin{equation}\label{Fmula:comb2}
(q-1)^{n-\ell(\mu)}\sum_{\mathcal{F'}} q^{
\sum_{j=1}^s\alpha(\mu^j,\,\mu^{j-1},\,\lambda_j)
+
\gamma(\mu^j,\,\mu^{j-1},\,\lambda_j)
}\prod_{j=1}^s\theta_{(\mu^j/\mu^{j-1})'}(1/q).
\end{equation}
\end{proposition}
\begin{proof}
This follows from iterative substitution of the $\F_{\nu\lambda^*}$ and the explicit formula for $\G^*(\mu,\nu,\lambda_s)$ given in~\cref{G:Simplification}
in the recursive relation~\cref{rec:F}. The condition $\lambda^*\less\nu'$ can be dropped because it is equivalent to  $\F_{\nu\lambda^*}(q)\neq 0$. The condition $\mu\in\HH_{\lambda_s}(\nu)$ states that $\mu'/\nu'$ is a horizontal $\lambda_s$-strip. 
Recall that $$\beta(\mu^j,\mu^{j-1},\lambda_j)=\lambda_j-\left(\left(\mu^j\right)'_1-\left(\mu^{j-1}\right)'_1\right).$$
Therefore
$$
\sum_{j=1}^s\beta(\mu^j,\mu^{j-1},\lambda_j)=\sum_{j=1}^s \left(\lambda_j-\left(\left(\mu^j\right)'_1-\left(\mu^{j-1}\right)'_1\right)\right)=n-\left(\mu^{s}\right)'_1=n-\mu'_1=n-\ell(\mu). 
$$
Finally,~\cref{deg_gama:lemma} allows the substitution of $\theta_{(\mu^j/\mu^{j-1})'}(q)$ by
\[
q^{\gamma(\mu^j,\,\mu^{j-1},\,\lambda_j)}\theta_{(\mu^j/\mu^{j-1})'}(1/q),
\] which complets the proof.
\end{proof}

\begin{proposition}\label{Simplify:prop} Let $\emptyset=\mu^{0}\subseteq \mu^{1}\subseteq \cdots\subseteq \mu^{s}=\mu$ be a chain of partitions such that $(\mu^{j}/\mu^{j-1})'$ is a horizontal $\lambda_j$-strip. Then
$$
\sum_{j=1}^s \left(\alpha(\mu^j,\mu^{j-1},\lambda_j)+\gamma(\mu^j,\mu^{j-1},\lambda_j)\right)=\binom{n}{2}+\ell(\mu)-\fn(\mu)-n,
$$
where $\fn(\mu)$ is defined in~\cref{n_mu:dif}. 
\end{proposition}
\begin{proof}
By~\cref{alph_simplify:lemma} we have
\begin{align*}
\alpha(\mu^j,\mu^{j-1},\lambda_j)+\gamma(\mu^j,\mu^{j-1},\lambda_j)
&=
\lambda_j(\lambda_1+\cdots +\lambda_j)\\
&-\binom{\lambda_j}{2}-2\lambda_j+\fn(\mu^{j-1})-\fn(\mu^{j})+\ell(\mu^j)-\ell(\mu^{j-1}).
\end{align*}
But we have 
\[
\sum_{j=1}^s\lambda_j(\lambda_1+\cdots+\lambda_j)=
\frac{1}{2}
\left(\left(\sum_{j=1}^s\lambda_j\right)
^2+\sum_{j=1}^s\lambda_j^2\right)
=\binom{n+1}{2}+\fn(\lambda').
\]
It follows that 
\begin{align*}
\sum_{j=1}^s \big(\alpha(\mu^j,\mu^{j-1},\lambda_j)+\gamma(\mu^j,\mu^{j-1},\lambda_j)\big)
&=
\sum_{j=1}^s\lambda_j(\lambda_1+\cdots+\lambda_j)-\fn(\lambda')-2n
-\fn(\mu)+\ell(\mu)\\
&
=\binom{n}{2}-n+\ell(\mu)-\fn(\mu).
\qedhere\end{align*}
\end{proof}
Parts (1) and (2) of~\cref{mainthm:Fmula-comb} readily follow from~\cref{F_simpler:Corr} and~\cref{Simplify:prop}. For Part (2) note that 
$$
\deg_q\left( \F_{\mu\lambda}(q)\right)=n-\ell(\mu)+\binom{n}{2}+\ell(\mu)-\fn(\mu)-n=\binom{n}{2}-\fn(\mu).
$$
In general, for $\lambda$ and $\mu$ partitions of $n$ with $\lambda \less \mu$, we have $\# \mathcal{F}(\mu,\lambda) =\K_{\mu\lambda}>0$.
From this we conclude that the leading coefficient of $\F_{\mu\lambda}(q)$ is $\K_{\mu'\lambda}$.


\section{Macdonald polynomials and the proof of~\cref{Intro_them: Main_Macdonald}
}
The goal of this section is to prove~\cref{Intro_them: Main_Macdonald}. We use the same notation introduced in~\cref{theta:dif} and~\cref{psi:dif}.
As before, we work with fixed partitions $\lambda $ and $\mu$ of $n$ with $\lambda\less \mu'$. 
Note that $\ell(\mu') \leq \ell(\lambda)$ since we assume $\lambda \less \mu'$. Thanks to~\cref{monomila_expansion:Mac}, we have 
\begin{equation}\label{monomila_expansion_1:Mac}
    \P_{\mu'}(\x;q,0)=\sum_{\nu}a_{\mu'\nu}(q)m_\nu(\x),
\end{equation}
\begin{proof}[Proof of~\cref{Intro_them: Main_Macdonald} (relation of $\F_{\mu\lambda}(q)$ to  $\P_{\mu'}(\x,1/q,0)$]
By~\cref{bmuamu} and the identity 
$
[n]_{1/q}!=q^{-\binom{n}{2}}[n]_q!\, ,
$
 we deduce
\begin{equation}\label{ba:eq}
b_{\mu'\lambda}\left(1/q\right)=q^{\sum_{i\geq 1}\left(\binom{\mu'_i-\mu'_{i+1}}{2}-\binom{\lambda_i}{2}\right)}\,\left(\prod_{i\geq 1}\frac{[\lambda_i]_q!}{[\mu'_i-\mu'_{i+1}]_q!}\right)a_{\mu'\lambda}\left(1/q\right). 
\end{equation}
From~\cref{monomila_expansion_1:Mac} we have $a_{\mu'\lambda}\left(1/q\right)=[\x^\lambda]\P_{\mu'}(x;1/q,0)$ and so from~\cref{ba:eq} and the first part of~\cref{mainthm:Fmula-comb}  we obtain
\begin{equation}
\label{eq:Fmula-q}
\F_{\mu\lambda}(q)=(q-1)^{n-\ell(\mu)}q^{E'}\prod_{i\geq 1}\frac{[\lambda_i]_q!}{[\mu_i'-\mu_{i+1}']_q!}\, [\x^\lambda]\P_{\mu'}(\x;1/q,0),
\end{equation}
where $E'\defin {\sum_{i\geq 1}\binom{\mu'_i-\mu'_{i+1}}{2}+\binom{n}{2}+\ell(\mu)-\fn(\mu)-\fn(\lambda')-n}$. 
This finishes the proof of the claim relating $\F_{\mu\lambda}(q)$ to  $\P_{\mu'}(\x,1/q,0)$.
\end{proof}
In the rest of this section we prove
the formula in~\cref{Intro_them: Main_Macdonald} that relates $\F_{\mu\lambda}(q)$ to  $\Q_{\mu'}(\x,1/q,0)$.
 Recall that the $q$-Pochhammer symbol, also called the $q$-shifted factorial, is defined for each integer $n\geq 0$ by 
$$
(a;q)_n\defin (1-a)(1-aq)\cdots(1-aq^{n-1}),\qquad (a;q)_0\defin 1.
$$
Let $\langle\cdot,\cdot\rangle_{q,t}$ denote the $q,t$-scalar product on symmetric functions~\cite[Sec. VI.1, p. 306]{Macdonald}, so that 
$$
\langle p_\lambda,p_\mu\rangle_{q,t}=\delta_{\lambda\mu}z_{\lambda}\prod_{i=1}^{\ell(\lambda)}\frac{1-q^{\lambda_i}}{1-t^{\lambda_i}},
$$
where the $p_\lambda$ are the power sum symmetric functions and $z_\lambda$ is defined in~\cref{z_rho}. 
For any partition $\mu$ of $n$, we define 
$$
b_{\mu}(q,t)\defin \langle\P_\mu(\x;q,t),\P_\mu(\x;q,t)\rangle_{q,t}^{-1}\,.
$$
Then  $\Q_{\mu}(\x;q,t)$ is  defined by
\begin{equation}\label{QPrelation}
\Q_\mu(\x;q,t)=b_\mu(q,t)\P_\mu(\x;q,t).
\end{equation}
Obviously we have $\langle\P_\lambda,\Q_\mu\rangle_{q,t}=\delta_{\lambda\mu}$. For each square $s = (i,j)$ in the diagram of a partition $\mu$, we define its arm-length as  $a(s) := \mu_i - j$
and its leg-length as  $l(s) := \mu'_j - i$.
Note that $a(s)$ and $l(s)$ represent the numbers of squares in the diagram of $\mu$ to the east and to the south of the square $s$, respectively. One can show~\cite[p. 339, Eq. (6.19)]{Macdonald} that
\begin{equation}\label{bmu-dual}
b_\mu(q, t) = \prod_{s\in \mu}\frac{1-q^{a(s)}t^{l(s)+1}}{1-q^{a(s)+1}t^{l(s)}}\, .
\end{equation}
The explicit relation between $\Q_\mu(\x;1/q,0)$ and $\P_\mu(\x;1/q,0)$ is as follows.   
\begin{lemma}\label{b_mu(1/q)} Let $\mu$ be a partition of $n$. 
Then 
\begin{equation}
    \frac{(q-1)^{\mu_1}}{q^{\mu_1+\sum_{j\geq 1}\binom{\mu_j-\mu_{j+1}}{2}}}\Q_{\mu}(\x;1/q,0)=\prod_{j\geq 1}\frac{1}{[\mu_j-\mu_{j+1}]_q!}\P_{\mu}(\x;1/q,0).
\end{equation}
\end{lemma}
\begin{proof}
    From~\cref{bmu-dual} we infer that
\begin{equation}\label{bmu0-rel}
    b_{\mu}(q,0)=\prod_{\substack{s\in\mu\\ l(s)=0}}\frac{1}{1-q^{a(s)+1}}=\prod_{j\geq 1}\frac{1}{(q\, ;q)_{\mu_j-\mu_{j+1}}}.
\end{equation}
    This combined with
the identity     
$
(1/q\,;1/q)_n=\prod_{j=1}^n\left(1-\frac{1}{q^j}\right)=\frac{(q-1)^n}{q^{n(n+1)/2}}[n]_q!$ gives us
    $$
    b_{\mu}(1/q,0)=\prod_{j\geq 1}\frac{q^{\mu_j-\mu_{j+1}+\binom{\mu_j-\mu_{j+1}}{2}}}{(q-1)^{\mu_j-\mu_{j+1}}[\mu_j-\mu_{j+1}]_q!}=\frac{q^{\mu_1+\sum_{j\geq 1}\binom{\mu_j-\mu_{j+1}}{2}}}{(q-1)^{\mu_1}}\prod_{j\geq 1}\frac{1}{[\mu_j-\mu_{j+1}]_q!}\, . 
    $$
  The lemma follows from this and the definition of $\Q_\mu(\x;q,t)$. 
\end{proof}
\begin{proof}[Proof of~\cref{Intro_them: Main_Macdonald}:
relating $\F_{\mu\lambda}$ to  $\Q_{\mu'}(\x,1/q,0)$]  
From Equation~\cref{eq:Fmula-q} and then \cref{b_mu(1/q)} (after replacing $\mu$ with $\mu'$), we obtain: 
\begin{equation*}
\begin{split}
&\F_{\mu\lambda}(q)=(q-1)^{n-\ell(\mu)}q^{E'}\prod_{i\geq 1}\frac{[\lambda_i]_q!}{[\mu_i'-\mu_{i+1}']_q!}\, [\x^\lambda]\P_{\mu'}(\x;1/q,0) \\
&=(q-1)^{n-\ell(\mu)}q^{E'}\frac{\left(\prod_{j\geq 1}[\lambda_j]_q!\right)(q-1)^{\ell(\mu)}}{q^{\ell(\mu)+\sum_{j\geq 1}\binom{\mu_j'-\mu'_{j+1}}{2}}}[\x^\lambda]\Q_{\mu'}(\x;1/q,0)\\
&=(q-1)^n q^{\binom{n}{2} - \fn(\mu) - \fn(\lambda') - n} \left( \prod_{i \geq 1} [\lambda_i]_q! \right) [\x^\lambda] \Q_{\mu'}(\x; 1/q, 0).
\end{split}
\end{equation*}
This completes the proof. 
\end{proof}

\section{Proof of the third part of~\cref{mainthm:Fmula-comb}} We begin by using the first part of~\cref{mainthm:Fmula-comb} to give another proof of the following result due to Fuchs and Kirillov~\cite[Theorem 4.1]{Fuchs-Kirillov}.

\begin{theorem}\label{Fuchs_Kirillov} Take $\lambda=(1^n)$ and let $\mu$ be an arbitrary partition. Then there exists an integer polynomial $\R_{\mu\lambda}(q)$ with $\R_{\mu\lambda}(0)=1$ and $\R_{\mu\lambda}(1)\neq 0$ such that 
$$
\F_{\mu\lambda}(q)=(q-1)^{n-\ell(\mu)}q^{\binom{n}{2}-\binom{\ell(\mu)}{2}-\sum_{j\geq 1}\mu_j'\mu_{j+1}'}\, \R_{\mu\lambda}(q).
$$
\end{theorem}
\begin{proof} From the first part of~\cref{mainthm:Fmula-comb} we have
\begin{equation}\label{Eq:Finproof}
\F_{\mu\lambda}(q)=(q-1)^{n-\ell(\mu)}q^{\binom{n}{2}+\ell(\mu)-\fn(\mu)-n}\,\, b_{\mu'\lambda}(1/q). 
\end{equation}
Recall that $\mathrm{SYT}(\mu')$ is the set of standard tableaux of shape $\mu'$. Then $\ssyt(\mu',1^n)=\mathrm{SYT}(\mu')$. Moreover $s_q(T)=1$ for any standard tableau $T$, and so from~\cref{Ram-Schlosser} we have 
$$b_{\mu'\lambda}(q)=\sum_{T\in\mathrm{SYT(\mu')}}r_q(T).$$
There is only one standard tableau that maximizes the degree of $r_q(T)$ which is the standard tableau obtained by inserting $\{1,2,\dots,\mu_1'\}$ into the first row of $\mu'$ and so on. For instance when $\mu'=(6,3,2)$, the tableau is 
  \begin{align*}\scriptsize
   T=\ytableaushort{123456,789\none,{10}{11} \none}.
  \end{align*}
Then
$
r_q(T)=\prod_{j\geq 2}\prod_{i=0}^{\mu'_j-1}[\mu'_{j-1}-i]_q.
$
Therefore
\begin{equation*}
\begin{split}
\deg_q(r_q(T))&=\sum_{j\geq 2} \left(\mu'_j\mu_{j-1}'-\sum_{i=1}^{\mu_j'}i\right)=\sum_{j\geq 1}\mu_j'\mu_{j+1}'-\sum_{j\geq 2}\binom{\mu_j'}{2}-\sum_{j\geq 2}\mu_j'\\
&=\sum_{j\geq 1}\mu_j'\mu_{j+1}'+\binom{\mu_1'}{2}+\mu_1'-\sum_{j\geq 1}\binom{\mu_j'}{2}-\sum_{j\geq 1}\mu_j'\\
&=\sum_{j\geq 1}\mu_j'\mu'_{j+1}+\binom{\ell(\mu)}{2}+\ell(\mu)-\fn(\mu)-n.
\end{split}
\end{equation*} 
Thus we can define
$$
\R_{\mu\lambda}(q)\defin
q^{\sum_{j\geq 1}\mu_j'\mu'_{j+1}+\binom{\ell(\mu)}{2}+\ell(\mu)-\fn(\mu)-n}b_{\mu'\lambda}(1/q)\in \mathbb{Z}[q]
$$
From this and~\cref{Eq:Finproof} we have 
$$
\F_{\mu\lambda}(q)=(q-1)^{n-\ell(\mu)}q^{\binom{n}{2}-\binom{\ell(\mu)}{2}-\sum_{j\geq 1}\mu_j'\mu_{j+1}'}\, \R_{\mu\lambda}(q).$$
Note that $b_{\mu'\lambda}(q)$ is monic and so $\R_{\mu\lambda}(0) = 1$. Moreover, $\lim_{q \to 1} [n]_q = n$, and thus $b_{\mu'\lambda}(1) \neq 0$, which confirms $\R_{\mu\lambda}(1) \neq 0$.
This proves the theorem. 
\end{proof}
For an arbitrary partition $\lambda$, unlike in the previous case, $s_q(T)$ contributes and this approach becomes complicated. For that reason, we use~\cref{Intro_them: Main_Macdonald} to prove~\cref{Into:eq_Fuchs-Kirillov} for a general partition $\lambda=(\lambda_1,\dots,\lambda_s)$. 

\subsection{Kostka polynomials}
The modified Hall-Littlewood polynomials satisfy
\begin{align}
\label{eq:tildeH=omegaP}
  \tilde{\HH}_\mu(\x;q)=
\sum_{\eta} \tilde{\K}_{\eta\mu}(q)s_\eta(\x)
  =
   \omega \left(q^{\fn(\mu)}\P_{\mu'}(\x; 1/q,0)\right),
\end{align}
where the $\tilde{\K}_{\eta\mu}(q)$ are the modified Kostka-Foulkes polynomials
(related to the Kostka polynomials by 
$\tilde{\K}_{\eta'\mu}(q):=q^{\fn(\mu)}\K_{\eta'\mu}(1/q)$),
the $s_\eta(\x)$ are the Schur symmetric functions,  
the $P_\mu(\x;q,t)$ are the (monic) Macdonald symmetric functions and $\omega$ is the $\mathbb Q(q)$-linear extension of the involution on the ring of symmetric functions with rational coefficients.
By~\cite[p. 354, Eq. (8.11)]{Macdonald}
we have 
$$
J_{\mu'}(\x;q,t)=\sum_{\eta}\K_{\eta\mu'}(q,t)S_\eta(\x;t),
$$
where the $\K_{\eta\mu'}(q,t)$ are the $q,t$-Kostka polynomials, and the $S_\eta(\x,t)$ satisfy  $S_\eta(\x,0) = s_\eta(\x)$. When $t=0$ we have $J_{\mu'}(\x; q,0)=\P_{\mu'}(\x,q,0)$ and~\cite[p.354, Eq. (8.15)]{Macdonald} implies that
$
\K_{\eta\mu'}(q,0)=\K_{\eta'\mu}(q),
$
where the $\K_{\eta\mu}(q)$ are the Kostka polynomials~\cite[Chapter III, Section 6]{Macdonald}.
 Therefore 
\begin{equation}\label{Mac_Schur}
\P_{\mu'}(\x;q,0)=\sum_{\eta}\K_{\eta'\mu}(q)s_\eta(\x). 
\end{equation}
From $\K_{\eta\lambda}=[\x^\lambda]s_\eta(x)$ we obtain 
\begin{equation}\label{PKostka}
[\x^\lambda]\P_{\mu'}(\x;q,0)=\sum_{\lambda\less\,\eta\,\less\mu'}\,\K_{\eta\lambda}\K_{\eta'\mu}(q).
\end{equation}
Note that the exponent of $q$ in Equation~\cref{eq:Fmula-q}
 is $E'$, which can easily be shown to be equal to
\begin{equation}\label{ex_modify}
    \binom{n}{2}+\fn(\mu)-\fn(\lambda')-\sum_{j\geq 1}\mu_j'\mu_{j+1}'-\binom{\ell(\mu)}{2}.
\end{equation}
From~\cref{Intro_them: Main_Macdonald} and~\cref{PKostka,ex_modify} it follows that
\begin{equation}\label{F_Kostka1:Eq}
\F_{\mu\lambda}(q)=(q-1)^{n-\ell(\mu)}q^{E''}\prod_{i\geq 1}\frac{[\lambda_i]_q!}{[\mu_i'-\mu_{i+1}']_q!}\sum_{\lambda\less\,\eta\,\less\mu'}\K_{\eta\lambda}\K_{\eta'\mu}(1/q), 
\end{equation}
where $E''\defin{\binom{n}{2}+\fn(\mu)-\fn(\lambda')-\sum_{i\geq 1}\mu'_i\mu'_{i+1}-\binom{\ell(\mu)}{2}}$. 
The following lemma is well-known and follows directly from Karamata's inequality.  
\begin{lemma}\label{n_ineq}
Let $\lambda$ and $\mu$ be two partitions of $n$ with $\lambda\less \mu$. Then $\fn(\lambda')\leq \fn(\mu')$ and the equality holds if and only if $\lambda=\mu$.
\end{lemma}
We are now ready to prove~\cref{Into:eq_Fuchs-Kirillov}.
\begin{proof}[Proof of the third part of~\cref{mainthm:Fmula-comb}]
It is known~\cite[p. 242, Eq. (6.5)]{Macdonald} that the Kostka polynomial $\K_{\eta'\mu}(q)$ is an integer monic polynomial of degree $\fn(\mu)-\fn(\eta')$ and $\K_{\eta'\mu}(1)=\K_{\eta'\mu}> 0$ when $\eta\less\mu'$. Therefore 
$$
\tilde{\K}_{\eta'\mu}(q)=q^{\fn(\eta')}\,\tilde{\R}_{\eta'\mu}(q),
$$
where $\tilde{\R}_{\eta'\mu}(q)$ is an integer polynomial of degree $\fn(\mu)-\fn(\eta')$ with $\tilde{\R}_{\eta'\mu}(0)=1$ and $\tilde{\R}_{\eta'\mu}(1)=\K_{\eta'\mu}>0$. For two distinct partitions $\lambda$ and $\eta$ with $\lambda\less \eta$, from~\cref{n_ineq}  we have $\fn(\lambda')< \fn(\eta')$. Then we can rewrite~\cref{F_Kostka1:Eq} as 
\begin{equation*}
\F_{\mu\lambda}(q)=(q-1)^{n-\ell(\mu)}
q^{\bar E}\prod_{i\geq 1}\frac{[\lambda_i]_q!}{[\mu_i'-\mu_{i+1}']_q!}\sum_{\lambda\less\, \eta\,\less\mu'} q^{\fn(\eta')-\fn(\lambda')}\K_{\eta\lambda}\tilde{\R}_{\eta'\mu}(q),
\end{equation*}
where $\bar E\defin{\binom{n}{2}-\sum_{i\geq 1}\mu'_i\mu'_{i+1}-\binom{\ell(\mu)}{2}}$. Define 
$$
\R_{\mu\lambda}(q):=\prod_{i\geq 1}\frac{[\lambda_i]_q!}{[\mu_i'-\mu_{i+1}']_q!}\sum_{\lambda\less\, \eta\,\less\mu'} q^{\fn(\eta')-\fn(\lambda')}\K_{\eta\lambda}\tilde{\R}_{\eta'\mu}(q),
$$
so that
$$
\F_{\mu\lambda}(q)=(q-1)^{n-\ell(\mu)}q^{\binom{n}{2}-\sum_{j\geq 1}\mu_j'\mu_{j+1}'-\binom{\ell(\mu)}{2}}\, \R_{\mu\lambda}(q).
$$
Obviously $\R_{\mu\lambda}(0)=1$ and $\R_{\mu\lambda}(1)> 0$ since $\tilde{\R}_{\eta'\mu}(1)>0$ and $\K_{\eta\lambda}>0$ for each $\lambda\less \eta\less\mu'$. 
Thus $\R_{\mu\lambda}(q)$ is a rational function without a pole at $0$ or $1$. Since $\F_{\mu\lambda}(q) \in \ZZ[q]$, we obtain
$\R_{\mu\lambda}(q)\in \ZZ[q]$.
\end{proof}


\section{The modular law and the proofs of~\cref{Hessenberg_non-zero,Hess_theorem,thm:nilphess}
}
This section is devoted to the proofs of our results on orbital varieties of ad-nilpotent ideals, and their application to nilpotent Hessenberg varieties. 
\subsection{
Proof of~\cref{Hess_theorem}:  modular law and ad-nilpotent ideals}
\label{subsec:modular}
Throughout this section, for a given Hessenberg function $\hes$, we always set $\hes(0)=0$ by convention. We begin by recalling the modular law~\cite{Abreu-Nigro}. 
\begin{definition}\label{triple_cond} Let $\hes_0, \hes_1, \hes_2 \in \mathcal{H}_n$ be Hessenberg functions. We say $(\hes_0,\hes_1,\hes_2)$ is a compatible sequence whenever one of the following two conditions holds:

\begin{description}
\item[Case I] There exists $i \in [n-1]$ such that 
\begin{equation}\label{caseI_h}
\hes_1(i-1) < \hes_1(i) < \hes_1(i+1)\quad \text{and}\quad \hes_1(\hes_1(i)) = \hes_1(\hes_1(i)+1).
\end{equation}
Moreover $\hes_0$ and $\hes_2$ are obtained from $\hes_1$ via:
\begin{equation}\label{caseI_modular}
\hes_0(l)= 
\begin{cases}
\hes_1(l) &  l \neq i, \\
\hes_1(i) - 1 &  l = i,
\end{cases} \qquad
\hes_2(l) = 
\begin{cases}
\hes_1(l) & l \neq i, \\
\hes_1(i) + 1 & l = i.
\end{cases}
\end{equation}
\medskip

\item[Case II] There exists $i \in [n-1]$ such that 
\begin{equation}\label{caseII_h}
\hes_1(i+1) = \hes_1(i)+1\quad \text{and}\quad \hes_1^{-1}(i) = \emptyset.
\end{equation}
Moreover $\hes_0$ and $\hes_2$ are obtained from $\hes_1$ via:
\begin{equation}\label{caseII_modular}
\hes_0(l) = 
\begin{cases}
\hes_1(l) &  l \neq i+1, \\
\hes_1(i) &  l = i+1,
\end{cases} \qquad
\hes_2(l) = 
\begin{cases}
\hes_1(l) & l \neq i, \\
\hes_1(i)+1 &  l= i.
\end{cases}
\end{equation}
\end{description}
\end{definition}

\begin{remark}
In the first case, the condition $\hes_1(i-1) < \hes_1(i)$ guarantees that $\hes_0 \in \mathcal{H}_n$. In particular, when $i=1$, from~\cref{caseI_h}, we deduce that $\hes_1(1) > 1$. Moreover in both cases $\hes_1(i)\leq n-1$.
\end{remark}
\begin{example} The function $\hes_1=(2,3,4,6,6,6)$ for $i=3$ satisfies the first case. Here we have $\hes_0=(2,3,3,6,6,6)$ and $\hes_2=(2,3,5,6,6,6)$. Also the function $\hes_1=(1,2,4,5,6,7,7)$ for $i=3$ satisfies the second  case. For this function we obtain $\hes_0=(1,2,4,4,6,7,7)$ and $\hes_2=(1,2,5,5,6,7,7)$
\end{example}
\begin{definition}[$q$-modular law] Let $f:\mathcal{H}_n\to \RR$ be a real-valued function  and let $q\in\RR$. We say that $f$ satisfies the $q$-modular law if for any compatible sequence $(\hes_0,\hes_1,\hes_2)$ of Hessenberg functions in $\mathcal{H}_n$ we have 
\begin{equation}\label{q-modular}
    (q+1)f(\hes_1)=qf(\hes_0)+f(\hes_2).
\end{equation}    
\end{definition}
\begin{remark} The definition of the modular law in~\cite{Abreu-Nigro} is slightly different from what we give here. There, the requirement on  $f: \mathcal{H}_n \to V$ is  that $V$ is a $\mathbb{Q}(q)$-vector space for an indeterminate $q$. Here,  we consider $q$ to be a number.
However, the main results of~\cite{Abreu-Nigro} (in particular~\cite[Theorem 1.2]{Abreu-Nigro} and~\cite[Algorithm 2.8]{Abreu-Nigro}) are still valid for all but finitely many choices of $q$ (more precisely, when $q$ is not a root of finitely many polynomials).
\end{remark}
Now let $\mu$ be a partition of $n$. Our next goal is to show that the linear function  
\begin{equation}\label{f_mu_modular}
    f_\mu: \mathcal{H}_n\to \RR,\qquad \hes\mapsto q^{|h|}\F_{\mu\hes}(q),
\end{equation}
where $|\hes|=\sum \hes(l)$, 
satisfies the $q$-modular law. Let $(\hes_0,\hes_1,\hes_2)$ be a compatible sequence. Note that $|\hes_0|=|\hes_1|-1$ and $|\hes_2|=|\hes_1|+1$. Thus~\cref{f_mu_modular} would follow from 
\begin{equation}\label{F_mu_modular}
    (q+1)\F_{\mu\hes_1}(q)=\F_{\mu\hes_0}(q)+q\F_{\mu\hes_2}(q).
\end{equation}
The proof of~\cref{F_mu_modular} is given in~\cref{lem:modularl}.  To begin, we recall basic properties of elementary row and column operations. For $x\in \FF_q$ and $1\leq s,t\leq n$, denote $\EE_{st}(x)=I+\E_{st}(x)$, where  $\E_{st}(x)$ is the elementary $n\times n$ matrix with $x$ in the $(s,t)$-entry  and zeros in other entries. Given an $n \times n$ matrix, we immediately obtain:
\begin{equation}
\label{elementary_oper}
\mathrm{Row}_l(\EE_{st}(x)A) = \begin{cases} \mathrm{Row}_l(A) & l \neq s; \\ x \,\mathrm{Row}_t(A) + \mathrm{Row}_s(A) & l = s. \end{cases}
\end{equation}
and \[
\mathrm{Col}_l(A\EE_{st}(x)) = \begin{cases} \mathrm{Col}_l(A) & l \neq t; \\ x \,\mathrm{Col}_s(A) + \mathrm{Col}_t(A) & l = t. \end{cases}
\]
In the following, we use $\mathrm{Per}(s,t)$ to denote 
the $n\times n$ permutation matrix associated with the transposition $(s\,\, t)\in \mathrm{S}_n$.

\begin{lemma}
\label{lem:modularl}
 The function $f_\mu: \mathcal{H}_n\to \mathbb{R}$ defined in~\cref{f_mu_modular} satisfies the $q$-modular law. 
\end{lemma}
\begin{proof}
We consider two cases according to~\cref{triple_cond}. 
\medskip

{\noindent \bf Case I:} Set $j=\hes_1(i)$. Define 
$$
\mathcal{A}_0=\{(a_{st})\in \u_{\hes_0}(\FF_q)\cap \mathcal{C_\mu}: a_{ij}\neq 0\}\qquad
\mathcal{A}_1=\{(a_{st})\in \u_{\hes_1}(\FF_q)\cap \mathcal{C_\mu}: a_{i(j+1)}\neq 0\}.
$$
One observes that
$$
\F_{\mu\hes_0}(q)=\F_{\mu\hes_1}(q)+|\mathcal{A}_0|\qquad
\F_{\mu\hes_1}(q)=\F_{\mu\hes_2}(q)+|\mathcal{A}_1|. 
$$
Thus~\cref{F_mu_modular} follows from $q|\mathcal{A}_1|=|\mathcal{A}_0|$. To prove the latter equality, consider 
$\varphi: \mathcal{A}_0\to \mathcal{A}_1$ defined by the assignment
$$
A\mapsto \mathrm{Per}(j,j+1)\,\EE_{j(j+1)}\left(\frac{a_{i(j+1)}}{a_{ij}}\right)\,A\,\EE_{j(j+1)}\left(-\frac{a_{i(j+1)}}{a_{ij}}\right)\,\mathrm{Per}(j,j+1).
$$
We will show that $\varphi(A)\in\mathcal{A}_1$ with $A_{ij}=\varphi(A)_{i(j+1)}$. 
It follows from~\cref{elementary_oper} that the $(i,j+1)$ entry of $A\,\EE_{j(j+1)}\left(-{a_{i(j+1)}}/{a_{ij}}\right)$ is zero. The assumptions $\hes_1(i)<\hes_1(i+1)$ and $\hes_1(\hes_1(i))=\hes_1(\hes_1(i)+1)$ imply that $i+1\leq j$. From~\cref{elementary_oper} we deduce  $A_1=\EE_{j(j+1)}\left({a_{i(j+1)}}/{a_{ij}}\right)\,A\,\EE_{j(j+1)}\left(-{a_{i(j+1)}}/{a_{ij}}\right)$ belongs to $\mathcal{A}_0$ and its $(i,j+1)$ entry is zero. Since $\hes_1(i)<\hes_1(i+1)$ we have $a_{(i+1)(j+1)}=0$. Then $A_1\, \mathrm{Per}(j,j+1)$, which swaps the $j$th and $(j+1)$th column, belongs to $\u_{\hes_1}$ with a non-zero $(i,j+1)$ entry. Thus, the matrix $\mathrm{Per}(j,j+1)\, A_1\, \mathrm{Per}(j,j+1)$ belongs to $\u_{\hes_1}$ with a non-zero $(i,j+1)$ entry, because $i<j$. Since the Jordan form of $A_1$ is $J_\mu$, we conclude that $\varphi(A) = \mathrm{Per}(j,j+1)\, A_1\, \mathrm{Per}(j,j+1) \in \mathcal{A}_1$.

\smallskip

We now show surjectivity of $\varphi$. If $B\in \mathcal{A}_1$ with $b=b_{i(j+1)}\neq 0$, then one can check that for any $x\in\FF_q$ 
\begin{equation}\label{fiber_phi}
\EE_{j(j+1)}(-x/b)\,\mathrm{Per}(j,j+1)\,B\,\mathrm{Per}(j,j+1)\,\EE_{j(j+1)}(x/b)\in \mathcal{A}_0,
\end{equation}
with $(i,j)$ entry $b$ and $(i,j+1)$ entry $x$. This matrix maps to $B$, so $\varphi$ is surjective. Moreover, the matrices of the form~\cref{fiber_phi} are all distinct. This follows from the identity $$\mathrm{Per}(j,j+1)\,\EE_{j(j+1)}(x)\,\mathrm{Per}(j,j+1) = \EE_{(j+1)j}(x)$$
and the fact that $\EE_{(j+1)j}(x) B\, \EE_{(j+1)j}(-x) = B$, where $B \in \mathcal{A}_1$, implies $x = 0$. Thus the fiber of each element has cardinality $q$, and so $q|\mathcal{A}_1|=|\mathcal{A}_0|$.

\medskip

{\noindent \bf Case II:} This case is similar to the previous case. Set $j=\hes_1(i)$. Define 
$$
\mathcal{B}_0=\{(b_{st})\in \u_{h_0}\cap\mathcal{C}_\mu: b_{(i+1)(j+1)}\neq 0\},\qquad \mathcal{B}_1=\{(b_{st})\in \u_{h_1}\cap\mathcal{C}_\mu: b_{i(j+1)}\neq 0\}.
$$
Then we have 
$$
\F_{\mu\hes_0}(q)=\F_{\mu\hes_1}(q)+|\mathcal{B}_0|\qquad
\F_{\mu\hes_1}(q)=\F_{\mu\hes_2}(q)+|\mathcal{B}_1|. 
$$
This implies that~\cref{F_mu_modular} follows from $q|\mathcal{B}_1|=|\mathcal{B}_0|$. To prove this we now define
the map 
$\psi: \mathcal{B}_0\to \mathcal{B}_1
$
by the assignment  
$$
B\mapsto \mathrm{Per}(i,i+1)\,\EE_{i(i+1)}\left(\frac{-b_{i(j+1)}}{b_{(i+1)(j+1)}}\right)\,B\,\EE_{i(i+1)}\left(-\frac{b_{i(j+1)}}{b_{(i+1)(j+1)}}\right)\,\mathrm{Per}(i,i+1).
$$
From~\cref{caseII_h} we have $j>i$ which implies $\psi(B)\in \mathcal{B}_1$ and a similar argument as above shows the fiber of each element has cardinality $q$. 
\end{proof}
\begin{proposition}\label{F_mu_rational} Let $\hes\in\mathcal{H}_n$ for some $n\geq 1$. Then  $\F_{\mu\hes}(q)$ is a rational function of $q$.    
\end{proposition}
\begin{proof}
This follows from~\cite[Theorem 1.2, Algorithm 2.8]{Abreu-Nigro}, which also holds for real-valued functions that satisfy the $q$-modular law for generic $q$. Indeed~\cite[Algorithm 2.8]{Abreu-Nigro} expresses $\F_{\mu\hes}(q)$ as a linear combination of the $\F_{\mu\lambda}(q)$ with coefficients which are rational functions of $q$. 
But by~\cref{mainthm:Fmula-comb} the $\F_{\mu\lambda}(q)$ are polynomials in $q$.
\end{proof}

We recall~\cite[Cor.~3.2]{Abreu-Nigro}: if $f:\mathcal{H}_n\to\QQ(q)$ satisfies the $q$-modular law (with $q$ variable), then
\begin{equation}\label{Abreu-Nigro_cor}
    f(\hes)=\sum_{\lambda\, \vdash n} \frac{c_{\lambda\hes}(q)}{[\lambda]_q!}f(k_\lambda),
\end{equation} 
where $k_\lambda=k_{\lambda_1}\sqcup\cdots\sqcup k_{\lambda_l}$ with $l=\ell(\lambda)$. By applying~\cref{F_mu_rational,Abreu-Nigro_cor} we deduce 
\begin{equation}\label{fmuh:eq}
    \F_{\mu\hes}(q)=\frac{1}{q^{|\hes|}}\sum_{\lambda\vdash n}\frac{c_{\lambda\hes}(q)}{[\lambda]_q!}q^{|k_\lambda|}\F_{\mu\lambda}(q).
\end{equation}
One can show 
$$|k_\lambda|=(n^2+n)/2+\fn(\lambda')=n^2-\binom{n}{2}+\fn(\lambda')=n^2-\dim\u_\lambda.$$
This combined with~\cref{fmuh:eq} proves the the first equation of~\cref{Hess_theorem}. The rest of the theorem follows from~\cref{F_mulambda_G:thm}. 
Finally, to show that $\F_{\mu\hes}(q) \in \ZZ[q]$ note that we have $\F_{\mu\hes}(q) = \frac{g(q)}{h(q)}$ where $g(q), h(q) \in \ZZ[q]$ and $h(q)$ is monic. By long division  $f(q)=f_1(q)+(r(q)/h(q))$ for $f_1(q),r(q)\in\ZZ[q]$ satisfying $\deg(r(q))<\deg(h(q))$. Since $f(q)\in\ZZ$ for every prime $q$, it follows that $r(q)=0$ for sufficiently large primes $q$, hence $r(q)=0$.

\subsection{Proof of~\cref{Hessenberg_non-zero}: Non-emptiness of $\mathcal{C}_\mu\cap \u_\hes$} In this section $\FF$ denotes an arbitrary field,
$\mathcal C_\mu$ is the $\GL_n(\FF)$-conjugacy class of $J_\mu$, and $\u_\hes=\u_\hes(\FF)$.
  Recall that
$[n]:= \{ 1, \dots,n \}$. For positive integers $m$ and $n$, set $[ m:n ]=\{ m , \dots,n \}$. If $m>n$, we define $[m:n] =\emptyset$.  For $i<j$, we write $\E_{ij}$ for the elementary $n\times n$ matrix whose only nonzero entry is a $1$ in the $i$-th row and the $j$-th column.  If we denote the standard basis vectors of $\FF^n$ by $e_1, \dots, e_n$, then 
\begin{equation}\label{Elementary-matrix-action}
\E_{ij}(e_\ell)=
\begin{cases}
e_i & \ell=j;\\
0 & \text{otherwise.} 
\end{cases}
\end{equation}
If $J= \{ i_1, \dots, i_r \}
\subseteq \{1,\cdots, n\}$ with $i_1< \cdots < i_r$, we write $\E_J:=\sum_{ l =1}^{r-1} \E_{ i_l i_{l+1} }$. We set $\E_J=0$ when $J$ has only one element. Then from~\cref{Elementary-matrix-action} we have 
\[ \E_J ( e_{i_1})=0,\quad \E_J( e_{i_2})= e_{i_1},\quad  \dots\quad  \E_J( e_{i_r})= e_{i_{r-1}}. \]
From this we observe that $\E_J$ has the Jordan canonical form of type $(r, 1, \dots, 1)$. Similarly, if $\mathcal{J}= \{ J_1, \dots, J_k \}$ is a partition of $[n]$  with $|J_i|=\mu_i$, then the Jordan canonical form of $\E_{\mathcal{J}}:=\sum_{ i=1}^k \E_{J_i}$ is $\mu=(\mu_1, \dots, \mu_k)$.  
\smallskip

Now let $\hes \in \mathcal{H}_n$ be a Hessenberg function. We recall the definition of the partial order $\mathcal{P}_\hes$ on $[n]$ given by $i \prec_\hes j$ if and only if $\hes(i) < j$. Let $\lambda_\hes$ be the Greene-Kleitman shape of $\mathcal{P}_\hes$. 
A simple algorithm for determining $\lambda_{\hes}$  essentially goes back to an early influential paper of Gerstenhaber~\cite[p. 535]{MR136683} (also see Fenn and Sommers~\cite[Proposition~6.8]{MR4334164}). We now describe Gerstenhaber's algorithm.

    Given $\hes:[n]\to [n]$ define a sequence $I_1,I_2,\ldots$ of disjoint subsets of $[n]$, called \emph{characteristic sequences} as follows. The first element of $I_1$ is 1. For each $i\in I_1$,  the next element in $I_1$ is the least positive integer $j$ such that $i\prec_\hes j$ if such a $j$ exists, otherwise $i$ is the last element of $I_1$. Once $I_1,\ldots,I_{r-1}$ have been constructed, we declare the smallest element of $I_r$ to be the least positive integer in $[n]\setminus \bigcup_{i=1}^{r-1}I_i$. Successive elements in $I_r$ are defined as before: for each $i\in I_r$,  the next element in $I_r$ is the least positive integer $j\in [n]\setminus \bigcup_{i=1}^{r-1}I_i$ such that $i\prec_\hes j$ if such a $j$ exists, otherwise $i$ is the last element of $I_r$. It is clear that we must have $[n]=\bigcup_{i=1}^{s}I_i$ for some integer $s$. If $\lambda_i=|I_i|$, then the Greene-Kleitman shape of $\mathcal{P}_{\hes}$ is given by $\lambda_\hes:=(\lambda_1,\ldots,\lambda_s)$.
  \begin{example}
For the Hessenberg function $\hes=(3,4,6,7,8,8,8,8)$, Gerstenhaber's algorithm yields the characteristic sequences $I_1=\{1,4,8\}$, $I_2=\{2,5\}$, $I_3=\{3,7\}$, $I_4=\{6\}$. Therefore $\lambda_\hes=(3,2,2,1)$.
   \end{example}

\begin{proposition}\label{decomp_poset} Let $\mu\less \lambda=\lambda_\hes$ be a partition of $n$. Then $[n]$ can be partitioned into $\ell(\mu)$ chains in $\mathcal{P}_\hes$ whose cardinalities are the parts of $\mu$.
\end{proposition}
Before giving the proof of~\cref{decomp_poset}, we note that 
for a given partition $\mu\less\lambda_\hes$, by applying~\cref{decomp_poset} we obtain a family of disjoint subsets $\mathcal{J}=\left\{J_1,\dots,J_{\ell(\mu)}\right\}$ of $[n]$, where each $J_i$ is a chain with respect to $\preceq_\hes$ and $|J_i|=\mu_i$. Thus $\E_\mathcal{J}\in \C_\mu\cap\u_\hes$ which proves one direction of~\cref{Hessenberg_non-zero}. We require the following lemma to prove~\cref{decomp_poset} . 

\begin{lemma}\label{lem:exchange}
  If $C_1,C_2$ are disjoint chains in $\mathcal{P}_\hes$ with $|C_1|> |C_2|+1$, then there exist disjoint chains $C_1'$ and $C_2'$ in $\mathcal{P}_\hes$ such that $C_1\cup C_2=C_1'\cup C_2'$ with $|C_1'|=|C_1|-1$ and $|C_2'|=|C_2|+1$.
\end{lemma}
\begin{proof}
It is known~\cite[Proposition 4.1]{Shareshian-Wachs} that $\mathcal{P}_\hes$ is isomorphic to a natural unit interval order (a collection of $n$ closed unit intervals on the real line with intervals $U_1<U_2$ precisely when $U_1$ is completely to the left of $U_2$). Therefore we may assume that $C_1=\{U_1< U_2< \cdots <U_\ell\}$ and $C_2=\{V_1<V_2< \cdots<V_k\}$ for unit intervals $U_i$ and $V_j$ on the real line in the Euclidean plane with $1\leq i\leq \ell, 1\leq j\leq k$.
The proof is by induction on $k=|C_2|$. The cases $k=0,1$ are straightforward. Suppose the lemma holds whenever $|C_2|<k$. Now consider the case $|C_2|=k$.  If some interval $U$ in $C_1$ does not intersect any interval in $C_2$, then $C_1'=C_1-\{U\}$ and $C_2'=C_2\cup\{U\}$ have the desired property.

Now suppose no interval can be moved from $C_1$ to $C_2$. For $1\leq i\leq \ell-1$ let $L_i$ be a line perpendicular to the $x$-axis which intersects the real line strictly between $U_{i}$ and $U_{i+1}$. Since each interval in $C_2$ intersects at most one such line, it can be seen that some line $L_j$ does not intersect any interval in $C_2$. For $m=1,2$, write $C_m=C_m^L\cup C_m^R$ where $C_m^L$ denotes all intervals in $C_m$ to the left of $L_j$ and $C_m^R$ denotes all intervals in $C_m$ to the right of $L_j$. By the movability assumption, both $|C_2^L|$ and $|C_2^R|$ are nonempty. Suppose either $|C_1^L|> |C_2^L|+1$ or $|C_1^R|> |C_2^R|+1$. Since both $|C_2^L|$ and $|C_2^R|$ are less than $k$, we can use the inductive hypothesis to construct $C_1'$ and $C_2'$ as required. If not, then we must have $|C_1^L|= |C_2^L|+1$ and $|C_1^R|= |C_2^R|+1$. In this case it is clear that $C_1'=C_2^L\cup C_1^R$ and $C_2'=C_1^L\cup C_2^R$ have the desired property. This completes the inductive step and the proof.
\end{proof}
\begin{remark}
  The above lemma fails for arbitrary partial orders. For instance, if $P$ is isomorphic to the direct sum of a 3-element chain and a 1-element chain. 
\end{remark}

\begin{proof}[Proof of~\cref{decomp_poset}] 
  Write $\lambda=\lambda_\hes$ and suppose $C=\{C_i\}_{1\leq i\leq k}$ is a disjoint chain cover of $\mathcal{P}_\hes$ with $|C_i|=\lambda_i$ for $1\leq i\leq k$ (such a cover exists by Gerstenhaber's algorithm). Now suppose $\mu \trianglelefteq \lambda$ is a partition.  Then there is a sequence $\mu=\nu^0\trianglelefteq \nu^1\trianglelefteq \cdots \trianglelefteq \nu^m=\lambda_\hes$ of partitions such that $\nu^{i-1}$ is covered by $\nu^{i}$ in the dominance order for $1\leq i\leq m$~\cite[p. 9]{Macdonald}. Therefore, for each $i$, the partition $\nu^{i-1}$ can be obtained from $\nu^i$ by moving a single cell in some row of the Young diagram of $\nu^i$ to a lower row. If there exists a chain cover for $\mathcal{P}_\hes$ where the chains have cardinalities given by the parts of $\nu^{i}$, then by Lemma \ref{lem:exchange} we can construct another chain cover for $\mathcal{P}_\hes$ where the chains have cardinalities given by the parts of $\nu^{i-1}$. This proves~\cref{decomp_poset}. 
\end{proof}

We now prove the other direction of~\cref{Hessenberg_non-zero}. Let $\mathcal{P}$ be a finite poset with $n$ elements. Fix a linear extension of $\mathcal{P}$ which identifies its elements with $[n]$ (thus if $i\preceq_Pj$ then $i\leq j$ as positive integers). The incidence algebra $I(\mathcal{P},\FF)$ consists of all $n\times n$ matrices $A=(a_{ij})$ over $\FF$ such that $a_{ij}=0$ unless $i\preceq_\mathcal{P} j$. Note that $I(\mathcal{P},\FF)$ is an algebra of upper triangular matrices. 
A matrix $A\in I(\mathcal{P},\FF)$ is nilpotent precisely when its diagonal entries are all zero. Write ${\rm Nil}(\mathcal{P},\FF)$ for the collection of all nilpotent matrices in $I(\mathcal{P},\FF)$. Let $\delta_k(A)$ denote the greatest common divisor of all $k\times k$ minors of $xI-A$. Given a matrix $A\in {\rm Nil}(\mathcal{P},\FF)$, suppose $\delta_k(A)=x^{d_k}$ for $0\leq k\leq n$ (by convention $\delta_0(A)=1$, so $d_0=0$)~\cite[Lemma 6.2]{MR1814900}. Note that $d_n=n$. Then $A$ is similar to $J_\lambda$ where $\lambda=(\lambda_1,\lambda_2,\ldots)$ with $\lambda_i=d_{n+1-i}-d_{n-i}$ for $1\leq i\leq n$. Therefore $\lambda_1+\cdots+\lambda_r=n-d_{n-r}$ for $r\geq 1$. 
\begin{proposition}\label{thm:dominated}
 If ${\rm Nil}(\mathcal{P},\FF)\cap \C_\mu\neq \emptyset$, then $\mu\less \lambda_\mathcal{P}$. In particular, if $ \mathcal{C}_\mu\cap\u_\hes  \neq \emptyset$ then $\mu \less \lambda_\hes$.
\end{proposition}
\begin{proof}
  Let $x_{ij} \,(i,j\geq 1)$ be independent indeterminates over the field $\FF$. Consider the generic $n\times n$ matrix $G=(g_{ij})$ where
  $$
  g_{ij}=
  \begin{cases}
    x_{ij} & \mbox{if } i\prec_\mathcal{P} j,\\
    0 & \mbox{otherwise.}
  \end{cases}
  $$ 
By Britz and Fomin \cite[Theorem 6.1]{MR1814900} the matrix $G$ is similar to $J_\lambda$ where $\lambda=\lambda_\mathcal{P}$ (the proof in \cite{MR1814900} is given for the complex field but carries over verbatim for any field $\FF$). Suppose $\delta_i(G)=x^{g_i}$ for $i\geq 1$. Let $A\in {\rm Nil}(\mathcal{P},\FF)\cap \C_\mu$ and suppose $\delta_i(A)=x^{a_i}$ for $1\leq i\leq n$. Since $A$ can be obtained from $G$ by specializing some entries to elements of $\F$, it follows that $\delta_i(G)=x^{g_i}$ divides $\delta_i(A)$ for each $1\leq i\leq n$. In other words, $g_i\leq a_i$ for $1\leq i\leq n$. From $A\in \C_\mu$, we deduce $\mu_1+\cdots+\mu_j=n-a_{n-j}\leq n-g_{n-i}=\lambda_1+\dots+\lambda_j$ which shows $\mu\less \lambda=\lambda_{\mathcal{P}}$. 
\end{proof}

\subsection{Proof of~\cref{cor:n-ell}}

It is proved in~\cite[Theorem 6.3]{Shareshian-Wachs} that 
\begin{equation}\label{Schur_positive}
    X_{\graph(\hes)}(\x; q) = \sum_{\lambda \vdash n} g_{\lambda \hes}(q) s_\lambda(\x),\qquad g_{\lambda\hes}\in\ZZ_{\geq 0}[q].
\end{equation}
From~\cref{Hess_theorem} and~\cref{Schur_positive} we obtain
 $$
 \F_{\mu\hes}(q)=\frac{q^{n^2-|\hes|} (q-1)^n}{|\cen(I+J_\mu)|}\sum_{\lambda \vdash n}\tilde{\K}_{\lambda\mu}(q)\, g_{\lambda \hes}(q),
 $$   
 where $\tilde{\K}_{\lambda\mu}(q)\defin q^{\fn(\mu)}\K_{\lambda\mu}(1/q)$ is the modified Kostka polynomial. It is known from~\cite[p.~242, Eq.~(6.5)]{Macdonald} that the modified Kostka polynomial $\tilde{\K}_{\lambda\mu}(q)$ has nonnegative integer coefficients. Since $\mu \less \lambda_\hes$, we have $\F_{\mu\hes}(q) \neq 0$, and so, using the positivity of $g_{\lambda\hes}(q)$, equation~\cref{center_size2}, and~\cref{mainthm:Fmula-comb}, we conclude that the exact power of $q - 1$ in $\F_{\mu\hes}(q)$ is $n - \ell(\mu)$.

\subsection{Proof of~\cref{thm:nilphess}: nilpotent Hessenberg varieties}
 Let $\B$ be the standard upper triangular Borel subgroup of $\G = \GL_n(\FF_q)$. For $\hes\in \mathcal{H}_n$ and $X$ a nilpotent $n\times n$ matrix  with Jordan form $J_\mu$ we define 
$$
\mathcal{B}_\mu^\hes\defin \left\{g\B \in \G/\B: g^{-1}(I+X)g\in \U_\hes\right\}.
$$
Note that $\U_\hes$ is a normal subgroup of $\B$. One can easily see that
$$
\#\left\{g: g^{-1}(I+X)g\in \U_\hes\right\}=\mid \B\mid\cdot\mid\mathcal{B}_{\mu}^\hes\mid .
$$
Then from~\cref{ind-form,ind-form-char} 
and~\cref{Hess_theorem}
we deduce 
\begin{align*}
  \#\text{$\mathcal{H}$ess$_{nil}$}(\hes,X)&=|\mathcal{B}_{\mu}^\hes|=\frac{|Z_G(I+J_\mu)|}{|\B|}\F_{\mu\hes}(q)
  =q^{n^2-|\hes|-\binom{n}{2}}\langle \tilde{\HH}_\mu(\x;q),X_{\graph(\hes)}(\x;q)\rangle.
\end{align*}
 The result now follows from the fact that $E_\hes=\sum_i(\hes(i)-i).$

\section{Proof of~\cref{Hess_theorem_tableaux}}
Our proof of~\cref{Hess_theorem_tableaux} uses some results from~\cite{CarlssonMellit}, which we learned about from~\cite{BasuBhattacharya}. Throughout this section we denote $\P_\mu(\x;q,0)$ by $\W_{\mu}(\x;q)$ known as $q$-Whittaker function. Given a Hessenberg function $\hes:[n]\to [n]$, let $\boldsymbol\chi_\hes(\x;q)$ be  the associated \emph{unicellular LLT symmetric function}, defined by
\[
\boldsymbol{\chi}_\hes(\x;q)\defin\sum_{w\in\ZZ_{\geq 1}^n}
q^{\mathrm{inv}(\hes,w)}\prod_{i=1}^nx_{w_i},
\]
where for $w=(w_1,\ldots,w_n)\in \ZZ_{\geq 1}^n$ we have
\[
\mathrm{inv}(\hes,w)\defin
\#\{(i,j)\,:\,i<j\leq \hes(i)\text{ and }w_i>w_j\}.
\]
The $\boldsymbol\chi_\hes(\x;q)$ are symmetric functions (see for example~\cite[Proposition 3.2]{CarlssonMellit}). By~\cite[Proposition 3.5]{CarlssonMellit}, they are related to
the chromatic quasisymmetric functions by a plethystic substitution, i.e., 
\begin{equation}
\label{eq:boldchipleth}
\boldsymbol\chi_\hes(\x;q)=(q-1)^nX_{\graph(\hes)}\left[\frac{\x}{q-1};q\right].
\end{equation}
As verified in~\cite[Proposition 5.1]{BasuBhattacharya}, from~\cite[Theorem 4.1]{GriffinMellitetal} it follows that
\[
\boldsymbol\chi_\hes(\x;q)=
\sum_{\mu\vdash n}
(1-q)^{n-\mu_1}\underline{\F}_{\mu \hes}(q)
\W_{\mu}(\x;q).
\] 
Thus, from~\cref{eq:boldchipleth} we obtain
\begin{align}
\label{eq:trickysign}
X_{\graph(\hes)}(\x;q)&
=(q-1)^{-n}\boldsymbol\chi_\hes[(q-1)\x;q]
=(-1)^n
\sum_{\mu\vdash n}
(1-q)^{-\mu_1}\underline{\F}_{\mu\hes}(q)
\W_{\mu}[(q-1)\x;q].
\end{align}
We now compute the dual of $\W_\mu[(q-1)\x]$ with respect to the Hall inner product.
\begin{lemma}\label{dual of whittaker} The dual of $\W_\mu[(q-1)\x;q]=\P_\mu[(q-1)\x;q,0]$ with respect to the Hall inner product is 
$$
\frac{(-1)^n(1-q)^{-\mu_1}}{\prod_{j\geq 1}[\mu_j-\mu_{j+1}]_q!}q^{\fn(\mu')}\tilde{\HH}_{\mu'}(\x;1/q).
$$
\end{lemma}
\begin{proof} We recall two basic facts about plethystic substitution (see~\cite[Section 2]{Haiman}). 
If $f$ is homogeneous of degree $n$, then $f[-\x] = (-1)^n \omega(f(\x))$. For symmetric functions $f,g \in \mathbb{Q}(q,t)$,
$$
\langle f(\x), g(\x) \rangle_{q,t}
=
\left\langle f[\x], g\left[\frac{1-q}{1-t}\x\right] \right\rangle.
$$
These imply
\begin{equation*}
\begin{split}
\delta_{\lambda\mu}&=\langle\Q_{\lambda}(\x;q,0),\P_\mu(\x;q,0)\rangle_{q,0}=\langle \Q[\x;q,0], \P_\mu[(1-q)\x;q,0]\rangle\\
&=\langle \Q[-\x;q,0], \P_\mu[(q-1)\x;q,0]\rangle
\end{split}
\end{equation*}
Thus the dual of $\W_\mu[(q-1)\x;q]$ is $\Q_\mu[-\x;q,0]$. From~\cref{QPrelation,bmu0-rel,eq:tildeH=omegaP} we obtain:
\begin{equation*}
\begin{split}
\Q_\mu[-\x;q,0]&=\frac{(1-q)^{-\mu_1}}{\prod_{j\geq 1}[\mu_j-\mu_{j+1}]_q!}\P_\mu[-x;q,0]=\frac{(-1)^n(1-q)^{-\mu_1}}{\prod_{j\geq 1}[\mu_j-\mu_{j+1}]_q!}\omega(\P_\mu(\x;q,0))\\[0.2cm]
&=\frac{(-1)^n(1-q)^{-\mu_1}}{\prod_{j\geq 1}[\mu_j-\mu_{j+1}]_q!}q^{\fn(\mu')}\tilde{\HH}_{\mu'}(\x;1/q)
\qedhere\end{split}
\end{equation*}
\end{proof}
Now we prove~\cref{Hess_theorem_tableaux} using the following fact~\cite[Proposition 2.6]{Shareshian-Wachs}:
\begin{equation}\label{eq:XG}
X_{\graph(\hes)}(\x;q)=q^{E_\hes}X_{\graph(\hes)}(\x;1/q)=q^{|\hes|-\binom{n}{2}-n}X_{\graph(\hes)}(\x;1/q).
\end{equation}
\begin{proof}[Proof of~\cref{Hess_theorem_tableaux}] By applying the orthogonality relation in~\cref{dual of whittaker} together with~\cref{eq:trickysign,eq:XG}, we obtain:
\begin{equation}
\begin{split}
\underline{\F}_{\mu\hes}(q)&=\frac{q^{\fn(\mu')}}{\prod_{j\geq 1}[\mu_j-\mu_{j+1}]_q!}\left\langle X_{\graph(\hes)}(\x;q),\tilde{\HH}_{\mu'}(\x;1/q)\right\rangle\\
&=\frac{q^{\fn(\mu')+|\hes|-\binom{n}{2}-n}}{\prod_{j\geq 1}[\mu_j-\mu_{j+1}]_q!}\left\langle X_{\graph(\hes)}(\x;1/q),\tilde{\HH}_{\mu'}(\x;1/q)\right\rangle.
\end{split}
\end{equation}
Thus, using the identity $[n]_{1/q}! = q^{-\binom{n}{2}} [n]_q!$, we deduce
\begin{equation}\label{XH_innerproduct}
\left\langle X_{\graph(\hes)}(\x;q),\tilde{\HH}_{\mu}(\x;q)\right\rangle=\frac{q^{-\sum_{j\geq 1}\binom{\mu_j'-\mu'_{j+1}}{2}}\prod_{j\geq 1}[\mu_j'-\mu_{j+1}']_q!}{q^{-\fn(\mu)-|\hes|+\binom{n}{2}+n}}\underline{\F}_{\mu'\hes}\left(\frac{1}{q}\right)
\end{equation}
Now~\cref{Hess_theorem_tableaux} follows from~\cref{Hess_theorem,center_size2,XH_innerproduct}.
\end{proof}

\section{Counting solutions of $X^2=0$ in $\u_\lambda(\FF_q)$: proofs of~\cref{eq: Kirillov-general,anklambda-explicit,Ekhad_thm}
}
\label{matrix2=0:sec}
Let $\Lambda$ be a composition of $n$ and let $\lambda=\mathrm{sort}(\Lambda)$. We define
\begin{equation}
\mathrm{C}_{\Lambda}(q)\defin \#\{X\in \mathfrak{u}_{\Lambda}(\mathbb{F}_q): X^2=0\}.
\end{equation}
The condition $X^2=0$ is equivalent to having only blocks of size at most 2 in the Jordan form. Thus,  for $X\in \C_\mu\cap \u_\Lambda(\FF_q)$ we have $X^2=0$
if and only if
  $\ell(\mu')\leq 2$.  From this and noting that $\F_{(1^n)\Lambda}(q)=1$ we deduce:
\begin{equation}\label{Alambda:identity}
\mathrm{C}_\Lambda(q)=1+\sum_{\substack{\lambda\less \mu'\\ \ell(\mu')=2}}\F_{\mu\Lambda}(q). 
\end{equation}
Since $\F_{\mu\Lambda}(q)=\F_{\mu\lambda}(q)$, by~\cref{Alambda:identity} it follows that
$\mathrm{C}_\Lambda(q)=\mathrm{C}_\lambda(q)$.
Our first application of~\cref{Alambda:identity} is the following estimate of $\mathrm{C}_{\lambda}(q)$. 
\begin{proposition} Let $\lambda=(\lambda_1,\dots,\lambda_s)$ be a partition of $n$ and set $\rho_\alpha=(n-\alpha,\alpha)$ where $\alpha=\min\{\lfloor n/2\rfloor, n-\lambda_1\}$. 
Then $\mathrm{C}_\lambda(q)=\K_{\rho_\alpha\lambda}\, q^{n\alpha-\alpha^2}\left(1+o_q(1)\right)$. 
\end{proposition}
\begin{proof} The conjugates of partitions $\mu$ of $n$ satisfying $\lambda \less \mu'$ and $\ell(\mu') = 2$ must be of the form $\mu' = (n-k, k)$, where $k \leq \min\left\{\lfloor n/2\rfloor, n - \lambda_1\right\}$. Thanks to~\cref{mainthm:Fmula-comb}, we know that $\F_{\mu\lambda}(q)$ is a polynomial of degree 
\begin{equation}\label{binom-n(nu)}
    \binom{n}{2}-\fn(\mu)=\binom{n}{2}-\binom{n-k}{2}-\binom{k}{2}=nk-k^2,
\end{equation}
with leading coefficient $\K_{\mu'\lambda}$.  Note that $nk-k^2$ is increasing on $k\leq n/2$  and so
$
\K_{\rho_\alpha\lambda}\, q^{n\alpha-\alpha^2}
$
dominates the terms in~\cref{Alambda:identity}. This  proves the claim. 
\end{proof}

\begin{remark} For $\lambda=(1^n)$ we have 
\begin{equation}
\rho_\alpha=\begin{cases}
(n/2,n/2) & n\, \text{is even};\\
((n+1)/2,(n-1)/2)  & n\, \text{is odd}.
\end{cases}
\end{equation}
In this case $\K_{\rho_\alpha\lambda}$ is the number of standard Young tableaux of shape $\rho_\alpha$. It is well-known that the number of standard Young tableaux of shape $(k,k)$ (or of shape $(k,k-1)\,$) is equal to the Catalan number $\frac{1}{k+1}\binom{2k}{k}$.
\end{remark}

\begin{notation}
In the rest of this section we work with symmetric polynomials rather than with symmetric functions. Thus, for example we assume $\x=(x_1,\ldots,x_s)$ for suitable $s\geq 1$ and then $\P_\mu(\x;q,t)$ means the Macdonald polynomial in $s$ variables $x_1,\ldots,x_s$. 
\end{notation}

We now turn our attention to $a_{nk}(\lambda)$ defined in~\cref{eq:ankaka}. 
As usual, set $s=\ell(\lambda)$. From~\cref{Intro_them: Main_Macdonald} and~\cref{binom-n(nu)}, we obtain
\begin{equation}\label{ankQ_formula}
a_{nk}(\lambda)=(q-1)^n q^{nk-k^2- \fn(\lambda') - n} \left( \prod_{1\leq i\leq s} [\lambda_i]_q! \right) [\x^\lambda] \Q_{\mu'}(\x; 1/q, 0),
\end{equation}
where $\x=(x_1,\dots,x_s)$. 
Evidently $a_{nk}(\lambda)=0$ when $k>n/2$. 
\subsection{The Cauchy identity  and the proof of~\cref{eq: Kirillov-general}}
Our next goal is to demonstrate that a generalization of Kirillov's mysterious 
relation~\cref{X2Hermit}
follows from  the Cauchy identity for Macdonald polynomials. Let $\x=(x_1,\dots,x_s)$ and $\y=(y_1,\dots,y_r)$ be two sequences of independent indeterminates, and define 
\begin{equation}\label{PI}
    \Pi(\x,\y; q,t)=\prod_{1\leq i\leq s}\prod_{1\leq j\leq r}\frac{(tx_iy_j\, ;q )_\infty}{(x_iy_j\, ; q)_\infty},\qquad (a\, ; q)_\infty :=\prod_{l\geq 0}\left(1-aq^l\right).
\end{equation}
The Cauchy identity for Macdonald polynomials~\cite[p. 324, Eq. (4.13)]{Macdonald}
 states that
 \begin{equation*}
    \sum_{\mu}\P_\mu(\y;q,t)\Q_\mu(\x; q,t)=\Pi(\x,\y; q,t). 
\end{equation*}
One can express $\Pi(\x,\y; q,t)$ in terms of the monomial symmetric polynomials. Let $g_l(\x; q, t)$ denote the coefficient of $y^l$ in the power series expansion of the infinite product
\begin{equation}\label{gm:dif}  
\prod_{1\leq i\leq s}\frac{(tx_iy\, ;q )_\infty}{(x_iy\, ; q)_\infty}=\sum_{l\geq 0}g_l(\x; q,t)y^l.
\end{equation}
This coefficient can be calculated explicitly. From~\cite[p. 314, Example 1]{Macdonald} we have
\begin{equation}\label{g_m}
    g_l(\x; q,t)=\sum_{\nu \,\vdash l}\frac{(t\,;q)_\nu}{(q\,;q)_\nu} m_\nu(\x),\qquad (a\, ; q)_\nu=\prod_{i\geq 1}(a\, ; q)_{\nu_i},
\end{equation}
where $m_{\nu}(\x)$ is the monomial symmetric polynomial. For any partition $\rho=(\rho_1,\rho_2,\dots)$ define
$$
g_{\rho}(\x; q,t):=\prod_{i\geq 1}g_{\rho_i}(\x; q,t).
$$
Moreover we have (see~\cite[p. 311, Eq. (2.10)]{Macdonald} for a proof) that
\begin{equation}
    \Pi(\x,\y; q,t)=\Pi(\y,\x; q,t)=\sum_{\rho}g_\rho(\y;q,t)m_\rho({\bf \x}).
\end{equation}
Since $\P_\rho({\bf y}; q,t)=0$ when $\ell(\rho)>r$, if we set ${\bf y}=(y_1,y_2)$ in the Cauchy identity, we obtain 
\begin{equation}\label{Cauchy}
    \sum_{\ell(\mu')\leq 2}\P_{\mu'}(y_1,y_2;q,0)\Q_{\mu'}(\x;q,0)=\sum_{\rho}g_{\rho}(y_1,y_2; q,0)\, m_{\rho}(\x).
\end{equation}
This in particular implies:
\begin{equation}\label{Cauch_cons}
    \sum_{\substack{\mu'=(n-k,k)
    \\ 0\leq k\leq n/2}}\P_{\mu'}(y_1,y_2;1/q,0)\,[\x^\lambda]\Q_{\mu'}(\x;1/q,0)=g_{\lambda}(y_1,y_2;1/ q,0).
\end{equation}
From~\cref{Cauch_cons} and~\cref{ankQ_formula} we deduce the following result.
\begin{theorem}\label{Cauch_PQg} Let $\lambda = (\lambda_1, \dots, \lambda_s)$ be a partition of $n$, and let $a_{nk}(\lambda)$ denote the number of matrices $X$ in $\u_\lambda(\mathbb{F}_q)$ of rank $k$ such that $X^2 = 0$. Then
\begin{align}\label{Cauch_PQg:eq}
        \sum_{0\leq k\leq n/2} a_{nk}(\lambda)\, q^{k(k-n)}&\P_{(n-k,k)}(y_1,y_2; 1/q,0)
        =[\lambda]_q!\,(q-1)^nq^{-n-\fn(\lambda)}g_{\lambda}(y_1,y_2;1/q,0),
\end{align}
where $[\lambda]_q!:=[\lambda_1]_q!\dots [\lambda_s]_q!$. 
\end{theorem}

We will now proceed with finding  explicit formulas for $\P_{(n-k,k)}(y_1,y_2; 1/q,0)$ and $g_\lambda(y_1,y_2;q,0)$. From~\cite[p. 323, Eq. (4.9)]{Macdonald} 
we have 
$$
\P_{(l)}(y_1,y_2; q,0)=(q\, ; q)_l\, g_l(y_1,y_2;q,0).
$$
Note that $m_{\nu}(y_1,y_2)=0$ when $\ell(\nu)\geq 3$. Thus in~\cref{g_m} we only need to consider partitions $\nu=(l-j,j)$ with $j\leq l/2$. Moreover we have 
$$
m_{(l-j,j)}(y_1,y_2)=
\begin{cases}
y_1^{l-j}y_2^j+y_1^jy_2^{l-j} & j<l-j;\\
y_1^{l/2}y_2^{l/2} & j=l-j.
\end{cases}
$$
As a consequence of these we obtain
\begin{equation*}
    \begin{split}
        \P_{(l)}(y_1,y_2;q,0)
        &=\sum_{l-j\geq j\geq 0}\frac{(q\, ; q)_l}{(q\, ; q)_{l-j}\, (q\, ; q)_j}m_{(l-j,j)}(y_1,y_2)
        =\sum_{j=0}^l {l\brack j}_q y_1^{l-j}y_2^{j}.
    \end{split}
\end{equation*}
But  $
\P_{(n-k,k)}(y_1,y_2; q,0)=(y_1y_2)^k\, \P_{(n-2k)}(y_1,y_2; q,0)
$ by~\cite[p. 325, Eq. (4.17)]{Macdonald}),
 hence for  $0\leq k\leq n/2$ we have
    \begin{equation}\label{Mac_2}
    \P_{(n-k,k)}(y_1,y_2; 1/q,0)=(y_1y_2)^k\sum_{j=0}^{n-2k}{n-2k\brack j}_{1/q}\, y_1^{n-2k-j}y_2^j.
    \end{equation}
\begin{remark}
The formula for $\P_{(n-k,k)}(y_1,y_2; 1/q,0)$
in
\cref{Mac_2} is probably well-known, but we have included a proof for the reader's convenience.
\end{remark}

\begin{remark} By setting $y_1=e^{i\theta}$ and $y_2=e^{-i\theta}$ in~\cref{Mac_2} we obtain
\begin{align}
\label{P_Hermit}
\P_{(n-k,k)}(e^{i\theta},e^{-i\theta};1/q,0)&=\sum_{l=0}^{n-2k}{n-2k \brack l}_{1/q}e^{(n-2k-2l)i\theta}\\
&=H_{n-2k}(\cos(\theta)\, |\, 1/q).
\notag
\end{align}
where $H_m(w\, |\, q)$ is the $q$-Hermite polynomial. 
\end{remark}

We turn our attention to $g_\lambda(y_1,y_2;q,0)$ by expressing it in terms of power sum symmetric polynomials. It is known~\cite[p. 311, Eq. (2.9)]{Macdonald} that  
\begin{equation*}
    g_n(y_1,y_2; q, 0) = \sum_{\rho \,\vdash\, n} z_\rho(q,0)^{-1} \, p_\rho(y_1,y_2),\qquad 
    \text{where }\quad z_\rho(q, t)\defin z_\rho \prod_{i \geq 1} \frac{1 - q^{\rho_i}}{1 - t^{\rho_i}},
\end{equation*}
with $z_\rho$ as in~\cref{z_rho} and $p_\rho(y_1,y_2)$ denoting the power sum symmetric polynomial defined as  
\[
p_\rho(y_1,y_2) \defin \prod_{i \geq 1} p_{\rho_i}(y_1,y_2), \qquad p_k(y_1,y_2) = y_1^k + y_2^k.
\]
It is straightforward to show that 
$$
z_\rho(1/q,0)^{-1}=q^{|\rho|}z_\rho^{-1}\prod_{i\geq 1}\frac{1}{q^{\rho_i}-1},
$$
which implies
\begin{equation}\label{glambda:eq}
    g_\lambda(y_1,y_2; 1/q,0)=q^n\prod_{i=1}^s\sum_{\rho\, \vdash \lambda_i}\frac{1}{z_\rho\prod_{i\geq 1}(q^{\rho_i}-1)}p_\rho(y_1,y_2).
\end{equation}
Recall the definition of 
the Chebyshev polynomials
$T_{\rho}(w)$ from
\cref{eq:Trho}.
With this preparation, we are ready to prove~\cref{eq: Kirillov-general}.
\begin{proof}[Proof of~\cref{eq: Kirillov-general}] We prove this by specializing $y_1=e^{i\theta}$ and $y_2=e^{-i\theta}$ in~\cref{Cauch_PQg:eq}. We start by simplifying the right side of the equation. From the definition of $p_k(y_1,y_2)$ we get $$p_k\left(e^{i\theta},e^{-i\theta}\right)=2T_k(\cos(\theta)).$$
Hence from~\cref{glambda:eq} we infer that
$$
g_\lambda\left(e^{i\theta},e^{-i\theta}; 1/q,0\right)=q^n\prod_{i=1}^s\sum_{\rho\, \vdash \lambda_i}\frac{2^{\ell(\rho)}T_{\rho}(\cos(\theta))}{z_\rho\prod_{i=1}^{\ell(\rho)}\left(q^{\rho_i}-1\right)}.
$$
Thus the right hand side of~\cref{Cauch_PQg:eq} becomes
$$
[\lambda]_q! (q-1)^nq^{-\fn(\lambda')}\prod_{i=1}^s\sum_{\rho\, \vdash \lambda_i}\frac{2^{\ell(\rho)}T_{\rho}(\cos(\theta))}{z_\rho\prod_{i=1}^{\ell(\rho)}\left(q^{\rho_i}-1\right)}.
$$
On the other hand, from~\cref{P_Hermit}, the left hand side of~\cref{Cauch_PQg:eq} becomes
$$
\sum_{k=0}^n a_{nk}(\lambda)q^{k(k-n)}H_{n-2k}(\cos(\theta)\,|\, 1/q).
$$
This finishes the proof. 
\end{proof}
\begin{remark} Specializing $y_1=1$ and $y_2=-1$ in~\cref{Cauch_PQg:eq} gives us another interesting identity. Let $\lambda$ be a partition of $n=2m$ with at least one odd part. Obviously $p_k(y_1,y_2)=0$ when $k$ is odd. Thus, from~\cref{glambda:eq}, we infer that the right hand side of~\cref{Cauch_PQg:eq} is zero. One can also show
$$
\sum_{j=0}^n (-1)^j{n\brack j}_q=
\begin{cases}
    0 & \text{if $n$ odd};\\
    \prod_{i=1}^{n/2}(1-q^{2j-1}) & \text{if $n$ is even}.
\end{cases}
$$
This in particular implies 
$$
\P_{(n-k,k)}(y_1,y_2;1/q,0)=(-1)^k\prod_{j=1}^{m-k}\left(1-q^{1-2j}\right)=(-1)^m\prod_{j=1}^{m-k}\left(q^{1-2j}-1\right)
.$$
Therefore we deduce 
$$
\sum_{k=0}^{m}a_{nk}(\lambda)q^{k(k-n)}\prod_{j=1}^{m-k}\left(q^{1-2j}-1\right)=0
.$$
\end{remark}

\subsection{Proofs of~\cref{anklambda-explicit},~\cref{cor-diffekhad}, and~\cref{Ekhad_thm}} 
Recall that $\Q_\lambda(\x,q,t)$ denotes the dual Macdonald polynomial indexed by $\lambda$. For integers $r\geq s\geq 0$
set \[\Q_{(r,s)}(\x,q):=\Q_{(r,s)}(\x\, ;q,0).\]
Here $(r,s)$ denotes a partition of length at most two. 
We use the main theorem of 
Jing and J\'{o}zefiak~\cite[Eq. (1)]{Jing}, which gives a formula for the $Q_{(r,s)}$  
in terms of the $\Q_{(r+i)}$ and $\Q_{(s-i)}$. Setting $t=0$, $r=n-k$, and $s=k$ in the formula~\cite[Eq. (1)]{Jing}, we obtain 
\begin{align}\label{reu_Mac}
    \Q_{(n-k,k)}&(\x,q)\\
    &= \sum_{j=0}^k (-1)^jq^{\binom{j}{2}}{n-2k+j \brack j}_q\frac{1-q^{n-2k+2j}}{1-q^{n-2k+j}}\Q_{(n-k+j)}(\x,q)\Q_{(k-j)}(\x,q),
    \notag
\end{align}
where we use the conventions 
$\Q_{(0)}(\x\, ;q,t)=1$ and
 $(1-q^{n-2k+2j})/(1-q^{n-2k+j})=1$ when $n-2k=j=0$. Moreover by~\cite[p. 329, Eq. (2.9)]{Macdonald}  we have
$$\Q_{(l)}(\x\, ;q,t)=g_l(\x\, ;q,t),$$
where $g_l(\x\, ;q,t)$ was defined in~\cref{gm:dif}. Henceforth we set 
\[g_{l}(\x,q):=g_{l}(\x\, ;q,0).
\]
Thus from~\cref{g_m} we obtain
\begin{equation*}
\begin{split}
    [\x^\lambda]\Q_{(n-k,n)}&(\x,q)\\
    &=\sum_{j=0}^k (-1)^jq^{\binom{j}{2}}{n-2k+j \brack j}_q\frac{1-q^{n-2k+2j}}{1-q^{n-2k+j}}\, [\x^\lambda]g_{n-k+j}(\x,q)g_{k-j}(\x,q)\\
    &=\sum_{j=0}^k (-1)^jq^{\binom{j}{2}}{n-2k+j \brack j}_q\frac{1-q^{n-2k+2j}}{1-q^{n-2k+j}}\, [\x^\lambda]\sum_{\substack{\nu\, \vdash n-k+j\\ \rho\, \vdash k-j}}\frac{m_{\nu}(\x)m_\rho(\x)}{(q\, ;q)_\nu\,(q\, ;q)_\rho}.
\end{split}
\end{equation*}
The monomial symmetric polynomial is a specialization of the Hall-Littlewood polynomial at $t=1$. Thus from~\cite[Sec. III.3, p. 215]{Macdonald} we have 
$$
[\x^\lambda]\sum_{\substack{\nu\, \vdash n-k+j\\ \rho\, \vdash k-j}}\frac{m_{\nu}(\x)m_\rho(\x)}{(q\, ;q)_\nu\,(q\, ;q)_\rho}=\sum_{\substack{\nu\, \vdash n-k+j\\ \rho\, \vdash k-j}}\frac{f_{\nu\rho}^\lambda(1)}{(q\, ;q)_\nu\,(q\, ;q)_\rho},
$$
where, as before, $f_{\nu\rho}^\lambda(t)$ is a structure constant associated with products of Hall-Littlewood polynomials.  
But \[
(1/q\,;1/q)_n=\prod_{j=1}^n\left(1-\frac{1}{q^j}\right)=\frac{(q-1)^n}{q^{n+\binom{n}{2}}}[n]_q!\quad\text{and}\quad
{n \brack k}_q=q^{k(n-k)}{n\brack k}_{1/q}.
\]
Hence  $[\x^\lambda]\Q_{(n-k,k)}(\x\, ;1/q)$ is equal to
\begin{equation*}
\sum_{j=0}^k(-1)^j\frac{q^{n-\binom{j+1}{2}-(n-2k)j}}{(q-1)^n}{n-2k+j \brack j}_q\frac{q^{n-2k+2j}-1}{q^{n-2k+j}-1}\sum_{\substack{\nu\, \vdash n-k+j\\ \rho\, \vdash k-j}}\frac{q^{\fn(\nu')+\fn(\rho')}f_{\nu\rho}^\lambda(1)}{[\nu]![\rho]!}.
\end{equation*}
This combined with~\cref{ankQ_formula} proves~\cref{anklambda-explicit}. 
\begin{remark}
The formula~\cref{reu_Mac} of    Jing and J\'{o}zefiak  has been generalized by Lassalle and Schlosser~\cite[Theorem 4.1]{Lassalle}. It might be possible to use their result to count the number of matrices $X\in\u_\lambda(\FF_q)$ satisfying $X^d=0$.  
\end{remark}
We now proceed towards the proof of~\cref{Ekhad_thm}.
Since $a_{nk}=0$ for $k>\lfloor n/2\rfloor $, 
 from~\cref{ank_explicitformula} it follows immediately that 
\begin{equation*}
\mathrm{C}_n=\sum_{k=0}^{ \lfloor\frac{n}{2}\rfloor}
\sum_{j=0}^k(-1)^jq^{nk-k^2-\binom{j+1}{2}-(n-2k)j}{n-2k+j \brack j}_q\binom{n}{k-j}\frac{q^{n-2k+2j}-1}{q^{n-2k+j}-1}
.
\end{equation*}
For $k_1=k-j$ and $k_2=j$ we have $0\leq k_1+k_2\leq \lfloor\frac n2\rfloor$
and we can rearrange the right hand summation as
\begin{equation*}
\mathrm{C}_{n}=
\sum_{0\leq k_1+k_2\leq\lfloor\frac n 2 \rfloor}q^{nk_1-k_1^2}\binom{n}{k_1}
(-1)^{k_2}q^{\binom{k_2}{2}}{n-2k_1-k_2 \brack k_2}_q\frac{[n-2k_1]_q}{[n-2k_1-k_2]_q},
\end{equation*}
where on the right hand side we use the convention $(q^0-1)/(q^0-1)=1$ when $n-2k_1=0$. This proves~\cref{cor-diffekhad}.
Next we define 
$$
\B_n:=\begin{cases}
    1 & n=0;\\
\displaystyle    \sum_{j=0}^{\lfloor n/2\rfloor }(-1)^jq^{\binom{j}{2}}{n-j \brack j}_q\frac{[n]_q}{[n-j]_q} & n\geq 1.
\end{cases}
$$
Thus the last formula for $\mathrm{C}_n$ can be rewritten  as
\begin{equation}\label{C_nB_n}
        \mathrm{C}_{n}=\sum_{k_1=0}^{\lfloor n/2\rfloor }q^{nk_1-k_1^2}\binom{n}{k_1}\B_{n-2k_1}. 
\end{equation}
Our next task  is to simplify $\B_n$. We need the following lemma, which was first proved in~\cite[Lemma 2]{Ekhad} by the ``A=B'' method, using a computer. For a $q$-hypergeometric proof see~\cite{Warnaar}. 
\begin{lemma}\label{Ekahd_lemma} Let 
$\A_n:=\sum_{j=0}^{\lfloor n/2\rfloor }(-1)^jq^{\binom{j}{2}}{n-j \brack j}_q
$ for $n\geq 0$. Then 
\[
\A_n=(-1)^n\chi_3(n+1)q^{\frac{1}{3}\binom{n}{2}}
,
\] where $\chi_3$ is the Dirichlet character defined in~\cref{Dirichlet}. 
    \end{lemma}
We let $\A_n=0$ when $n<0$.  Note that $[n]_q=[n-j]_q+q^{n-j}[j]_q$ and so for each $n\geq 0$ we have 
$\B_n=\A_n-q^{n-1}\A_{n-2}$. Thus by~\cref{Ekahd_lemma} we have
\begin{equation}\label{B_n}
\B_n=(-1)^n\left[\chi_3(n+1)q^{\frac{1}{3}\binom{n}{2}}-\chi_3(n-1)q^{\frac{1}{3}\binom{n+1}{2}}\right]\qquad \text{ for }n\geq 1.
\end{equation}
We only prove~\cref{Ekhad_thm} for $\mathrm{C}_{2n+1}$, as the proof for 
$\mathrm{C}_{2n}$ is analogous. 
Set \[
\chi(n,j)\defin        
\left[        \chi_3(2(n-j))q^{\frac{1}{3}\binom{2(n-j)+2}{2}}-\chi_3(2(n-j+1))q^{\frac{1}{3}\binom{2(n-j)+1}{2}}
\right]
\]
and 
\[
\chi'(n,j)\defin
\left[\chi_3(n-j+1)q^{\frac{1}{3}\binom{2(n-j)+1}{2}}-\chi_3(n-j)q^{\frac{1}{3}\binom{2(n-j)+2}{2}}\right].
\]
By~\cref{C_nB_n} and~\cref{B_n},
\begin{equation*}
\begin{split}
        \mathrm{C}_{2n+1}&=\sum_{j=0}^n q^{(2n+1)j-j^2}\binom{2n+1}{j}\B_{2n+1-2j}\\
        &=\sum_{j=0}^n q^{(2n+1)j-j^2}\binom{2n+1}{j}
        \chi(n,j)
        \\
        &=\sum_{j=0}^n q^{(2n+1)j-j^2}\binom{2n+1}{j}
		\chi'(n,j)        
        \\
        &=\sum_{j=0}^n q^{n^2+n-j^2-j}\binom{2n+1}{n-j}\left[\chi_3(j+1)q^{\frac{1}{3}\binom{2j+1}{2}}-\chi_3(j)q^{\frac{1}{3}\binom{2j+2}{2}}\right]
.\end{split}
\end{equation*}
This proves the first formula in~\cref{Ekhad_thm} for $\mathrm{C}_{2n+1}$. To prove the second formula for $\mathrm{C}_{2n+1}$, define $i
\mapsto \1_i$ to be the characteristic function of $\{0,1,\dots,n\}$, i.e., $\1_i=1$ for $0\leq i\leq n$ and $\1_i=0$ otherwise. 
Set
\[
\chi''(i,j)\defin
\left[\chi_3(i+1)q^{\frac{1}{3}\binom{2(3j+i)+1}{2}}-\chi_3(i)q^{\frac{1}{3}\binom{2(3j+i)+2}{2}}\right].
\]

Then from the last formula for $\mathrm{C}_{2n+1}$ we obtain
\begin{align*}
\frac{\mathrm{C}_{2n+1}}{q^{n^2+n}}&=
\sum_{i=0}^2\sum_{j\in \ZZ}q^{-(3j+i)^2-(3j+i)}\binom{2n+1}{n-3j-i}\1_{3j+i}\, \chi''(i,j)
\\
&=I+II,
\end{align*}
where 
\begin{equation*}
\begin{split}
I&=\sum_{j\in \ZZ}q^{-3j^2-2j}\left[\1_{3j}\binom{2n+1}{n-3j}-\1_{3j+1}\binom{2n+1}{n-3j-1}\right],
\end{split}\end{equation*}
and
\begin{equation*}
    \begin{split}
II&=\sum_{j\in \ZZ}q^{-3j^2-4j-1}\left[\1_{3j+2}\binom{2n+1}{n-3j-2}-\1_{3j+1}\binom{2n+1}{n-3j-1}\right].
\end{split}
\end{equation*}
Using $\binom{2n+1}{a}=\binom{2n+1}{2n+1-a}$ and then substituting $j$ by  $-j-1$ in $II$, we obtain
\begin{equation*}
\begin{split}
    II&=\sum_{j\in \ZZ}q^{-3j^2-4j-1}\left[\1_{3j+2}\binom{2n+1}{n+3j+3}-\1_{3j+1}\binom{2n+1}{n+3j+2}\right]\\
    &=\sum_{j\in \ZZ}q^{-3j^2-2j}\left[\1_{-3j-1}\binom{2n+1}{n-3j}-\1_{-3j-2}\binom{2n+1}{n-3j-1}\right].
\end{split}
\end{equation*}
Finally, note that
\begin{equation*}
    \begin{split}
       \binom{2n+1}{n-3j}&= \1_{3j}\binom{2n+1}{n-3j}+\1_{-3j-1}\binom{2n+1}{n-3j},
    \end{split}
\end{equation*}
and
\begin{equation*}
    \begin{split}
        \binom{2n+1}{n-3j-1}&=\1_{3j+1}\binom{2n+1}{n-3j-1}+\1_{-3j-2}\binom{2n+1}{n-3j-1}.
    \end{split}
\end{equation*}
Using the latter relations in $I+II$ proves the second formula for $\mathrm{C}_{2n+1}$ in~\cref{Ekhad_thm}. 

\section{Hook partitions and the proof of~\cref{hook_formula}} 
The goal of this section is to prove~\cref{hook_formula} using~\cref{mainthm:Fmula-comb} combined with~\cref{Ram-Schlosser}. The conjugate of $\mu$ is $\mu'=(n-k,1^k)$. Then the exponent of $q-1$ in~\cref{hook_formula}, using~\cref{mainthm:Fmula-comb}, is $n-(n-k)=k$ and the exponent of $q$ is 
$$
\binom{n}{2}+\ell(\mu)-\fn(\mu)-n=\binom{n}{2}+n-k-\binom{n-k}{2}-n=\binom{n}{2}-\binom{n-k}{2}-k.
$$
Next we prove this simple lemma. 
\begin{lemma}\label{ks:ine} Let $\lambda=(\lambda_1,\dots,\lambda_s)$ be a partition of $n$ and let $\mu'=(n-k,1^k)$ be a hook partition. Then $\lambda\less\mu'$ if and only if $k\leq s-1$.
\end{lemma}
\begin{proof}
    Obviously $\lambda\less \mu'$ implies $\ell(\mu')\leq \ell(\lambda)$ and so $k\leq s-1$. Now assume $k\leq s-1$. Then 
    $$
    \lambda_1+(s-1)\leq \lambda_1+\cdots+\lambda_s=n-k+k\Rightarrow \lambda_1\leq n-k+(k-(s-1))\Rightarrow \lambda_1\leq n-k.
    $$
    Similarly 
    $
    \lambda_1+\lambda_2+(s-2)\leq n-k+1+(k-1),$
    which gives $ \lambda_1+\lambda_2\leq n-k+1+(k-1-(s-2))$ and so $\lambda_1+\lambda_2\leq n-k+1$ since $k\leq s-1$. The inequality for $\lambda_1+\dots+\lambda_i$ is deduced from the same pattern. 
\end{proof}
To apply~\cref{mainthm:Fmula-comb}, we need to calculate $s_q(T)r_q(T)$ explicitly. For this purpose, we must describe all semistandard Young tableaux of shape $\mu'$ and content $\lambda$. We claim that any such semistandard Young tableau arises from a subset of $\{2,\dots,s\}$ of size $k$. The existence of such a subset is guaranteed by~\cref{ks:ine}.

Note that the top-left box of any tableau of shape $\mu'$ and content $\lambda$ must be 1. Given a semistandard tableau of shape $\mu'$ and content $\lambda$, the set of entries in rows 2 through $k + 1$ of $\mu'$ determine a unique size subset of ${2,\ldots,s}$ of size $k$. this proves the claim and show that
$$
\K_{\mu'\lambda}=\binom{s-1}{k}.
$$
Now fix $2 \leq j \leq s - k + 1$ and pick $A$, a subset of $\{2, \dots, s\}$, with $\min A = j$. This subset produces a unique tableau $T$. Note that numbers strictly less than $j$ must only appear in the first row of $T$. Because of this, the number of entries in the first row of $T$ that are strictly less than $j$ is $\lambda_1 + \dots + \lambda_{j-1}$. Thus
$$
r_q(T)=[\lambda_1+\cdots+\lambda_{j-1}]_q\, .
$$
For any $i\not\in A$ we have $\beta^i=(\lambda_i,0,\dots)$, on the other hand $\beta^i=(\lambda_i-1,1,0,\cdots)$ when $i\in A$. Therefore 
$$
s_q(T)=\prod_{i\in A}[\lambda_i]_q. 
$$
By combining these facts with~\cref{mainthm:Fmula-comb}, we deduce~\cref{hook_formula}.
\section{Double cosets and the proof of~\cref{Double_cosets:thm}}
Let $\hes_1\in\mathcal{H}_n$. By the bijective correspondence $X\mapsto I+X$
between $\u_{\hes_1}(\FF_q)$ and 
$\U_{\hes_1}=\U_{\hes_1}(\FF_q)$,  the number of matrices in $\U_{\hes_1}$ of Jordan type $\mu$ is equal to $\F_{\mu\hes_1}(q)$. Let $\theta_{\hes_1}$ denote the character of $\Ind_{\U_{\hes_1}}^\G 1$, where $\G = \GL_n(\FF_q)$. Note that $|\U_{\hes_1}|=q^{n^2-|\hes_1|}$ and so
from~\cref{ind-form} we obtain
\begin{equation}\label{theta_lmabda:Form}
\theta_{\hes_1}(u)=
\begin{cases}
0 & u\, \text{is not a unipotent matrix};\\
\frac{|\cen(I+J_\mu)|}{q^{n^2-|\hes_1|}}\F_{\mu\hes_1}(q) & u\in\conj(I+J_\mu). 
\end{cases}
\end{equation}
Now let $\hes_2\in\mathcal{H}_n$ be another Hessenberg function. We have the well-known equality
\begin{equation}\label{char_double}
    \langle\theta_{\hes_1},\theta_{\hes_2}\rangle=\left|\U_{\hes_1}\backslash\G/\U_{\hes_2}\right|,
\end{equation}
which follows from Frobenius reciprocity and Mackey decomposition. 
Let \[
\mathbf{U}=\bigsqcup_{\mu\,\vdash n}\conj(I+J_\mu)
\] denote the set of all unipotent matrices. Then 
\begin{equation}
\begin{split}
\langle\theta_{\hes_1},\theta_{\hes_2}\rangle&=\frac{1}{|\G|}\sum_{u\in \G}\theta_{\hes_1}(u)\overline{\theta_{\hes_2}(u)}=\frac{1}{|\G|}\sum_{u\in \mathbf{U}}\theta_{\hes_1}(u)\overline{\theta_{\hes_2}(u)}\\
&=\frac{1}{|\G|}\sum_{\mu\vdash n}\sum_{u\in\conj(I+J_\mu)}\frac{|\cen(I+J_\mu)|^2}{q^{2n^2-|\hes_1|-|\hes_2|}}\F_{\mu\hes_1}(q)\,\F_{\mu\hes_2}(q)\\
&=\frac{1}{q^{2n^2-|\hes_1|-|\hes_2|}}\sum_{\mu\vdash n} |\cen(I+J_\mu)|\F_{\mu\hes_1}(q)\, \F_{\mu\hes_2}(q). 
\end{split}
\end{equation}
This calculation combined with~\cref{char_double} proves~\cref{Double_cosets:thm}.

\section{Appendix: Counting matrices with a given rank sequence}

Recall that $\mathrm{C}^r_l$ is defined in~\cref{Cnr}. The following lemma is probably well-known.
\begin{lemma}\label{rank1} Suppose $l\ge 1$ is an arbitrary positive integer and  let $ 0 \le r \le l$. Let $W$ be an $r$-dimensional vector space over $\FF_q$. Then the number of $l$-tuples $(w_1, \dots, w_l) \in W^{\oplus l}$ with the property that $W= \spn(v_1, \dots, v_l)$ is equal to $\mathrm{C}_l^r$. 
\end{lemma}
\begin{proof}
By expressing elements of  $W$ in a chosen basis, one observes that the assertion is equivalent to counting the number of $l \times r$ matrices over $\mathbb{F}_q$ of rank $r$. A classical formula due to Landsberg gives the number of $l \times n$ matrices of rank $r$ as
\[
{n \brack r}_q \, \left( q^l - 1 \right) \left( q^l - q \right) \cdots \left( q^l - q^{r-1} \right)
= {n \brack r}_q q^{\binom{r}{2}} (q-1)^r \frac{[l]_q!}{[l - r]_q!}.
\]
Setting $n = r$ yields the lemma.
\end{proof}

\begin{lemma}\label{rank_profile} Let $n = n_1 \geq n_2 \geq \cdots \geq n_\ell > n_{\ell+1} = 0$ be a sequence of positive integers, and let $m = r_0 \geq r_1 \geq \cdots \geq r_\ell \geq r_{\ell+1} = 0$ be a sequence of non-negative integers, with the property that $r_i - r_{i+1} \leq n_i - n_{i+1}$ for $1 \leq i \leq \ell$. Then, the number of $n \times m$ matrices over $\mathbb{F}_q$ with $\text{rank}(A_i) = r_i$ for $1 \leq i \leq \ell$, where $A_i$ is the submatrix of $A$ formed by the first $n_i$ rows of $A$, is
\begin{equation}\label{rank_profile:number}
    \prod_{j=1}^{\ell}q^{r_{j+1}(n_j-n_{j+1})}{m-r_{j+1}\brack r_{j}-r_{j+1}}_q\mathrm{C}^{r_j-r_{j+1}}_{n_j-n_{j+1}}.
\end{equation}
\end{lemma}

\begin{proof}
The proof is by induction on $\ell$. The case $\ell=1$
 follows from~\cref{rank1} and the fact that 
there are $
{m\brack r_1}_q
$ 
subspaces of dimension $r_1$ in an $\FF_q$-vector space of dimension $m$.
Next we assume the assertion holds for some $\ell$, and we prove it for $\ell+1$. Given vectors $w_1, \dots, w_{n_2}\in\mathbb{F}_q^m$ such that their span, denoted by $W_2$, has dimension $r_2$, we aim to count the number of $(n_1 - n_2)$-tuples $(w_{n_2+1}, \dots, w_{n_1})$ in $\mathbb{F}_q^m$ such that $\mathrm{span}\{w_1, \dots, w_{n_1}\}$ is $r_1$-dimensional. 
To this end, we first select a subspace $W_1$ of $\mathbb{F}_q^m$ of dimension $r_1$ that contains $W_2$. There are
$
{m - r_2 \brack r_1 - r_2}_q
$
such subspaces. Now let $W_1 \supseteq W_2$ be fixed. Next, we count the number of $(n_1 - n_2)$-tuples $(w_{n_2+1}, \dots, w_{n_1})$ of vectors in $W_1$ such that  $\mathrm{span}\{w_1, \dots, w_{n_1}\}$ is $r_1$-dimensional. Since the first $n_2$ vectors already generate $W_2$, these $n_1 - n_2$ vectors must be chosen in such a way that their images in $W_1 / W_2$ under the natural projection span $W_1 / W_2$. To do this, we first choose a spanning set for $W_1 / W_2$ and then lift the vectors to $W_1$. By \cref{rank1}, there are $C_{n_1 - n_2}^{r_1 - r_2}$ spanning sets for $W_1/W_2$. Since each vector in $W_1 / W_2$ has $q^{r_2}$ lifts to $W_1$, the total number of ways to choose $w_{n_2+1},\cdots,w_{n-1}$ is
$
q^{r_2(n_1 - n_2)} \mathrm{C}_{n_1 - n_2}^{r_1 - r_2}
$.
Therefore, the number of $(n_1 - n_2)$-tuples $(w_{n_2+1}, \dots, w_{n_1})$ with the desired property is
$$
q^{r_2(n_1 - n_2)} {m - r_2 \brack r_1 - r_2}_q \mathrm{C}_{n_1 - n_2}^{r_1 - r_2}.
$$
Combining the latter formula with the hypothesis of induction readily implies the claim of the lemma.
\end{proof}

  
\addcontentsline{toc}{section}{References} 
\bibliographystyle{plain}
\bibliography{Ref-Feb21}

\end{document}